\documentclass[10pt]{amsart}

\usepackage{amssymb,mathrsfs,graphicx,extpfeil,amsmath}
\usepackage{ifthen,latexsym,float,colortbl}
\usepackage[margin=0.9in]{geometry}
\usepackage{caption}
\usepackage{sidecap}
\usepackage{rotating,dsfont,stackengine}

\title[Temporal decays and asymptotic behaviors for a kinetic-fluid system]{Temporal decays and asymptotic behaviors  for a Vlasov equation with a flocking term
coupled to incompressible fluid flow}

\author[Young-Pil Choi]{Young-Pil Choi}
\address[Young-Pil Choi]{\newline Department of Mathematics \newline Yonsei University, Seoul 03722, Korea (Republic of)}
\email{ypchoi@yonsei.ac.kr}
\author[Kyungkeun Kang]{Kyungkeun Kang}
\address[Kyungkeun Kang]{\newline Department of Mathematics \newline Yonsei University, Seoul 03722, Korea (Republic of)}
\email{kkang@yonsei.ac.kr}
\author[Hwa Kil Kim]{Hwa Kil Kim}
\address[Hwa Kil Kim]{\newline Department of Mathematics Education \newline Hannam University, Daejeon 34430, Korea (Republic of)}
\email{hwakil@hnu.kr}
\author[Jae-Myoung Kim]{Jae-Myoung Kim}
\address[Jae-Myoung Kim]{\newline Department of Mathematics Education \newline Andong National University, Andong, 36729, Korea (Republic of)}
\email{jmkim02@andong.ac.kr}

\DeclareMathOperator*{\esssup}{ess\,sup}

\begin{document}

\numberwithin{equation}{section}

\newtheorem{theorem}{Theorem}[section]
\newtheorem{lemma}{Lemma}[section]
\newtheorem{corollary}{Corollary}[section]
\newtheorem{proposition}{Proposition}[section]
\newtheorem{remark}{Remark}[section]
\newtheorem{definition}{Definition}[section]

\newcommand{\bke}[1]{\left( #1 \right)}
\newcommand{\bkt}[1]{\left[ #1 \right]}
\newcommand{\bkv}[1]{\left| #1 \right|}
\newcommand{\bket}[1]{\left\{ #1 \right\}}

\renewcommand{\theequation}{\thesection.\arabic{equation}}
\renewcommand{\thetheorem}{\thesection.\arabic{theorem}}
\renewcommand{\thelemma}{\thesection.\arabic{lemma}}
\newcommand{\bbr}{\mathbb R}
\newcommand{\T}{\mathbb T}
\newcommand{\R}{{\mathbb R }}
\newcommand{\Z}{{\mathbb Z }}
\newcommand{\bbe}{\mathbb E}
\newcommand{\abs}[1]{\left| #1 \right|}
\newcommand{\bbz}{\mathbb Z}
\newcommand{\bbn}{\mathbb N}
\newcommand{\bbs}{\mathbb S}
\newcommand{\bbp}{\mathbb P}
\newcommand{\bbt}{\mathbb T}
\newcommand{\norm}[1]{\left\Vert #1 \right\Vert}
\newcommand{\Ass}{\textup{(\textbf{A})\,}}
\newcommand{\ddiv}{\textrm{div}}
\newcommand{\bn}{\bf n}
\newcommand{\rr}[1]{\rho_{{#1}}}
\newcommand{\rabs}[1]{\left. #1 \right|}
\newcommand{\thh}{\theta}
\newcommand{\mw}{\mathcal{W}}
\def\charf {\mbox{{\text 1}\kern-.24em {\text l}}}
\renewcommand{\arraystretch}{1.5}

\newenvironment{mainthm}{{\par\noindent\bf
            Proof of Theorem \ref{main-thm}. }}{\hfill\fbox{}\par\vspace{.2cm}}
\newenvironment{thm2}{{\par\noindent\bf
            Proof of Theorem \ref{thm2}. }}{\hfill\fbox{}\par\vspace{.2cm}}

\newenvironment{thm3}{{\par\noindent\bf
            Proof of Theorem \ref{thm3}. }}{\hfill\fbox{}\par\vspace{.2cm}}

\newenvironment{lem3}{{\par\noindent\bf
            Proof of Lemma \ref{r-300}. }}{\hfill\fbox{}\par\vspace{.2cm}}
\newenvironment{lem4}{{\par\noindent\bf
            Proof of Lemma \ref{r-400}. }}{\hfill\fbox{}\par\vspace{.2cm}}

\newenvironment{lem5}{{\par\noindent\bf
            Proof of Lemma \ref{local-ex-u}. }}{\hfill\fbox{}\par\vspace{.2cm}}

%%%%%%%%Young-Pil%%%%%%%%

\newcommand{\inttr}{\int_{\T^d \times \R^d}}
\newcommand{\intt}{\int_{\T^d}}
\newcommand{\inttrd}{\int_{\T^{2d} \times \R^{2d}}}
\newcommand{\lt}{\left}
\newcommand{\rt}{\right}
\newcommand{\pa}{\partial}
\newcommand{\ml}{\mathcal L}
\newcommand{\mc}{\mathcal C}
\newcommand{\e}{\varepsilon}
\newcommand{\bq}{\begin{equation}}
\newcommand{\eq}{\end{equation}}
\newcommand{\mf}{\mathcal{F}}
\newcommand{\erf}{\mbox{erf}}
\newcommand{\pp}{\mathcal{P}}

%%%%%%%%%%%%%%%%%%%%%%%%

\thanks{}

\subjclass[2010]{}


\begin{abstract}
We are concerned with large-time behaviors of solutions for Vlasov--Navier--Stokes equations in two dimensions and Vlasov-Stokes system in three dimensions including the effect of velocity alignment/misalignment. We first revisit the large-time behavior estimate for our main system and refine assumptions on the dimensions and a communication weight function. In particular, this allows us to take into account the effect of the misalignment interactions between particles. We then use a sharp heat kernel estimate to obtain the exponential time decay of fluid velocity to its average in $L^\infty$-norm. For the kinetic part, by employing a certain type of Sobolev norm weighted by modulations of averaged particle velocity, we prove the exponential time decay of the particle distribution, provided that local particle distribution function is uniformly bounded. Moreover, we show that the support of particle distribution function in velocity shrinks to a point, which is the mean of averaged initial particle and fluid velocities, exponentially fast as time goes to infinity. This also provides that for any $p \in [1,\infty]$, the $p$-Wasserstein distance between the particle distribution function and the tensor product of the local particle distributions and Dirac measure at that point in velocity converges exponentially fast to zero as time goes to infinity. 
\end{abstract}

\date{\today}

\maketitle\centerline{\date}
%
%
%\tableofcontents

%%%%%%%%%%%%%%%%%%%%%%%%%%%%%%%%%%%%%
%
%
%. \section{Introduction}\label{sect:intro}
%
%
%%%%%%%%%%%%%%%%%%%%%%%%%%%%%%%%%%%%%
\section{Introduction}\label{sect:intro}
\setcounter{equation}{0}

In this paper, we consider a situation where a large number of interacting particles are immersed in a viscous fluid. More precisely, let $f = f(x,v,t)$ be the one-particle distribution function in the phase space $(x,v) \in \T^d \times \R^d$ at time $t \in \R_+$, and $u = u(x,t) \in \R^d$ and $p = p(x,t) \in \R$ be the bulk velocity and pressure of the incompressible fluid, respectively. Here $\T^d = (\R/2\pi\Z)^d$ with $d \geq 1$ is the spatial periodic domain. Then our main kinetic-fluid system reads as
\begin{align}
\begin{aligned} \label{main}
&\partial_t f +v \cdot \nabla_xf+\nabla_v \cdot (F(f,u)f)=0, \quad (x,v,t) \in \T^d \times \R^d \times \R_+,\\
&\partial_t u +(1-\delta_{d,3})(u\cdot\nabla_x) u - \mu\Delta_x u +\nabla_x p=-\int_{\R^d}(u-v)f\,dv, \\
&\nabla_x \cdot u=0,
\end{aligned}
\end{align}
subject to initial data
\begin{equation}\label{initial}
\lt(f(x,v,t), u(x,t)\rt)|_{t=0} =: \lt(f_0(x,v), u_0(x)\rt), \quad (x,v) \in \T^d \times \R^d,
\end{equation}
where $\delta_{d,3}$ denotes the Kronecker delta function and $\mu>0$ is the viscosity coefficient. Throughout this paper, we set $\mu=1$ for simplicity. Thus if $d=3$, the fluid system in \eqref{main} becomes Stokes equations, otherwise it becomes the Navier--Stokes equations. The force term $F(f,u)$ in the kinetic equation consists of velocity alignment force $F_a(f)$ given by
\[
F_{a}(f)(x,v,t) = \inttr\psi(x- y) (w- v)f (y,w,t)\,dydw
\]
with a communication weight function $\psi$ and drag force $u-v$, i.e., $F(f,u) = F_a(f) + u-v$. In the present work,
the communication weight function $\psi\in \mc^2(\T^d) $ is even, i.e., $\psi(x) = \psi(-x)$, and we  assume further that there are some numbers $\psi_m, \psi_M \in \R$ such that
\begin{equation}\label{april02-10}
\psi_m\le \psi (x)\le \psi_M,\quad x\in \T^d.
\end{equation}
%Without loss of generality, we also assume that total mass of $f$ is normalized, i.e.
%\[
%\inttr f_0(x,v)\,dxdv = 1.
%\]

The study on mathematical modeling describing the interactions between particle and fluid has received a bulk of attention from variable research fields of such as biotechnology, medicine and sedimentation analysis, see \cite{BBJM05, BDM03, O81} and references therein. On the other hand, dispersed particles immersed in a fluid can be modeled by kinetic-fluid type interactions. If the coupling between particles and fluid is not taken into account, i.e., the drag force $u-v$ in $F(f,u)$, then the kinetic part in \eqref{main} becomes the kinetic Cucker--Smale equation \cite{CFRT10,HL09}. This equation can be rigorously derived from the celebrated Cucker--Smale model \cite{CS07} describing flocking behaviors as the number of particles goes to infinity \cite{CCHS19, CFRT10, HL09}. The system \eqref{main} is thus a coupled system consisting of kinetic Cucker--Smale equation and incompressible viscous fluid coupled through the drag force. This is first proposed in \cite{B-C-H-K} to describe the dynamics of self-propelled particles that are influenced by neighboring fluids. We refer to \cite{B-C-H-K} and references therein for a detailed description of the modeling and relevant literatures.

The existence theories for weak and strong solutions of the coupled kinetic-fluid equations have been well developed: the global-in-time existence of weak solutions for the Vlasov-type or Vlasov--Fokker--Planck equation coupled with homogeneous/inhomogeneous fluid equations are obtained in \cite{B-C-H-K, B-D09, CCK16, C-K, CL16, Ham, HMMM17, M-V, MPP18, WY14, Y}. For the local/global-in-time existence of strong solutions, we refer to \cite{BCHK142,B-D,CDM11, CKL13,DL13,LMW17}. The collision-type interactions between particles are also considered, and global-in-time existence of weak solutions are found in \cite{CY20,YY18} and local-in-time existence of classical solution is obtained in \cite{CLYpre, Math}. We also refer to \cite{CCK16, CJ20,CJpre,G-J1,G-J2, MV08} for the hydrodynamic limits and \cite{C17} for the finite-time breakdown of $\mc^2$-regularity of solutions.

Even though these fruitful studies on the existence theories for the kinetic-fluid systems, there are few available literature on the large-time behavior of solutions in the absent of diffusion term, see \cite{CDM11,DL13,LMW17} for the case with diffusion. In the case of without diffusion, it is not clear to find the nontrivial equilibria. In \cite{BCHK14}, see also \cite{CCK16}, the large-time behavior for the system \eqref{main} showing the alignment between particle and fluid velocities is first investigated based on the Lyapunov functional approach. More precisely, let us define averaged quantities:
\bq\label{aver}
v_c(t) :=  \frac{\inttr vf\,dxdv}{M(t)}  \quad \mbox{and} \quad u_c(t) := \intt u\,dx,
\eq
where $M(t):= \inttr f\,dxdv$. Using these newly defined functions, we introduce a Lyapunov function:
\bq\label{Lyap}
\mathcal{L}(f(t),u(t)):= \frac12\inttr |v - v_c(t)|^2 f\,dxdv + \frac12\intt |u - u_c(t)|^2\,dx + \frac1{2(1 + M(t))}|u_c(t) - v_c(t)|^2.
\eq
Note that the first two terms in \eqref{Lyap} are modulated kinetic energies measuring the fluctuation of velocities from the corresponding averaged quantities. The third term measures the difference between the averages of particle and fluid velocities. By employing the above Lyapunov function or its variant, in \cite{BCHK14, CCK16}, see \cite{C16} for the coupling with compressible fluids, the exponential decay estimate is discussed under suitable a priori assumptions on the regularity of solutions and uniform-in-time integrability of the local particle density. Very recently, under smallness conditions of sufficiently regular initial data, the uniform-in-time integrability assumption on the local particle density is removed in \cite{HMM20} and the large-time behavior of solutions is rigorously established for the system \eqref{main} without the nonlocal velocity alignment force $F_a$. We also refer to \cite{BCHK142,C16,CK15} for the inhomogeneous/compressible fluid cases.

Before proceeding further, we introduce several notations used throughout the current work. For functions, $f(x,v)$ and $u(x)$, $\|f\|_{L^p}$ and $\|u\|_{L^p}$ represent the usual $L^p(\T^d \times \R^d)$- and $L^p(\T^d)$-norms, respectively. We denote by $C$ a generic, not necessarily identical, positive constant, independent of $t$. For any nonnegative integer $k$ and $p \in [1,\infty]$, $W^{k,p} = W^{k,p}(\T^d)$ stands for the $k$-th order $L^p$ Sobolev space, and $W^{k,p}_{\sigma}(\T^d)=\bket{u\in W^{k,p}(\T^d)\,:\, \nabla_x \cdot u=0}$. In case $k=0$, we typically write $W^{0, q}_{\sigma}(\T^d)$ as $L^q_{\sigma}(\T^d)$, and we denote by $H^k=H^k(\T^d) = W^{k,2}(\T^d)$. $\mc^k([0,T];E)$ is the set of $k$-times continuously differentiable functions from an interval $[0,T] \subset \R$ into a Banach space $E$, and $L^p(0,T;E)$ is the set of functions from an interval $(0,T)$ to a Banach space $E$. $\nabla^k$ denotes any partial derivative $\pa^\alpha$ with multi-index $\alpha$, $|\alpha| = k$.

\subsection{Main result}

Our main contribution is to provide the estimate of asymptotic behavior of strong solutions. We first revisit the large-time behavior estimate of solutions for the system \eqref{main}. We refine the assumptions on the dimensions and the communication weight $\psi$ compared to the previous results \cite{BCHK14,BCHK142,CCK16,C16}, where the assumptions on the dimension $d \geq 3$ and $\psi_m \geq 0$ are required. However, our careful analysis shows that the drag force and the small initial mass $M_0 := M(0)$ can even cope with the misalignment interactions between particles, i.e., the communication weight $\psi$ can be negative. Throughout this paper, without loss of generality, we may assume $M_0 \leq 1$. In addition, we  relax the dimension restriction and provide the exponential decay of the Lyapunov function for $d\geq2$ and $\psi_m > -\infty$. To be more specific, our first main result reads as follows.

\begin{theorem}\label{thm_lt}Let $d\geq2$ and
$(f,u)$ be a solution to the system \eqref{main} with sufficient integrability. Suppose the following integrability condition on the local particle density holds:
\[
\|\rho_f\|_{L^\infty(\R_+;L^\infty(\T^d))} < \infty  \quad \mbox{with} \quad \rho_f(x,t) := \int_{\R^d} f(x,v,t)\,dv.
\]
In case $\psi_m < 0$, the initial mass $M_0 = M_0(\psi_m, \|\rho_f\|_{L^\infty}) > 0$ is assumed to be sufficiently small. Otherwise, $M_0$ is arbitrary. 
Then we have
\[
\mathcal{L}(f(t),u(t)) \leq C\mathcal{L}(f_0,u_0) e^{-Ct}  \quad \forall \,t \geq 0,
\]
where $C > 0$ is independent of $t$.
\end{theorem}
\begin{remark} If $\psi_m \geq 0$, then the integrability condition on $\rho_f$ can be relaxed as
\[
\|\rho_f\|_{L^\infty(\R_+;L^\delta(\T^d))} < \infty  \quad \mbox{with} \quad \rho_f(x,t) := \int_{\R^d} f(x,v,t)\,dv,
\]
where $\delta > 0$ is given by
\[
\delta \in \left\{ \begin{array}{ll}
 (1,\infty] & \textrm{if $d=2$}\\[2mm]
 [d/2,\infty] & \textrm{if $d \geq 3$}
\end{array} \right..
\]
\end{remark}

\begin{remark}\label{rmk_limit}Since the total momentum of the system \eqref{main} is conserved in time, see Lemma \ref{lem_energy} (ii) below, 
\[
|u_c(t) - v_c(t)| \leq Ce^{-Ct}
\]
implies
\[
|u_c(t) - v^\infty| + |v_c(t) - v^\infty| \leq Ce^{-Ct},
\]
where $v^\infty:= (M_0v_c(0) + u_c(0))/(1 + M_0)$. Indeed, it follows from Lemma \ref{lem_energy} (ii) that
\[
M(t)v_c(t) + u_c(t) = M_0v_c(0) + u_c(0) \quad \mbox{for} \quad t \geq 0.
\]
This gives $v^\infty = (M_0v_c(t) + u_c(t))/(1 + M_0)$ and subsequently
\[
u_c(t) - v^\infty = u_c(t) - \frac1{1 + M_0}\lt( M_0v_c(t) + u_c(t)\rt) = \frac{M_0}{1 + M_0}\lt( u_c(t) - v_c(t)\rt).
\]
Similarly, we find
\[
v_c(t) - v^\infty = \frac{1}{1 + M_0}\lt( v_c(t) - u_c(t)\rt).
\]
Thus we obtain
\[
|u_c(t) - v^\infty| + |v_c(t) - v^\infty| \leq |u_c(t) - v_c(t)| \leq Ce^{-Ct}
\]
and
\[
v_c(t) \to \frac1{1 + M_0}\lt(M_0v_c(0) + u_c(0)\rt) \quad \mbox{and} \quad u_c(t) \to \frac1{1 + M_0}\lt(M_0v_c(0) + u_c(0)\rt)
\]
exponentially fast as $t \to \infty$. 
\end{remark}

We further improve the asymptotic behavior estimate for the system \eqref{main}. More precisely, we show that the support of the particle distribution function $f$ in velocity asymptotically shrinks to a point, which is the mean of averaged initial particle and fluid velocities, exponentially fast as time tends to infinity. In particular, this implies the exponential convergence of $f$ towards the mono-kinetic distributions in the $p$-Wasserstein metric with $p \in [1,\infty]$, see Section \ref{sec_wasser} for the definition of the $p$-Wasserstein metric and its properties. In order to state our second main result, we recall the notions of weak and strong solutions to the system \eqref{main}-\eqref{initial}.

\begin{definition}\label{strong-sol}
We say that $(f, u)$ is a pair of weak solution to the system  \eqref{main}-\eqref{initial}, provided that
\[
f\in L^{\infty} _{\rm loc} (\R_+; (L^1\cap L^{\infty})( \T^d \times \R^d)) \quad \mbox{and} \quad
u\in  L_{\rm loc}^{\infty}  (\R_+; L^2_{\sigma}( \R^d)) \cap L_{\rm loc}^{2}  (\R_+; W^{1,2}_{\sigma}( \R^d)),
\]
and $(f, u)$ solves the system  \eqref{main}-\eqref{initial} in the sense of distributions (see e.g. \cite[Definition 4.1]{B-C-H-K}).
On the other hand, if
\[
f\in L^{\infty} _{\rm loc} (\R_+; W^{2,q}( \T^d \times \R^d)), \quad
u\in L^{q} _{\rm loc} (\R_+; W^{2,q}_{\sigma}( \R^d)),\quad \mbox{and} \quad u_t \in L^{q} _{\rm loc} (\R_+; L^{q}( \R^d))
\]
for $1\le q<\infty$, and $(f, u)$ solves the system  \eqref{main}-\eqref{initial} pointwise a.e., then $(f, u)$ is called to be a strong solution to the system  \eqref{main}-\eqref{initial}.
\end{definition}

We consider the two dimensional Navier--Stokes system and the three dimensional Stokes system for the fluid part in \eqref{main}. For the system \eqref{main}-\eqref{initial}, the global-in-time existence of strong solutions are studied in \cite{CL16} for two dimensional case and in \cite{BCHK16} for the three dimensional case. In these works, any smallness assumptions on the initial data are not required. On the other hand, the global-in-time existence of weak solutions to the system \eqref{main}-\eqref{initial} is discussed in \cite{B-C-H-K}.

\begin{theorem}\label{main-thm}
Let $d=2, 3$ and $(f_0, u_0)$ satisfy
\[
f_0 \in W^{2,q}( \T^d \times \R^d), \quad u_0\in W^{2,q}(\R^d), \quad 1\le q<\infty,
\]
and the support of $f_0$ in velocity is bounded. Suppose that $(f,u)$ is a strong solution to the system \eqref{main}-\eqref{initial} in the sense of Definition \ref{strong-sol}.
Suppose that the assumptions in Theorem \ref{thm_lt} hold. Then the followings hold:
\begin{itemize}
\item[(i)] For given $k\in \bket{0,1,2}$ and $q\in [1, \infty)$ there exist $p=p(q, k, d, \psi, f_0, u_0)$ and $C$, independent of $t$ and $p$, such that
\[
\lt(\inttr \abs{v-v_c(t)}^p \abs{(\nabla_x^{\alpha} \nabla_v^{\beta} f)(x,v,t)}^q\,dxdv\rt)^{1/p} \le Ce^{-Ct},
\]
where $v_c(t)$ is defined in \eqref{aver}, and $\alpha$ and $\beta$ are multi-indices with $0\le \abs{\alpha}+\abs{\beta}\le k$.

\item[(ii)] The fluid velocity $u$ converges to its average $u_c(t)$  exponentially fast as $t \to \infty$:
\[
\|u(\cdot,t)-u_c(t)\|_{L^{\infty}}\le Ce^{-Ct},
\]
where $C>0$ is independent of $t$ and $u_c(t)$ is defined in \eqref{aver}. Moreover, the vorticity field $\omega := \nabla_x \times u$ vanishes exponentially fast as $t \to \infty$ in $L^p$-norm, where $p$ is dependent on the dimension, namely 
\[
\|\omega(\cdot,t)\|_{L^p(\T^2)}\le Ce^{-Ct} \quad \mbox{for} \quad p<\infty \qquad \mbox{and} \qquad \|\omega(\cdot,t)\|_{L^p(\T^3)}\le Ce^{-Ct} \quad \mbox{for} \quad p < 6,
\]
where $C>0$ is independent of $t$.
\item[(iii)] The support of $f$ in velocity shrinks to a point $v^\infty= (M_0 v_c(0) + u_c(0))/(1+M_0)$ exponentially fast as $t \to \infty$:
\[
\sup_{v \in \Sigma_v(t)} |v - v^\infty| \leq Ce^{-Ct},
\]
where $C>0$ is independent of $t$, and $\Sigma_v$ denotes the $v$-projection of support of $f$.
%\[
%\Sigma_v(t) := \overline{\{v \in \R^d : \exists (x,v) \in \T^d \times \R^d \mbox{ such that } f(x,v,t) \neq 0\}}.
%\]
\end{itemize}
\end{theorem}

To the best of our knowledge, the previous results \cite{C16, CK15, HMM20} are based on the estimate of modulated kinetic energies \eqref{Lyap} from which we can obtain
\[
\mw_1(f(\cdot,\cdot,t), \rho_f(\cdot,t) \otimes \delta_{u(\cdot,t)}(\cdot)) \to 0
\]
as $t \to \infty$ exponentially fast, i.e., the particle density $f$ converges to the mono-kinetic distributions in the first order Wasserstein metric. In the present work, as presented in Theorem \ref{main-thm}, we employ a new weighted norm by the $p$-th modulated velocity moments and establish the asymptotic behavior of solutions. Moreover, our careful estimates do not depend on the exponent $p$, and this enables us to have that the support of the particle distribution function $f$ in velocity shrinks to the fluid velocity $u$ as time goes to infinity, in fact, $u$ converges to the mean of averaged initial particle and fluid velocities. In particular, this yields
\[
\mw_p(f(\cdot,\cdot,t), \rho_f(\cdot,t) \otimes \delta_{u(\cdot,t)}(\cdot)) \to 0
\]
as $t \to \infty$ exponentially fast for any $p \in [1,\infty]$.  Indeed, if we consider a map $\pi_{v^\infty}: (x,v) \mapsto (x,v^\infty)$, then for any $\varphi \in \mc_b(\T^d \times \R^d)$,
$$\begin{aligned}
\inttr \varphi(x,v) (\pi_{v^\infty} \# f)(dxdv) &= \inttr \varphi (x,v^\infty) f(x,v)\,dxdv \cr
&= \intt \varphi (x,v^\infty) \rho_f(x)\,dx \cr
&=  \inttr \varphi (x,v) \rho_f(x)\,dx \otimes \delta_{v^\infty}(dv),
\end{aligned}$$
where $\cdot \,\# \,\cdot$ denotes the push-forward of a measure, which is defined in Definition \ref{def_push} below.
This together with Theorem \ref{main-thm} (iii) yields that for any $p \in [1,\infty)$
$$\begin{aligned}
&\mw_p (f(\cdot,\cdot,t), \rho_f(\cdot,t) \otimes \delta_{u(\cdot,t)}(\cdot)) \cr
&\quad \leq \mw_p (f(\cdot,\cdot,t), \rho_f(\cdot,t) \otimes \delta_{v^\infty}(\cdot)) + \mw_p (\rho_f(\cdot,t) \otimes \delta_{u(\cdot,t)}(\cdot), \rho_f(\cdot,t) \otimes \delta_{v^\infty}(\cdot)) \cr
&\quad \leq \lt(\inttr |(x,v) - (x,v^\infty)|^p f(x,v,t)\,dxdv \rt)^{1/p} + \|u(\cdot,t) - v^\infty\|_{L^\infty} \leq Ce^{-Ct},
\end{aligned}$$
where $C>0$ is independent of $t$. Since the right hand side of the above inequality does not depend on $p$, we can pass to the limit $p \to \infty$, and this concludes
\[
\mw_\infty(f(\cdot,\cdot,t), \rho_f(\cdot,t) \otimes \delta_{u(\cdot,t)}(\cdot))  \leq Ce^{-Ct},
\]
where $C>0$ is independent of $t$.

\begin{remark} Our strategy requires the a priori assumption on the uniform-in-time bound of the local particle density, however, as mentioned above, this assumption can be removed in \cite{HMM20} when there is no velocity-alignment force $F_a$ in \eqref{main} and the initial Lyapunov function $\ml(f_0, u_0)$ is small enough. Thus we can make our computations fully rigorous by ignoring the velocity alignment force $F_a$ with restriction of smallness of initial data.
\end{remark}

\subsection{Strategies of the proof and organization of the paper} We first recall some known a priori energy estimates and the notion of Wasserstein distance of order $p$ and its properties in Section \ref{sec:pre}.

In Section \ref{sec:thm_lt}, we revisit the large-time behavior estimate of solutions to the system \eqref{main} which proves our first main result Theorem \ref{thm_lt}. As mentioned before, we refine some assumptions that used in previous works under the uniform-in-time bound assumption on the local particle density. We consider the drag force as a relative damping of particle velocity and extract the dissipative effect for the modulated kinetic energy of the particle distribution. On the other hand, for the convergence of fluid velocity, we properly use the diffusion term to control the energy growth. This allows us to take care of the velocity misalignment interactions. Note that the Lyapunov functional $\ml(f,u)$ is constructed based on the modulated kinetic energies, in particular, this result asserts
\[
\int_{\T^d \times \R^d} |v - v_c(t)|^2 f(x,v,t)\,dxdv + \int_{\T^d} | u(x,t) - u_c(t)|^2\,dx \leq Ce^{-Ct} \quad \forall\, t \geq 0.
\]

Section \ref{sec:zero} is devoted to the proof of Theorem \ref{main-thm} for $k=0$. We notice that the exponential decay estimate of the Lyapunov functional $\ml(f,u)$ appeared in \eqref{Lyap} implies the exponential decay of the drag forces in the fluid equations in \eqref{main}. More precisely, in Lemma \ref{lem_drag2}, we show
\[
\lt\|\int_{\R^d} f(\cdot,v,t)(u(\cdot,t)-v)\,dv\rt\|_{L^2} \leq Ce^{-Ct} \quad \forall \, t \geq 0.
\]
We then improve the decay estimate of fluid velocity. In order to have the decay estimate in better function spaces, we rewrite the fluid equations in \eqref{main} as the equations of vorticity $\omega = \nabla_x \times u$:
\[
\partial_t \omega - \Delta_x \omega + (1 - \delta_{d,3})u\cdot \nabla_x \omega=\nabla_x\times \lt(\int_{\R^d} f(u-v)\,dv\rt).
\]
By using the exponential decay of the drag force and heat kernel estimate in $\T^d$, in Lemma \ref{lemma-fluid-decay} we improve the exponential decay estimate of the fluid velocity:
\[
\| u(\cdot,t) - u_c(t)\|_{L^\infty} \leq Ce^{-Ct} \quad \forall \, t \geq 0.
\]
Moreover, as stated in Theorem \ref{main-thm} we have the exponential decay estimates of the vorticity $\omega$ in $L^p(\T^d)$-norm, where $p$ depends on the dimensions. We finally combine these improved decay estimate of fluid velocity and the growth estimate of $L^q$-norm of the particle distribution function $f$ to have the exponential decay estimate: 
\[
\lt(\inttr |v - v_c(t)|^p f^q(x,v,t)\,dxdv\rt)^{1/p} \leq Ce^{-Ct} \quad \forall \,t \geq 1
\]
in Proposition \ref{thm_main}. Here, the proper dissipation rate is obtained from the drag forces. Furthermore, our careful analysis also provides that the constant $C>0$ appeared in the above is independent of both $p$ and $t$. This enables us to consider the limit $p\to+\infty$ and have the estimate of support of $f$ in velocity, see Corollary \ref{main_cor}. 

In Section \ref{sec:high}, we further extend these estimates to the weighted $W^{2,q}$-norm. Here our starting point is again to estimate the drag force in the fluid equations in \eqref{main}. The core idea is the same with the previous estimate. We use the drag and viscous forces to have a proper dissipative effects. Since the decay estimates obtained in Section 4.1 imply that $u - u_c$ and $\int_{\R^d} (v - v_c)f\,dv$ converge to zero exponentially fast in $L^p(\T^d)$ with any $p \in [2,\infty]$, we also have
\[
\lt\|\int_{\R^d} f(\cdot,v,t)(u(\cdot,t)-v)\,dv\rt\|_{L^p} \leq Ce^{-Ct} \quad \forall \,t \geq 1
\]
for $2\le p\leq \infty$, see Lemma \ref{lem_drag} for details. It seems hard to obtain some decay estimates of fluid velocity $u$ with higher order derivatives by using this decay estimate. However, at least, this provides some integrability of $u$, and in particular we have the following growth estimate in Lemma \ref{omega-u-estimate}:
\[
\|\nabla_x \omega(\cdot, t)\|_{L^{\infty}} +\|\nabla_x^2 u(\cdot, t)\|_{L^{p}} \leq Ce^{C(1+t)} \quad \forall \,t \geq 1
\]
for any $p < \infty$. Let us point out that in this step it is important to have the linear growth rate of $t$ in the exponent. Using those observations combined with the growth estimate of $W^{k,q}(\T^d \times \R^d)$-norm of $f$, $k=1,2$, we estimate 
\[
\lt(\inttr \abs{v-v_c(t)}^p \abs{(\nabla_x^{\alpha} \nabla_v^{\beta} f)(x,v,t)}^q\,dxdv\rt)^{1/p} \le Ce^{-Ct}  \quad \forall \,t \geq 1
\]
for $1\le \abs{\alpha}+\abs{\beta}\le k$. This completes the proof of Theorem \ref{main-thm}.

%%%%%%%%%%%%%%%%%%%%%%%%%%%%%%%%%%%%%
%
%
%. \section{Preliminaries}
%
%
%%%%%%%%%%%%%%%%%%%%%%%%%%%%%%%%%%%%%
\section{Preliminaries}\label{sec:pre}
%In this section, we briefly present the well-known a priori estimates for the system \eqref{main}.

%%%%%%%%%%%%%%%%%%%%%%%%%%%%%%%%%%%%%
%
%
% \subsection{A priori energy estimates}
%
%
%%%%%%%%%%%%%%%%%%%%%%%%%%%%%%%%%%%%%

\subsection{A priori energy estimates}

We first provide estimates of conservation laws and total energy dissipation in the lemma below. Since the proof can be found in \cite{B-C-H-K, BCHK14}, we omit it here.
\begin{lemma}\label{lem_energy} Let $(f,u)$ be a solution to the system \eqref{main} with sufficient integrability. Then we have the following estimates:
\begin{itemize}
\item[(i)] The total mass of $f$ is conserved in time:
\[
\inttr f(x,v,t)\,dxdv = \inttr f_0(x,v)\,dxdv = M_0 \quad \forall \,t \geq 0.
\]
\item[(ii)] The total momentum is conserved in time:
\[
\inttr vf(x,v,t)\,dxdv + \intt u(x,t)\,dx = \inttr vf_0(x,v)\,dxdv + \intt u_0(x)\,dx \quad \forall \, t \geq 0.
\]
\item[(iii)] The total energy is not increasing in time:
\begin{align*}
&\frac12\frac{d}{dt} \lt(\inttr |v|^2f\,dxdv + \intt |u|^2\,dx \rt) +  \intt |\nabla_x u|^2\,dx + \inttr |u-v|^2 f\,dxdv\cr
&\qquad \leq - \frac12\inttrd \psi(x-y)|v-w|^2 f(x,v,t)f(y,w,t)\,dxdydvdw.
\end{align*}
\end{itemize}
\end{lemma}

We notice that the weak solution defined in the sense of Definition \ref{strong-sol} satisfies the energy estimates in Lemma \ref{lem_energy}. More precisely, we refer to  \cite[Lemma 3.1]{HMM20} for the first two conservations and \cite[Theorem 4.2]{B-C-H-K} for the energy dissipation estimates (iii).

\begin{remark} By Lemma \ref{lem_energy} (iii), if the communication weight function $\psi$ is nonnegative, i.e., $\psi_m \geq 0$, then we have the following integrability:
\[
\int_0^\infty  \inttr |u-v|^2 f\,dxdvdt < \infty.
\]
On the other hand, that integrand can be rewritten as
\[
\inttr |u-v|^2 f\,dxdv = \inttr \rho_f | u - u_f|^2\,dxdv + \inttr |u_f - v|^2 f\,dxdv,
\]
where $u_f = u_f(x,t)$ denotes the local particle density defined by 
\[
u_f(x,t) := \frac{\int_{\R^d} vf(x,v,t)\,dv}{\int_{\R^d} f(x,v,t)\,dv},
\]
and this yields
\[
\lim_{t \to \infty} \int_t^{t+1}\inttr \rho_f | u - u_f|^2\,dxdvds = 0 \quad \mbox{and} \quad \lim_{t \to \infty} \int_t^{t+1}\inttr |u_f - v|^2 f\,dxdvds = 0.
\]
\end{remark}

%%%%%%%%%%%%%%%%%%%%%%%%%%%%%%%%%%%%%
%
%
% \subsection{$p$-Wasserstein metric}
%
%
%%%%%%%%%%%%%%%%%%%%%%%%%%%%%%%%%%%%%
\subsection{$p$-Wasserstein distance}\label{sec_wasser}

In this subsection, we present several definition and properties of Wasserstein
distance.

\begin{definition}
Let $\mu$ and $\nu$ be two Borel probability measures on
$\T^d$. Then the Euclidean Wasserstein distance of order $1\leq
p<\infty$ between $\mu$ and $\nu$ is defined as
\[
\mw_p(\mu, \nu) := \inf_{\gamma \in \Gamma(\mu, \nu)} \lt(\int_{\T^d \times \T^d} |x-y|^p \gamma(dxdy)  \rt)^{1/p}
\]
for $p < \infty$ and
\[
\mw_\infty(\mu,\nu) := \inf_{\gamma \in \Gamma(\mu, \nu)} \esssup_{(x,y) \in supp(\gamma)} |x-y|,
\]
where $\Gamma(\mu,\nu)$ is the set of all probability measures on $\T^d \times \T^d$ with first and second marginals $\mu$ and $\nu$, respectively, i.e.,
\[
\int_{\T^d \times \T^d} (\varphi(x) + \psi(y)) \,\gamma(dxdy) = \intt \varphi(x)\,\mu(dx) + \intt \psi(y)\,\nu(dy)
\]
for each $\varphi, \psi \in \mc(\T^d)$.
\end{definition}

Let us denote by $\mathcal{P}_p(\T^d)$ the set of probability measures in $\T^d$ with $p$-th moment bounded. Then $\mathcal{P}_p(\T^d), 1\leq p < \infty$ is a
complete metric space endowed with the $p$-Wassertein distance
$\mw_p$, see \cite{A-G-S,Vil}. In particular for $p=1$, Wasserstein-1 distance $\mw_1$, which is also often called {\it Monge-Kantorovich-Rubinstein distance}, is equivalent to the bounded Lipschitz distance:
\[
d_{BL}(\mu,\nu) = \sup\left\{ \int_{\T^d} \varphi(x)(\mu(dx) - \nu(dx)) : \varphi \in \mbox{Lip}(\T^d), \mbox{ Lip}(\varphi) \leq 1\right\},
\]
where Lip($\T^d$) and Lip($\varphi$) denote the set of Lipschitz functions on $\T^d$ and the Lipschitz constant of a function $\varphi$, respectively.

We next present the definition of the push-forward of a measure by
a mapping which gives some relation between Wasserstein
distances and optimal transportation.

\begin{definition}\label{def_push}
Let $\mu$ be a Borel measure on $\T^d$ and $\mathcal{T} :
\T^d \to \T^d$ be a measurable mapping. Then the push-forward
of $\mu$ by $\mathcal{T}$ is the measure $\nu$ defined by
\[
\nu(\mathcal{B}) = \mu(\mathcal{T}^{-1}(\mathcal{B})) \quad \mbox{for} \quad \mathcal{B}
\subset \T^d,
\]
and denoted as $\nu = \mathcal{T} \# \mu$.
\end{definition}

We then provide some classical properties whose proofs may be found in \cite{Vil}.

\begin{proposition}\label{prop-proper}
\begin{itemize}
\item[(i)] The definition of $\nu = \mathcal{T} \# \mu$ is equivalent
to
$$
\int_{\T^d} \phi(x)\, \nu(dx) = \int_{\T^d} \phi(\mathcal
T(x))\, \mu(dx)
$$
for all $\phi\in \mathcal{C}_b(\T^d)$.

\item[(ii)] Given $\mu_0 \in \mathcal{P}_p(\T^d)$, consider two measurable mappings
$X_1,X_2 : \T^d \to \T^d$, then the following holds:
\[
\mw_p(X_1 \# \mu_0, X_2 \# \mu_0) \leq \lt(\int_{\T^d \times
\T^d} |x-y|^p \,\gamma(dxdy)\rt)^{1/p} = \lt(\int_{\T^d} | X_1(x) - X_2(x)|^p
\,\mu_0(dx)\rt)^{1/p}.
\]
Here we took the transference plan $\gamma = (X_1 \times
X_2)\#\mu_0$.
\end{itemize}
\end{proposition}

%%%%%%%%%%%%%%%%%%%%%%%%%%%%%%%%%%%%%
%
%
% \subsection{A revisit to the estimate of large behavior for the system \eqref{main}}
%
%
%%%%%%%%%%%%%%%%%%%%%%%%%%%%%%%%%%%%%

\section{Proof of Theorem \ref{thm_lt}: A revisit to the estimate of asymptotic behavior}\label{sec:thm_lt}

In this section, we provide the details of the proof of Theorem \ref{thm_lt}.

Direct computations yield
\begin{align}\label{lt_est}
\begin{aligned}
&\frac12\frac{d}{dt}\lt(\inttr |v - v_c(t)|^2 f\,dxdv + \intt |u - u_c(t)|^2\,dx + \frac1{1 + M_0}|u_c(t) - v_c(t)|^2\rt)\cr
&\quad = - \inttr |u-v|^2 f\,dxdv  - \intt |\nabla_x u|^2\,dx\cr
&\qquad - \frac12\inttrd \psi(x-y)|v-w|^2 f(x,v,t)f(y,w,t)\,dxdydvdw\cr
&\quad =: - \mathcal{D}(f(t),u(t)).
\end{aligned}
\end{align}
On the other hand, due to Young's inequality, it follows that
\begin{align*}
&\inttr |u-v|^2 f\,dxdv\cr
&\quad = \inttr |u-u_c + u_c - v_c + v_c - v|^2 f\,dxdv\cr
&\quad\geq \lt(1 - \frac2\e\rt)\intt |u - u_c|^2 \rho_f\,dx + (1 - \e)M_0 |u_c - v_c|^2 + (1 - \e)\inttr |v - v_c|^2 f\,dxdv,
\end{align*}
where $\e > 0$ will be determined later. Together with the following estimate
\[
\frac12\inttrd \psi(x-y)|v-w|^2 f(x,v,t)f(y,w,t)\,dxdydvdw \geq \psi_m M_0\inttr |v - v_c|^2 f\,dxdv,
\]
we obtain
\begin{align}\label{est_diss}
\begin{aligned}
&\inttr |u-v|^2 f\,dxdv +  \frac12\inttrd \psi(x-y)|v-w|^2 f(x,v,t)f(y,w,t)\,dxdydvdw\cr
&\quad \geq  (1 - \e)M_0|u_c - v_c|^2 + (1 + \psi_mM_0 - \e)\inttr |v - v_c|^2 f\,dxdv  + \lt(1 - \frac2\e\rt)\intt |u - u_c|^2 \rho_f\,dx.
\end{aligned}
\end{align}
We then choose $\e>0$ and $M_0>0$ small enough so that $1 + \psi_mM_0 - \e >0$. Note that, via Sobolev embedding, we have
\begin{equation}\label{april02-20}
\intt |u - u_c|^2 \rho_f\,dx \leq  C\|\rho_f\|_{L^\delta}\|\nabla_x u\|_{L^2}^2.
\end{equation}
Indeed, for $\delta\in (1, \infty]$ by H\"older's inequality we find
\[
\int_{\T^2} |u - u_c|^2 \rho_f\,dx \leq \|\rho_f\|_{L^\delta}\lt(\int_{\T^2} |u - u_c|^{2\delta'} \,dx\rt)^{1/\delta'} = \|\rho_f\|_{L^\delta}\|u - u_c\|_{L^{2\delta'}}^2,
\]
where $\delta'$ is the H\"older conjugate of $\delta$. We now choose $p = 2\delta'/(1 + \delta')$, then we can easily check $p < 2$ and applying Gagliardo--Nirenberg--Sobolev inequality gives
\[
\|u - u_c\|_{L^{2\delta'}}^2 \leq \|u - u_c\|_{W^{1,p}} \leq C\|u - u_c\|_{W^{1,2}} \leq C\|\nabla_x u\|_{L^2} 
\]
due to the boundedness of our spatial domain and Poincar\'e inequality. This asserts
\[
\int_{\T^2} |u - u_c|^2 \rho_f\,dx \leq C\|\rho_f\|_{L^\delta}\|\nabla_x u\|_{L^2}^2
\]
for $\delta\in (1, \infty]$. On the other hand, if $d>2$, we have
\[
\intt |u - u_c|^2 \rho_f\,dx \leq \|\rho_f\|_{L^{d/2}} \|u- u_c\|_{L^{2d/(d-2)}}^2 \leq C\|\rho_f\|_{L^{d/2}}\|\nabla_x u\|_{L^2}^2 \leq C\|\rho_f\|_{L^\delta}\|\nabla_x u\|_{L^2}^2
\]
for $\delta \in [d/2,\infty]$. Since $(1-2/\e) < 0$, combining \eqref{est_diss} and \eqref{april02-20} implies
\begin{align}\label{est_diss2}
\begin{aligned}
&\inttr |u-v|^2 f\,dxdv +  \frac12\inttrd \psi(x-y)|v-w|^2 f(x,v,t)f(y,w,t)\,dxdydvdw\cr
&\quad \geq  (1 - \e)M_0|u_c - v_c|^2 + (1 + \psi_m M_0 - \e)\inttr |v - v_c|^2 f\,dxdv  + C\|\rho_f\|_{L^\delta}\lt(1 - \frac2\e\rt)\intt |\nabla_x u|^2\,dx.
\end{aligned}
\end{align}
If $\psi_m \geq 0$, due to \eqref{april02-20}, we get
\begin{align*}%\label{april02-30}
\begin{aligned}
\mathcal{L}(f(t),u(t)) &\leq c_0 (1 - \e)M_0 |u_c - v_c|^2 +c_0 (1 + \psi_mM_0 - \e)\inttr |v - v_c|^2 f\,dxdv + \intt |u - u_c|^2 \,dx \cr
&\leq C\intt |\nabla_x u|^2\,dx+c_0 \lt(\lt(\frac2\e - 1\rt)\intt |u - u_c|^2 \rho_f\,dx + \inttr |u-v|^2 f\,dxdv\rt) \cr
&\quad +  \frac{c_0}2\inttrd \psi(x-y)|v-w|^2 f(x,v,t)f(y,w,t)\,dxdydvdw \cr
&\le C\lt(\inttr |u-v|^2 f\,dxdv + \intt |\nabla_x u|^2\,dx\rt) \cr
&\quad + \frac C2\inttrd \psi(x-y)|v-w|^2 f(x,v,t)f(y,w,t)\,dxdydvdw,
\end{aligned}
\end{align*}
where $c_0 > 0$ is given by
\[
c_0 = \max\lt\{\frac{1}{4M_0(1 - \e)}, \frac{1}{2(1 + \psi_m M_0 - \e)} \rt\}.
\]
This implies that there exists $C>0$ independent of $t$ such that 
\[
\mathcal{D}(f(t),u(t)) \geq C\mathcal{L}(f(t),u(t)),
\]
Combining this and \eqref{lt_est} asserts
\[
\frac{d}{dt} \mathcal{L}(f(t),u(t)) + C^{-1} \mathcal{L}(f(t),u(t)) \leq 0,
\]
and thus, we deduce the exponential decay of the functional $\mathcal{L}$.

On the other hand, if $\psi_m < 0$, then we find from \eqref{est_diss2} with $\delta=d$ that
$$\begin{aligned}
\mathcal{D}(f(t),u(t))& \geq  (1 - \e)M_0 |u_c - v_c|^2 + (1 + \psi_m M_0- \e)\inttr |v - v_c|^2 f\,dxdv\cr
&\quad  + \lt(1 - C\|\rho_f\|_{L^d}\lt(\frac2\e-1\rt)\rt)\intt |\nabla_x u|^2\,dx.
\end{aligned}$$
Note that the interpolation inequality gives
\[
\|\rho_f(\cdot,t)\|_{L^d} \leq \|\rho_f(\cdot,t)\|_{L^\infty}^{(d-1)/d} \|\rho_f(\cdot,t)\|_{L^1}^{1/d} \leq \|\rho_f\|_{L^\infty(\T^d \times \R_+)}^{(d-1)/d} M_0^{1/d},
\]
and thus by choosing the initial mass $M_0 > 0$ small enough to get
\[
C\|\rho_f\|_{L^d}\lt(\frac2\e-1\rt) < C\|\rho_f\|_{L^\infty(\T^d \times \R_+)}^{(d-1)/d} M_0^{1/d}\lt(\frac2\e-1\rt) < 1.
\]
This together with applying the Poincar\'e inequality yields
$$\begin{aligned}
\mathcal{D}(f(t),u(t))&\geq  (1 - \e)M_0|u_c - v_c|^2 + (1 + \psi_m M_0- \e)\inttr |v - v_c|^2 f\,dxdv + c_1\intt |u - u_c|^2\,dx
\end{aligned}$$
for some $c_1 > 0$. This implies that there exists $C>0$ independent of $t$ such that 
\[
\mathcal{D}(f(t),u(t)) \geq C\mathcal{L}(f(t),u(t)).
\]
Hence we have from \eqref{lt_est} the desired exponential decay of $\mathcal{L}(f(t),u(t))$.

%{\color{red}
%\begin{remark}\label{rmk_new} When $\psi_m > -1$, we can easily find from \eqref{as_mu} that the condition on the viscosity coefficient $\mu > 0$ can be replaced by the smallness assumption on $\|\rho_f\|_{L^\infty(\R_+;L^\delta(\T^d))}$.
%\end{remark}
%}

%%%%%%%%%%%%%%%%%%%%%%%%%%%%%%%%%%%%%
%
%
% \section{Theorem}
%
%
%%%%%%%%%%%%%%%%%%%%%%%%%%%%%%%%%%%%%

\section{Proof of Theorem \ref{main-thm}:  Exponential decays estimates}

%{\color{red} In this section, we provide the details of the proof for Theorem \ref{main-thm}. In the rest of this paper we assume that the viscosity coefficient $\mu>0$ is chosen such that \eqref{new_as} holds. The constant $\mu$ may be greater than $1$, however, we set $\mu=1$ for simplicity of the presentation.} 

\subsection{Exponential decays of weighted $L^q$ estimates}\label{sec:zero}
We first begin with the estimate of the growth rate of $L^p$-norm of $f$.

\begin{lemma}[$L^{q}$-estimate of $f$]\label{f-growth} Let $1\le q \le \infty$ and $(f,u)$ be the strong solution  defined in Definition \ref{strong-sol} to the system \eqref{main}-\eqref{initial}.  Under the same assumptions as in Theorem  \ref{main-thm}, we have
\begin{equation}\label{april03-10}
\bke{\inttr f^{q}(x,v,t)\,dxdv}^{\frac{1}{q}}\leq \lt(\inttr f_0^{q}\,dxdv\rt) ^{\frac{1}{q}}e^{d\bke{\frac{q-1}{q}}(\psi_M+1)t} \quad \forall \,t \geq 0
\end{equation}
for $1\le q<\infty$ and 
\begin{equation}\label{april03-20}
\norm{f}_{L^{\infty}(\T^d \times \R^d)}\leq\norm{f_0}_{L^{\infty}(\T^d \times \R^d)}e^{d(\psi_M+1)t} \quad \forall \,t \geq 0.
\end{equation}
\end{lemma}
\begin{proof}
It follows from the kinetic part in \eqref{main} that
\begin{align*}
&\frac{1}{q}\frac{d}{dt}\inttr f^{q}\,dxdv =-\inttr \nabla_v\cdot (F(f,u)f)
f^{q-1}\,dxdv\cr
=& \frac{q-1}{q}\inttr F(f,u) \cdot \nabla_v(f^{q})\,dxdv
= \frac{d(q-1)}{q}\inttr (1 + \psi \star \rho_f ) f^{q}\,dxdv,
\end{align*}
\begin{equation*}
\end{equation*}
where we used
\[
\nabla_v \cdot F(f,u) = -d - d\inttr \psi(x-y)f(y,w)\,dydw = -d(1 + \psi \star \rho_f).
\]
Thus, due to \eqref{april02-10}, we obtain
\begin{equation}\label{f}
\frac{d}{dt}\inttr f^{q}\,dxdv \leq d(q-1)(1 + \psi_M) \inttr f^{q}\,dxdv.
\end{equation}
Applying Gr\"onwall's lemma, we deduce \eqref{april03-10}, and \eqref{april03-20} simply follows by passing $q$ to the limit in \eqref{april03-10}.
\end{proof}

We next present the exponential decay estimate of $L^2$-norm of the drag force in the fluid equation in \eqref{main}.

\begin{lemma}\label{lem_drag2}
Suppose that $(f,u)$ is the strong solution  defined in Definition \ref{strong-sol} to the system \eqref{main}-\eqref{initial}.
Under the same assumptions as in Theorem  \ref{main-thm}, there exists a constant $C>0$, independent of $t$, such that
\[
\lt\|\int_{\R^d} f(u-v)\,dv\rt\|_{L^2} \leq Ce^{-Ct} \quad \forall \, t \geq 0.
\]
\end{lemma}

\begin{proof} A straightforward computation gives
\[
\int_{\R^d} f(u-v)\,dv = \rho_f(u-u_c) + \rho_f(u_c - v_c) + \int_{\R^d} f(v_c - v)\,dv,
\]
and taking $L^2$-norm to the above yields
\begin{align*}
\lt\|\int_{\R^d} f(u-v)\,dv\rt\|_{L^2} &\leq \|\rho_f\|_{L^\infty}(\|u - u_c\|_{L^2} + (2\pi)^{\frac{d}{2}}|u_c - v_c| ) + \lt\| \int_{\R^d} f(v_c - v)\,dv \rt\|_{L^2} \cr
&\leq Ce^{-Ct} + \lt\| \int_{\R^d} f(v_c - v)\,dv \rt\|_{L^2},
\end{align*}
where we used the result of Theorem \ref{thm_lt}. For the estimate of the second term on the right hand side of the above inequality, we use the fact
\[
\int_{\T^d}\lt|\int_{\R^d} f(v_c - v)\,dv \rt|^2 dx \leq \int_{\T^d}\rho_f \lt(\int_{\R^d} f|v_c - v|^2\,dv \rt) dx \leq \|\rho_f\|_{L^\infty} \inttr f|v_c - v|^2\,dxdv
\]
together with Theorem \ref{thm_lt} to have
\[
\lt\| \int_{\R^d} f(v_c - v)\,dv \rt\|_{L^2} \leq C\|\rho_f\|_{L^\infty}^{1/2}\lt(\inttr f|v_c - v|^2\,dxdv\rt)^{1/2} \leq Ce^{-Ct}
\]
for some cosntatnt $C>0$, independent of $t$. This completes the proof.
\end{proof}

In order to obtain the time-asymptotic behavior of solutions, we need to estimate the heat kernel on the periodic domain $\T^d, d=2,3$. For an integrable function $u$ on $\T^d$, we define its multiple Fourier series expansion as
\[
u(x) = \sum_{\xi \in \Z^d} c_\mf(u)(\xi) e^{i \xi \cdot x},
\]
where $c_\mf(u)$ is the Fourier coefficient given by
\[
c_\mf(u)(\xi) := \frac{1}{(2\pi)^d} \intt u(x) e^{-i \xi \cdot x}\,dx.
\]
Then the heat semigroup generated by $\Delta_x$, which is defined by
\bq\label{sol_homo}
e^{t\Delta_x}u_0(x) = \sum_{\xi \in \Z^d} e^{-t|\xi|^2 + i \xi \cdot x}c_\mf(u_0)(\xi)
\eq
is the solution of the heat equation
\[
\pa_t u = \Delta_x u, \quad (x,t) \in \T^d \times \R_+
\]
with the initial data $\displaystyle u(x,t)|_{t=0} = u_0(x)$, $x \in \T^d$.
%\[
%u(x,t)|_{t=0} = u_0(x), \quad x \in \T^d.
%\]
Note that the heat semigroup can be rewritten as the convolution formula:
\[
e^{t\Delta_x}u_0(x) = \intt \Gamma(x-y,t) u_0(y)\,dy.
\]
Here the heat kernel $\Gamma(x,t)$ on $\T^d$ is given by
\bq\label{heat_k}
\Gamma(x,t) = \frac{1}{(2\pi)^d} \sum_{\xi \in \Z^d} e^{-t|\xi|^2 + i \xi \cdot x} = \frac{1}{(4\pi t)^{d/2}} \sum_{\xi \in \Z^d} e^{-\frac{|x- 2\pi \xi|^2}{4t}}.
\eq
In the following lemma, we provide some integrability estimates of the heat kernel $\Gamma$.
Although those results are probably standard,  details will be shown in Appendix \ref{Lemma33} since we cannot specify appropriate references.

\begin{lemma}\label{lem_hk} Let  $ \Gamma$ be the heat kernel defined in \eqref{heat_k}.  Then we have
\[
\lt\|\Gamma(\cdot,t) - \frac{1}{(2\pi)^d}\rt\|_{L^p} \leq  C \bke{t^{-1}+t^{-\frac{d}{2}}}^{1-\frac{1}{p}}e^{-t(1-\frac{1}{p})}
\]
%and
\[
\|\nabla_x \Gamma(\cdot,t)\|_{L^p} \leq  C\lt(t^{-\lt(1 - \frac{1}{2p}\rt)} + t^{-\frac d2\lt( 1 - \frac1p\rt) - \frac12} \rt)e^{-t\lt(1 - \frac1p\rt)}
\]
for $d=2,3$ and $1 \leq p \leq \infty$. Here $C>0$ is independent of $t$ and uniformly bounded in $p$.
\end{lemma}
\begin{remark}\label{rem_hk} For $p > 1$ and $t \geq 1$, we obtain from Lemma \ref{lem_hk} that
\[
\lt\|\Gamma(\cdot,t) - \frac{1}{(2\pi)^d}\rt\|_{L^p} + \|\nabla_x \Gamma(\cdot,t)\|_{L^p} \leq Ce^{-Ct},
\]
where $C>0$ is independent of $t$.
\end{remark}

Next, we will show the exponential decay of fluid field. For convenience, we denote the vorticity field  by $\omega = \nabla_x \times u$ in two and three dimensions.

\begin{lemma}\label{lemma-fluid-decay} Suppose that $(f,u)$ is the strong solution  defined in Definition \ref{strong-sol} to the system \eqref{main}-\eqref{initial}. Under the same assumptions as in Theorem   \ref{main-thm}, there exists a constant $C>0$, independent of $t$, such that
\begin{equation}\label{u-expdecay-10}
\norm{u(\cdot,t)-u_c(t)}_{L^{\infty}}\le Ce^{-Ct},
\end{equation}
\begin{equation}\label{u-expdecay-50}
\norm{ \omega(\cdot,t)}_{L^p(\T^2)}\le Ce^{-Ct}\quad \mbox{for} \quad p<\infty,
\end{equation}
%and
\begin{equation}\label{u-expdecay-60}
\norm{ \omega(\cdot,t)}_{L^p(\T^3)}\le Ce^{-Ct}\quad \mbox{for} \quad p < 6.
\end{equation}
\end{lemma}
\begin{proof}
For notational simplicity, we denote $G(x,t):=-\int_{\R^d}(u-v)f\,dv$.
We consider the vorticity equation $\omega = \nabla_x \times u$  for $d=2, 3$:
\begin{equation}\label{u-expdecay-20}
\partial_t \omega - \Delta_x \omega + (1 - \delta_{d,3})u\cdot \nabla_x \omega=\nabla_x\times G.
\end{equation}
First, we provide the exponential decay estimate of $\omega$ in $L^2(\T^d)$. Multiplying \eqref{u-expdecay-20} by $\omega$ and integrating the resulting equation gives the following energy estimate for $\omega$:
\[
\frac{d}{dt}\int_{\T^d} \abs{\omega}^2 dx+\int_{\T^d} \abs{\nabla_x \omega}^2 dx\le \int_{\T^d} \abs{G}^2 dx.
\]
Since  by Lemma \ref{lem_drag2} the right hand side of the above inequality is bounded by $Ce^{-Ct}$, applying Poincar\'e inequality yields
\[
\frac{d}{dt}\int_{\T^d} \abs{\omega}^2 dx+C\int_{\T^d} \abs{ \omega}^2 dx\le Ce^{-Ct}.
\]
We then use the Gr\"onwall's lemma to obtain
\[
\int_{\T^d} \abs{ \omega}^2 dx\le Ce^{-Ct} \quad \forall \, t\geq 0,
\]
where $C>0$ is independent of $t$. 

We next extend the above exponential decay estimate in $L^2(\T^d)$ space to $L^p(\T^d)$, where $p$ depends on the dimension.
First, we consider the case $d=3$ and write
\[
\omega(x,t)=(\Gamma\star\omega_0)(x,t) - \int_0^t \int_{\T^3} \nabla_x\Gamma(x-y,t-s)\times G(y,s)\,dyds,
\]
where $\Gamma$ is the heat kernel appeared in \eqref{heat_k}. Note that $\intt \omega_0\,dx = 0$, then it follows from Remark \ref{rem_hk} that
\[
\norm{(\Gamma*\omega_0)(t)}_{L^{p}(\T^3)}\le Ce^{-Ct}
\]
for all $1 \leq p \leq \infty$ and $t \geq 1$, where $C>0$ is independent of $t$ and $p$. On the other hand, we split the second term as follows:
\begin{align*}
\int_0^t \int_{\T^3}\nabla_x\Gamma(x-y,t-s)\times G(y,s)\,dyds &= \lt(\int_0^{\frac{t}{2}} + \int_{\frac{t}{2}}^t\rt) \int_{\T^3}\nabla_x\Gamma(x-y,t-s)\times G(y,s)\,dyds \cr
&=:I_1 + I_2.
\end{align*}
By definition, $\nabla_x \Gamma$ is regular for $s<t/2$, moreover, by Lemma \ref{lem_hk}, any $L^q$-norm of that is bounded from above by $Ce^{-Ct}$. Thus by using Young's convolution inequality, we estimate $I_1$ as
\[
\|I_1\|_{L^p}\le \int_0^{\frac{t}{2}} \norm{\nabla\Gamma(\cdot, t-s)}_{L^q}
\norm{G(\cdot,s)}_{L^2}ds\le  Ce^{-Ct}\int_0^{\frac{t}{2}}
\norm{G(\cdot,s)}_{L^2}ds\le  Ce^{-Ct},
\]
where $1/p = 1/q - 1/2$. For $I_2$, we use Lemmas \ref{lem_drag2} and \ref{lem_hk} to find
\[
\norm{I_2}_{L^p}\le \int_{\frac{t}{2}}^t (t-s)^{-\frac{3}{2}\bke{\frac{1}{2}-\frac{1}{p}}-\frac{1}{2}}\norm{G(\cdot,s)}_{L^2}ds \le\sup_{\frac{t}{2}\le s\le t}\norm{G(\cdot,s)}_{L^2}t^{-\frac{3}{2}\bke{\frac{1}{2}-\frac{1}{p}}+\frac{1}{2}}\le Ce^{-Ct}.
\]
Here we also used
%\[
$\displaystyle \frac{3}{2}\lt(\frac{1}{2}-\frac{1}{p}\rt)+\frac{1}{2} <1$
%$\]
due to $p < 6$. Combining the above estimates asserts \eqref{u-expdecay-60}, which further implies via Gagliardo-Nirenberg interpolation inequality and Korn's inequality that for $3<p < 6$
\begin{align}\label{decay_u}
\begin{aligned}
\norm{u(\cdot,t)-u_c(t)}_{L^{\infty}} &\le C\norm{u(\cdot,t)-u_c(t)}^{\theta}_{L^{2}}\norm{\nabla_x u(\cdot,t)}^{1-\theta}_{L^p}\cr
&\le C\norm{u(\cdot,t)-u_c(t)}^{\theta}_{L^{2}}\norm{\omega(\cdot,t)}^{1-\theta}_{L^p}+C\norm{u(\cdot,t)-u_c(t)}_{L^{2}}\cr
&\le  C\norm{\omega(\cdot,t)}^{\theta}_{L^{2}}\norm{\omega(\cdot,t)}^{1-\theta}_{L^p}+C\norm{\omega(\cdot,t)}_{L^{2}}\cr
&\le Ce^{-Ct},
\end{aligned}
\end{align}
where $\theta=2(p-3)/(5p-6)$ and we used \eqref{u-expdecay-60}.

Now, it remains to treat the two dimensional case. We notice that
\[
u\otimes \omega\in L^p((0,\infty);L^q(\T^2)) \quad \mbox{for all } \ 1\leq q <2 \mbox{ and } 1 \leq p \leq \infty
\]
and 
\[
\norm{(u\otimes \omega)(\cdot,t)}_{L^q}\le Ce^{-Ct}.
\]
Indeed, let $q*>0$ with $1/q*=1/q-1/2$
\begin{align*}
\norm{(u\omega)(\cdot,t)}_{L^q} &\le \norm{(u(\cdot,t)-u_c(t))\omega(\cdot,t)}_{L^q} +|u_c(t)|\norm{\omega(\cdot,t)}_{L^q}\cr
&\le \norm{u(\cdot,t)-u_c(t)}_{L^{q*}}\norm{\omega(\cdot,t)}_{L^2}+C\norm{u(\cdot,t)}_{L^2}\norm{\omega(\cdot,t)}_{L^2}\cr
&\le \norm{\nabla_x u(\cdot,t)}_{L^{2}}\norm{\omega(\cdot,t)}_{L^2}+C\norm{u(\cdot,t)}_{L^2}\norm{\omega(\cdot,t)}_{L^2}\cr
&\le C\|\omega(\cdot,t)\|^2_{L^2}+C\|u(\cdot,t)\|_{L^2}\norm{\omega(\cdot,t)}_{L^2}\cr
&\le Ce^{-Ct},
\end{align*}
where $C>0$ is independent of $t$. Since $G$ is in $L^{\infty}(\R_+;L^2(\T^2))$ by Lemma \ref{lem_drag2}, it also belongs to $L^{\infty}(\R_+;L^q(\T^2))$ for $1\le q<2$. Now we use the representation formula
\begin{align*}
\omega(x,t)&=(\Gamma\star\omega_0)(x,t) - \int_0^t\int_{\T^2}\nabla_x\Gamma(x-y,t-s) (u\otimes\omega)(y,s)\,dyds\cr
&\quad -\int_0^t\int_{\T^2}\nabla_x\Gamma(x-y,t-s)\times G(y,s)\,dyds.
\end{align*}
The first term, we can have the same estimate as in the three dimensional case. Thus it suffices to estimate the other terms.  It is worth noticing that the terms $G$ and $u\otimes \omega$ are in $L^{\infty}(\R_+;L^2(\T^2))$, and moreover they have the same exponential decay estimates.
For that reason, we only provide the decay estimate of one of them, say $G$. Similarly as before, for any $2\le p<q/(2-q)$ let us divide that into two terms:
\[
\lt(\int_0^{\frac{t}{2}} + \int_{\frac{t}{2}}^t\rt) \int_{\T^2}\nabla_x\Gamma(x-y,t-s)\times G(y,s)\,dyds =:I_3 + I_4.
\]
Then by using almost the same argument as before, we obtain
\[
\norm{I_3}_{L^p}\le \int_0^{\frac{t}{2}} \norm{\nabla_x\Gamma(\cdot,t-s)}_{L^r}
\|G(\cdot,s)\|_{L^q}\,ds\le  Ce^{-Ct}\int_0^{\frac{t}{2}}
\norm{G(\cdot,s)}_{L^q}ds\le  Ce^{-Ct}
\]
and
\[
\norm{I_4}_{L^p}\le \int_{\frac{t}{2}}^t (t-s)^{-\bke{\frac{1}{q}-\frac{1}{p}}-\frac{1}{2}}\norm{G(\cdot,s)}_{L^q}ds \le\sup_{\frac{t}{2}\le s\le t}\norm{G(\cdot,s)}_{L^q}t^{-\bke{1-\frac{1}{r}}+\frac{1}{2}}\le Ce^{-Ct},
\]
where $1=1/r+1/q-1/p$ and $C>0$ is independent of $t$ and uniformly bounded in $p$. This asserts \eqref{u-expdecay-50}. We then use a similar argument as \eqref{decay_u} to get
\[
\norm{u(\cdot,t)-u_c(t)}_{L^{\infty}} \leq C\norm{\omega(\cdot,t)}^{\theta}_{L^{2}}\norm{\omega(\cdot,t)}^{1-\theta}_{L^p}+C\norm{\omega(\cdot,t)}_{L^{2}}\le Ce^{-Ct},
\]
where $\theta = 2(p-2)/(4p-4)$.
This completes the proof.
\end{proof}

Next proposition shows the exponential decay estimate of generalized moments of particle distribution function $f$, which is a part of Theorem \ref{main-thm} for the case $k=0$.

\begin{proposition}\label{thm_main} Let $q \geq 1$ and $d=2$ or $3$.
Suppose that  $(f,u)$ is the strong solution  defined in Definition \ref{strong-sol} to the system \eqref{main}-\eqref{initial}. Under the same assumptions as in Theorem   \ref{main-thm},
there exists $p=p(q, d, \psi, f_0, u_0) > 2$ large enough such that
\[
\lt(\inttr |v - v_c(t)|^{p} f^{q}(x,v,t)\,dxdv\rt)^{1/p} \leq Ce^{-Ct} \quad \forall \,t \geq 1,
\]
where $C > 0$ is independent of $p$ and $t$.
\end{proposition}

\begin{proof}%[Proof of Theorem \ref{thm_main}]
Straightforward computations give
\begin{align*}
&\frac{1}{q}\frac{d}{dt}\inttr |v - v_c(t)|^{p} f^{q}\,dxdv\cr
&\quad = -\frac pq \inttr |v - v_c(t)|^{p-2}(v - v_c) \cdot v_c'(t) f^{q}\,dxdv + \inttr |v - v_c(t)|^{p} f^{q-1} \pa_t f\,dxdv
\cr
&\quad 
=: J_1 + J_2.
\end{align*}
To estimate $J_1$, we note, due to Theorem \ref{thm_lt}, that
\begin{align*}
|v_c'(t)| &= \lt|\inttr (u-v)f\,dxdv\rt|
%\cr
%&
=\lt| \inttr (u- u_c(t) + u_c(t) - v_c(t) + v_c(t) - v)f\,dxdv \rt|\cr
&\leq C\|\rho_f\|_{L^\infty}\lt(\intt |u - u_c(t)|^2\,dx \rt)^{1/2} + M_0|u_c(t) - v_c(t)| + \lt( \inttr |v - v_c(t)|^2 f\,dxdv\rt)^{1/2}\cr
&\leq Ce^{-Ct}
\end{align*}
for some $C>0$ independent of $t$. Thus it follows from Young's inequality that
\begin{align*}
J_1 &\leq \frac pq \inttr |v-v_c(t)|^{p-1}|v_c'(t)|f^{q}\,dxdv \cr
&\leq \frac{p \delta}{q} \inttr |v-v_c(t)|^{p} f^{q}\,dxdv + C_{q,\delta}^p e^{-Cp t}\inttr f^{q}\,dxdv,
\end{align*}
where $C_{q,\delta} > 0$ is independent of $p$ and $\delta >0$ will be determined later.

On the other hand, using the integration by parts, we compute $J_2$ as follows:
\begin{align*}
J_2 &= \inttr |v - v_c(t)|^{p} f^{q-1} ( - v \cdot \nabla_x f - \nabla_v \cdot (F(f,u)f))\,dxdv\cr
&=  \inttr \nabla_v \lt( |v - v_c(t)|^{p} f^{q-1}\rt) \cdot F(f,u) f\,dxdv,
\end{align*}
where we used 
\[
-\inttr |v - v_c(t)|^{p} f^{q-1}  v \cdot \nabla_x f \,dxdv= -\frac{1}{q} \inttr | v - v_c(t)|^{p} \nabla_x \cdot \lt(v f^{q} \rt)\,dxdv = 0.
\]
Continuing computations for $J_2$, we get
\begin{align*}
J_2 &= \inttr \lt(p|v-v_c(t)|^{p-2} (v - v_c(t))f^{q-1} + |v - v_c(t)|^{p} (q-1)f^{q-2} \nabla_v f \rt) \cdot F(f,u) f\,dxdv\cr
&= p\inttr |v - v_c(t)|^{p-2} (v - v_c(t))\cdot F(f,u) f^{q}\,dxdv 
%\cr
%&\quad 
+ (q-1)\inttr  |v - v_c(t)|^{p} f^{q-1}\nabla_v f \cdot F(f,u)\,dxdv.
\end{align*}
We further estimate the second term on the right side of the above equality as
\begin{align*}
&(q-1)\inttr  |v - v_c(t)|^{p} f^{q-1}\nabla_v f \cdot F(f,u)\,dxdv\cr
&\quad = \frac{q-1}{q}\inttr |v - v_c(t)|^{p} \nabla_v f^{q} \cdot F(f,u)\,dxdv
%\cr
%&\quad 
= \frac{1-q}{q}\inttr \nabla_v \cdot \lt(|v - v_c(t)|^{p} F(f,u) \rt)f^{q}\,dxdv\cr
&\quad = \frac{1-q}{q}\inttr \lt(p|v-v_c(t)|^{p-2}(v-v_c(t)) \cdot F(f,u) - d|v- v_c(t)|^{p}(1 + \psi\star\rho_f) \rt) f^{q}\,dxdv.
\end{align*}
Summing up, we have
\begin{align*}
J_2 &= \frac pq\inttr |v-v_c(t)|^{p-2}(v - v_c(t)) \cdot F(f,u) f^{q}\,dxdv
\cr
&\quad 
+ d\lt(1 - \frac{1}{q} \rt)\inttr (1 + \psi \star \rho_f) |v - v_c(t)|^{p}f^{q}\,dxdv
%\cr
%& 
=: J_2^1 +J_2^2.
\end{align*}
We next estimate the first term $J_2^1=: J_2^{11}+J_2^{12}$ as
\begin{align*}
J_2^{11}&:= \frac pq \inttr |v-v_c(t)|^{p-2}(v - v_c(t)) \cdot F_a(f) f^{q}\,dxdv\cr
&= \frac pq \inttrd |v-v_c(t)|^{p-2}(v - v_c(t)) \cdot (w-v_c(t) + v_c(t) - v) \psi(x-y) f(y,w) f^{q}(x,v)\,dxdydvdw \cr
&= -\frac pq \inttrd |v- v_c(t)|^{p}\psi(x-y)f(y,w) f^{q}(x,v)\,dxdydvdw \cr
&\quad + \frac pq \inttrd |v-v_c(t)|^{p-2}(v - v_c(t)) \cdot (w-v_c(t)) \psi(x-y) f(y,w) f^{q}(x,v)\,dxdydvdw\cr
& \leq - \frac{p\psi_m M_0}{q} \inttr | v- v_c(t)|^{p} f^{q}\,dxdv 
\cr
&\quad 
+ \frac{p\psi_M}{q} \inttr | v - v_c(t)|^{p-1} \lt(\inttr | w- v_c(t)|^2 f\,dydw\rt)^{1/2} f^{q}\,dxdv \cr
& \leq -\frac{p\psi_m M_0}{q} \inttr |v - v_c(t)|^{p}f^{q}\,dxdv + C_{q,\delta}^p e^{-Cp t}\inttr f^{q}\,dxdv 
\cr
&\quad 
+ \frac{p \psi_M \delta}{q}\inttr  |v - v_c(t)|^{p}f^{q}\,dxdv.
\end{align*}
Here $\delta >0$ will be chosen small enough later. Similarly, we also obtain
\begin{align*}
J_2^{12} &:=\frac pq \inttr |v-v_c(t)|^{p-2}(v - v_c(t)) \cdot (u-v) f^{q}\,dxdv\cr
& =\frac pq\inttr |v-v_c(t)|^{p-2}(v - v_c(t)) \cdot (u-u_c(t) + u_c(t) - v_c(t) + v_c(t) - v) f^{q}\,dxdv\cr
& \leq \frac pq \inttr |v-v_c(t)|^{p-1}|u - u_c(t)|f^{q}\,dxdv + \frac pq \inttr |v-v_c(t)|^{p-1}|u_c(t) - v_c(t)|f^{q}\,dxdv\cr
& \quad - \frac pq \inttr |v - v_c(t)|^{p}f^{q}\,dxdv\cr
& \leq -\frac{p}{q}\lt(1 - \delta \rt) \inttr |v - v_c(t)|^{p}f^{q}\,dxdv + C_{q,\delta}^p\|u - u_c(t)\|_{L^\infty}^{p} \inttr f^{q}\,dxdv 
\cr
&\quad 
+ C_{q,\delta}^p e^{-Cpt}\inttr f^{q}\,dxdv.
\end{align*}
Adding up the estimates, we obtain
\begin{align*}
&\frac{d}{dt}\inttr |v - v_c(t)|^{p} f^{q}\,dxdv \cr
&\quad \leq -p\psi_mM_0 \inttr |v - v_c(t)|^{p}f^{q}\,dxdv + C_{q,\delta}^p e^{-Cp t}\inttr f^{q}\,dxdv  -p(1 - C\delta) \inttr |v - v_c(t)|^{p}f^{q}\,dxdv\cr
&\qquad + C_{q,\delta}^p \|u - u_c(t)\|_{L^\infty}^{p} \inttr f^{q}\,dxdv - d\lt(q - 1 \rt)\inttr (1 + \psi \star \rho_f) |v - v_c(t)|^{p}f^{q}\,dxdv\cr
&\quad \leq -p\lt( 1 + \psi_mM_0 - C \delta \rt)\inttr |v - v_c(t)|^{p}f^{q}\,dxdv + C_{q,\delta}^p \lt(e^{-Cpt} + \|u - u_c(t)\|_{L^\infty}^{p} \rt) \inttr f^{q}\,dxdv,
\end{align*}
where $C>0$ is independent of $\delta$ and $p$, and we used $1 + \psi \star \rho_f \geq 1 + \|\rho_f\|_{L^1}\psi_m = 1+ \psi_mM_0 >0$ for $M_0 > 0$ small enough.
We further use Lemma \ref{f-growth} to have
\begin{align}\label{decay_fq}
\begin{aligned}
&\frac{d}{dt}\inttr |v - v_c(t)|^{p} f^{q}\,dxdv
\leq -p\lt( 1 + \psi_m M_0 - C \delta \rt)\inttr |v - v_c(t)|^{p}f^{q}\,dxdv\cr
&\qquad + C_{q,\delta}^p \lt(e^{-Cpt} + \|u - u_c(t)\|_{L^\infty}^{p} \rt) \lt(\inttr f_0^{q}\,dxdv\rt) e^{d(q-1)(\psi_M+1)t},
\end{aligned}
\end{align}
where $C_{q,\delta} > 0$ is independent of $p$.
Putting the estimate  \eqref{u-expdecay-10} into \eqref{decay_fq} yields
\begin{align*}
\frac{d}{dt}\inttr |v - v_c(t)|^{p} f^{q}\,dxdv
& \leq -p\lt( 1 + \psi_m M_0 - C\delta \rt)\inttr |v - v_c(t)|^{p}f^{q}\,dxdv\cr
&\quad + C_{q,\delta}^p \lt(e^{-Cpt} \rt) \lt(\inttr f_0^{q}\,dxdv\rt) e^{d(q-1)(\psi_M+1)t}.
\end{align*}
We finally choose $\delta > 0$ small enough so that $1 + \psi_m -C\delta > 0$ and apply Gr\"onwall's lemma to have
\begin{align*}
\lt(\inttr |v - v_c(t)|^{p} f^{q}\,dxdv \rt)^{1/p} &\leq \lt(\inttr |v - v_c(0)|^{p} f_0^{q}\,dxdv \rt)^{1/p} e^{-\lt( 1 + \psi_mM_0 -C\delta\rt)t}\cr
&\quad + C_{q,\delta} e^{- (C - d(q-1)(\psi_M+1)/p)t},
\end{align*}
where $C_{q,\delta} > 0$ is independent of $p$ and $t$. This completes the proof.
\end{proof}
\begin{remark}The decay estimate in Proposition \ref{thm_main} still holds even in the absence of velocity alignment force, i.e., $\psi \equiv 0$.
\end{remark}
As a consequence of Proposition \ref{thm_main}, we can show that the support of $f$ in velocity shrinks to a point $(M_0 v_c(0) + u_c(0))/(1 + M_0)$ with an exponential rate of decay. This proves Theorem \ref{main-thm} (iii). Recall that $\Sigma_v(t)$ represents the $v$-projection of support of $f$.

\begin{corollary}\label{main_cor}
Suppose that $(f,u)$ is the strong solution  defined in Definition \ref{strong-sol} to the system \eqref{main}-\eqref{initial}.
Let $v^\infty = (M_0v_c(0) + u_c(0))/(1 + M_0)$.
Under the same assumptions as in Theorem   \ref{main-thm},
we have
\[
\sup_{v \in \Sigma_v(t)} |v - v^\infty| \leq Ce^{-Ct} \quad \forall\, t\geq 0,
\]
where $C>0$ is independent of $t$.
\end{corollary}
\begin{proof} It follows from Remark \ref{rmk_limit} and Theorem \ref{thm_main} that
\[
\sup_{p \geq 1}\lt(\inttr | v - v^\infty|^p f(x,v,t)\,dxdv\rt)^{1/p} \leq C e^{-Ct}.
\]
Set
\[
S_v(t) := \sup_{v \in \Sigma_v(t)} |v - v^\infty|.
\]
We claim that
\bq\label{lim_mp}
\lt( \inttr | v - v^\infty |^p f(x,v,t)\,dxdv\rt)^{1/p} \to S_v(t) \quad \mbox{as} \quad p \to \infty.
\eq
It is clear that
\[
\lt( \inttr | v - v^\infty |^p f(x,v,t)\,dxdv\rt)^{1/p} \leq S_v(t) \lt(\inttr f(x,v,t)\,dxdv \rt)^{1/p} = S_v(t)M_0^{1/p} \leq S_v(t).
\]
On the other hand, for a fixed $\epsilon > 0$ and $t > 0$, if we let
\[
\Sigma^\epsilon_v(t):= \lt\{v \in \Sigma_v(t) : |v - v^\infty| \geq S_v(t) - \epsilon\rt\}
\]
for $\epsilon < S_v(t)$, then we find
\begin{align*}
\lt( \inttr | v -v^\infty|^p f(x,v,t)\,dxdv\rt)^{1/p} &\geq \lt( \int_{\T^d \times \Sigma^\epsilon_v(t)} \lt( S_v(t) - \epsilon\rt)^p f(x,v,t)\,dxdv\rt)^{1/p}\cr
&=\lt( S_v(t) - \epsilon\rt)\lt( \int_{\T^d \times \Sigma^\epsilon_v(t)}  f(x,v,t)\,dxdv\rt)^{1/p}.
\end{align*}
This implies
\[
\liminf_{p \to \infty}\lt( \inttr | v - v^\infty|^p f(x,v,t)\,dxdv\rt)^{1/p} \geq S_v(t) - \epsilon,
\]
and we have \eqref{lim_mp}. This completes the proof.
\end{proof}

%%%%%%%%%%%%%%%%%%%%%%%%%%%%%%%%%%%%%
%
%
% \section{High-order estimates}
%
%
%%%%%%%%%%%%%%%%%%%%%%%%%%%%%%%%%%%%%

\subsection{Exponential decays of weighted $W^{2,q}$ estimates}\label{sec:high}

In this subsection, we estimate $W^{2,q}$ norm of $f$. For this, we need to control the propagation of support of $f$ in velocity.

Next lemma shows that the support of $f$ in velocity is uniformly bounded in time.

\begin{lemma}\label{lem_supp}
Suppose that $(f,u)$ is the strong solution  defined in Definition \ref{strong-sol} to the system \eqref{main}-\eqref{initial}.
Under the same assumptions as in Theorem   \ref{main-thm}, we have
\[
R_v(t):= \max_{v \in \Sigma_v(t)}|v| \leq C_0
\]
for some $C_0 > 0$ which is independent of $t$.
\end{lemma}
\begin{proof} Note that $|v| \leq |v - v^\infty| + |v^\infty|$, where $v^\infty = (M_0 v_c(0) + u_c(0))/(1 + M_0)$. Taking the maximum of the above inequality over $\Sigma_v(t)$ and using Corollary \ref{main_cor} give
\[
R_v(t) \leq \max_{v \in \Sigma_v(t)}|v - v^\infty| + |v^\infty| \leq Ce^{-Ct} + |v^\infty| \leq C_0
\]
for some $C_0 > 0$ which is independent of $t$.
\end{proof}
We next present an auxiliary lemma showing the exponential decay of $L^p$-norm of the drag force in the fluid equation in \eqref{main}.

\begin{lemma}\label{lem_drag} Suppose that  $(f,u)$ is the strong solution  defined in Definition \ref{strong-sol} to the system \eqref{main}-\eqref{initial}. Under the same assumptions as in Theorem  \ref{main-thm},  we have
\[
\lt\|\int_{\R^d} f(u-v)\,dv\rt\|_{L^p} \leq Ce^{-Ct}
\]
for $2\le p\leq \infty$, where $C>0$ is independent of $t$ and $p$.
\end{lemma}

\begin{proof} A straightforward computation gives
\[
\int_{\R^d} f(u-v)\,dv = \rho_f(u-u_c) + \rho_f(u_c - v_c) + \int_{\R^d} f(v_c - v)\,dv,
\]
and taking $L^p$-norm to the above yields
\begin{align*}
\lt\|\int_{\R^d} f(u-v)\,dv\rt\|_{L^p} &\leq \|\rho_f\|_{L^p}(\|u - u_c\|_{L^\infty} + |u_c - v_c| ) + \lt\| \int_{\R^d} f(v_c - v)\,dv \rt\|_{L^p} \cr
&\leq Ce^{-Ct} + \lt\| \int_{\R^d} f(v_c - v)\,dv \rt\|_{L^p}.
\end{align*}
On the other hand, we note that
\[
\int_{\T^d}\lt|\int_{\R^d} f(v_c - v)\,dv \rt|^p dx \leq \int_{\T^d}\rho_f^{p-1} \lt(\int_{\R^d} f|v_c - v|^p\,dv \rt) dx \leq \|\rho_f\|_{L^\infty}^{p-1}\inttr f|v_c - v|^p\,dxdv.
\]
Therefore, due to Theorem \ref{thm_main}, we obtain
\[
\lt\|\int_{\R^d} f(u-v)\,dv\rt\|_{L^p} \leq \|\rho_f\|_{L^\infty}^{\frac{p-1}{p}}
\bke{\inttr f|v_c - v|^p\,dxdv}^{\frac{1}{p}}\le Ce^{-Ct},
\]
where $C>0$ is independent of $t$ and $p$. This completes the proof.
\end{proof}

Next, we control derivatives of $f$ with respect to phase variables in terms of spatial gradients of fluid velocity.

\begin{lemma}\label{lem_f}Let $T>0$ and suppose that $(f,u)$ is the strong solution  defined in Definition \ref{strong-sol} to the system \eqref{main}-\eqref{initial}. Under the same assumptions as in Theorem  \ref{main-thm},  we have
\[
\|f(\cdot,\cdot,t)\|_{W^{1,q}} \leq \|f_0\|_{W^{1,q}} \exp\lt( Ct + C\int_0^t \|\nabla_x u\|_{L^\infty}\,d\tau\rt),
\]
where $C>0$ is independent of $t$.

\end{lemma}
\begin{proof}

($\diamond$ Estimate of $\|\nabla_x f\|_{L^{q}}$): Taking $\pa_{x_j},j=1,\dots,d$ to the kinetic equation in \eqref{main} gives
\begin{align*}
\pa_t \pa_{x_j} f + v \cdot \nabla_x \pa_{x_j}f &= -\nabla_v \cdot\lt( (\pa_{x_j} F) f + F \pa_{x_j} f \rt)\cr
&= - (\nabla_v \cdot  \pa_{x_j} F)f -  \pa_{x_j} F \cdot \nabla_v f - (\nabla_v \cdot F) \pa_{x_j}f - F \cdot \nabla_v  \pa_{x_j} f.
\end{align*}
Note that
\[
\nabla_v \cdot F = -d(1 + \psi \star \rho_f),  \quad  \pa_{x_j} F =  \pa_{x_j} \psi \star (\rho_f u_f) - v  \pa_{x_j} \psi \star \rho_f +  \pa_{x_j} u,\quad
%\]
%and
%\[
\pa_{x_j} \nabla_v \cdot F = -d  \pa_{x_j} \psi \star \rho_f.
\]
Then we get
\begin{align*}
&\frac{d}{dt} \inttr | \pa_{x_j} f|^{q}\,dxdv
\quad = q \inttr | \pa_{x_j} f|^{q-2} \pa_{x_j} f  \pa_t \pa_{x_j} f\,dxdv\cr
&\quad =q \inttr | \pa_{x_j} f|^{q-2} \pa_{x_j} f  \lt( - (\nabla_v \cdot  \pa_{x_j} F)f -  \pa_{x_j} F \cdot \nabla_v f - (\nabla_v \cdot F) \pa_{x_j}f - F \cdot \nabla_v  \pa_{x_j} f\rt)dxdv
\cr
&\quad 
=: \sum_{i=1}^4 I_i.
\end{align*}
We estimate $I_i,i=1,\dots,4$ separately.
\begin{align*}
I_1 &= qd \inttr   | \pa_{x_j} f|^{q-2} \pa_{x_j} f (\pa_{x_j} \psi \star \rho_f) f\,dxdv
%\cr
%&\leq qd \|\pa_{x_j}\psi\|_{L^\infty} \inttr | \pa_{x_j} f|^{q-1} f\,dxdv \cr
%&
\leq qd \|\pa_{x_j}\psi\|_{L^\infty} \|\pa_{x_j} f\|_{L^{q}}^{q-1} \|f\|_{L^{q}},
\end{align*}
\begin{align*}
I_2 &= -q \inttr  | \pa_{x_j} f|^{q-2} \pa_{x_j} f \pa_{x_j} F \cdot \nabla_v f\,dxdv
%\cr
%&\leq q \|\pa_{x_j} F\|_{L^\infty} \inttr | \pa_{x_j} f|^{q-1}|\nabla_v f|\,dxdv\cr
%&
\leq qd \|\pa_{x_j} F\|_{L^\infty}  \|\pa_{x_j} f\|_{L^{q}}^{q-1} \|\nabla_v f\|_{L^{q}},
\end{align*}
\[
I_3 = qd \inttr |\pa_{x_j} f|^{q} (1 + \psi \star \rho_f)\,dxdv,
\]
%and
\begin{align*}
I_4 &= - \inttr F \cdot \nabla_v |\pa_{x_j}f|^{q}\,dxdv
%\cr
%&= \inttr (\nabla_v \cdot F) |\pa_{x_j}f|^{q}\,dxdv\cr
%&
= -d \inttr |\pa_{x_j} f|^{q} (1 + \psi \star \rho_f)\,dxdv.
\end{align*}
Adding up estimates, it follows  for $q\geq 1$ that
\begin{align*}
\frac{d}{dt} \|\pa_{x_j} f\|_{L^{q}}^{q}  &\leq qd \|\pa_{x_j}\psi\|_{L^\infty} \|\pa_{x_j} f\|_{L^{q}}^{q-1} \|f\|_{L^{q}} + qd \|\pa_{x_j} F\|_{L^\infty}  \|\pa_{x_j} f\|_{L^{q}}^{q-1} \|\nabla_v f\|_{L^{q}}\cr
&\quad + d(q-1) \inttr |\pa_{x_j} f|^{q} (1 + \psi \star \rho_f)\,dxdv\cr
&\leq qd \|\pa_{x_j}\psi\|_{L^\infty} \|\pa_{x_j} f\|_{L^{q}}^{q-1} \|f\|_{L^{q}} + qd \|\pa_{x_j} F\|_{L^\infty}  \|\pa_{x_j} f\|_{L^{q}}^{q-1} \|\nabla_v f\|_{L^{q}}\cr
&\quad + d(q-1) (1 + \|\psi\|_{L^\infty}) \|\pa_{x_j} f\|_{L^{q}}^{q},
\end{align*}
that is,
\begin{align}\label{est_f1}
\begin{aligned}
\frac{d}{dt} \|\pa_{x_j} f\|_{L^{q}} &\leq d \|\pa_{x_j}\psi\|_{L^\infty} \|f\|_{L^{q}} + d \|\pa_{x_j} F\|_{L^\infty}  \|\nabla_v f\|_{L^{q}} + d\lt(1-\frac{1}{q}\rt) (1 + \|\psi\|_{L^\infty}) \|\pa_{x_j} f\|_{L^{q}}.
\end{aligned}
\end{align}
On the other hand, it follows from Lemma \eqref{lem_supp} that
\begin{align}\label{est_F}
\begin{aligned}
\|\pa_{x_j} F\|_{L^\infty} &\leq \|\pa_{x_j} \psi \star (\rho_f u_f)\|_{L^\infty} +\| v  \pa_{x_j} \psi \star \rho_f\|_{L^\infty} +  \|\pa_{x_j} u\|_{L^\infty}\cr
&\leq \|\pa_{x_j} \psi\|_{L^\infty}\|\rho_f u_f\|_{L^1} + R_v(t)\|\pa_{x_j} \psi\|_{L^\infty} \|\rho_f\|_{L^1} + C\|\nabla_x u\|_{L^\infty}\cr
&\leq C(1  + \|\nabla_x u\|_{L^\infty}),
\end{aligned}
\end{align}
where $C > 0$ depends on $u_c(0)$, $v_c(0)$, $E_0$,  and $\psi$, but independent of $t$ and $q$. Here we also used
\[
\|\rho_fu_f\|_{L^1} \leq \inttr |v|f\,dxdv \leq \|\rho_f\|_{L^1}^{1/2}\lt(\inttr |v|^2f\,dxdv \rt)^{1/2} \leq E_0^{1/2},
\]
where
\[
E_0 := \inttr |v|^2f_0\,dxdv + \intt |u_0|^2\,dx.
\]
The estimate \eqref{est_F} together with \eqref{est_f1} asserts that
\begin{align}\label{est_xf}
\begin{aligned}
\frac{d}{dt} \|\nabla_x f\|_{L^{q}} &\leq Cd \|f\|_{L^{q}} + Cd(1 + \|\nabla_x u\|_{L^\infty})  \|\nabla_v f\|_{L^{q}} + Cd\lt(1-\frac{1}{q}\rt)\|\nabla_x f\|_{L^{q}},
\end{aligned}
\end{align}
where $C > 0$ depends on $E_0$, and $\psi$, but independent of $t$ and $q$. \newline

($\diamond$ Estimate of $\|\nabla_v f\|_{L^{q}}$): Differentiate the kinetic equation in \eqref{main} with respect to $v_j$, we obtain
\[
\pa_t \pa_{v_j} f + v \cdot \nabla_x \pa_{v_j} f = -\pa_{x_j} f - \nabla_v \cdot F \pa_{v_j} f - \pa_{v_j}F \cdot \nabla_v f - F\cdot \nabla_v \pa_{v_j} f,
\]
where we used $\pa_{v_j} \nabla_v \cdot F = 0$. Then we estimate
\begin{align*}
&\frac{d}{dt}\inttr |\pa_{v_j} f|^{q}\,dxdv
%\cr
%&\quad 
= q\inttr |\pa_{v_j} f|^{q-2} \pa_{v_j} f \pa_t \pa_{v_j}f\,dxdv \cr
&\quad = q\inttr |\pa_{v_j} f|^{q-2} \pa_{v_j} f \lt(-\pa_{x_j} f - \nabla_v \cdot F \pa_{v_j} f - \pa_{v_j}F \cdot \nabla_v f - F\cdot \nabla_v \pa_{v_j} f \rt)dxdv \cr
&\quad = -q\inttr |\pa_{v_j} f|^{q-2}\pa_{v_j} f \pa_{x_j} f\,dxdv + qd\inttr (1 + \psi \star \rho_f)|\pa_{v_j}f|^{q}\,dxdv \cr
&\qquad + q\inttr (1 + \psi \star \rho_f)|\pa_{v_j}f|^{q}\,dxdv - d\inttr (1 + \psi \star \rho_f)|\pa_{v_j}f|^{q}\,dxdv \cr
&\quad \leq q\inttr |\pa_{v_j} f|^{q-1}|\pa_{x_j} f|\,dxdv + (q(1+d)-d)\inttr (1 + \psi \star \rho_f)|\pa_{v_j}f|^{q}\,dxdv.
\end{align*}
Here we used
\[
\pa_{v_j} F \cdot \nabla_v f = -(1 + \psi \star\rho_f)\pa_{v_j}f.
\]
Thus for $q \geq d/(2(1+d))$ we find
\[
\frac{d}{dt}\|\pa_{v_j} f\|_{L^{q}}^{q} \leq q\|\pa_{v_j} f\|_{L^{q}}^{q-1}\|\pa_{x_j} f\|_{L^{q}}+ (q(1+d)-d)(1+\|\psi\|_{L^\infty})\|\pa_{v_j} f\|_{L^{q}}^{q}
\]
and, this subsequently implies
\[
\frac{d}{dt}\|\pa_{v_j} f\|_{L^{q}} \leq \|\pa_{x_j} f\|_{L^{q}} + \lt((1+d)-\frac{d}{q}\rt)(1+\|\psi\|_{L^\infty})\|\pa_{v_j} f\|_{L^{q}}.
\]
Hence we have
\bq\label{est_vf}
\frac{d}{dt}\|\nabla_v f\|_{L^{q}} \leq \|\nabla_x f\|_{L^{q}} + C\lt((1+d)-\frac{d}{q}\rt)\|\nabla_v f\|_{L^{q}},
\eq
where $C > 0$ depends on $\psi$, but independent of $t$ and $q$. We now combine \eqref{est_xf} and \eqref{est_vf} to have
\begin{align*}
&\frac{d}{dt}\lt(\|\nabla_x f\|_{L^{q}} + \|\nabla_v f\|_{L^{q}} \rt)\cr
&\quad \leq C\lt( 1+ d\lt(1-\frac{1}{q}\rt)\rt)\|\nabla_x f\|_{L^{q}} + C\lt(1 + d  + d\|\nabla_x u\|_{L^\infty} - \frac{d}{q}\rt)  \|\nabla_v f\|_{L^{q}} + Cd\|f\|_{L^{q}},
\end{align*}
where $C > 0$ depends on $E_0$,  and $\psi$, but independent of $t$, $d$, and $q$. We then combine this with \eqref{f} to obtain
\begin{align*}
&\frac{d}{dt}\lt(\|f\|_{L^{q}} + \|\nabla_x f\|_{L^{q}} + \|\nabla_v f\|_{L^{q}} \rt)\cr
&\quad \leq C\lt( 1+ d\lt(1-\frac{1}{q}\rt)\rt)\|\nabla_x f\|_{L^{q}} + C\lt(1 + d   + d\|\nabla_x u\|_{L^\infty} - \frac{d}{q}\rt)  \|\nabla_v f\|_{L^{q}} + Cd\lt(1 - \frac{1}{q} \rt)\|f\|_{L^{q}}\cr
&\quad \leq Cd(1 + \|\nabla_x u\|_{L^\infty})\lt(\|f\|_{L^{q}} + \|\nabla_x f\|_{L^{q}} + \|\nabla_v f\|_{L^{q}} \rt),
\end{align*}
where $C > 0$ depends on $u_c(0)$, $v_c(0)$, $E_0$, and $\psi$, but independent of $t$, $d$, and $q$.  This further yields
\[
\|f(\cdot,\cdot,t)\|_{W^{1,q}} \leq \|f_0\|_{W^{1,q}} \exp\lt( Cdt + Cd\int_0^t \|\nabla_x u\|_{L^\infty}\,d\tau\rt).
\]
This completes the proof.
\end{proof}

We also provide the growth estimate of $\|f(\cdot,\cdot,t)\|_{W^{2,q}}$ in the lemma below. Since its proof is quite lengthy and technical, we postpone it in Appendix \ref{app_a}.

\begin{lemma}\label{lem_f2}Let $T>0$ and suppose that $(f,u)$ is the strong solution  defined in Definition \ref{strong-sol} to the system \eqref{main}-\eqref{initial}. Under the same assumptions as in Theorem  \ref{main-thm},  we have
\[
\|f(\cdot,\cdot,t)\|_{W^{2,q}} \leq \|f_0\|_{W^{2,q}}  \exp\lt( Ct + \int_0^t\|\nabla_x u(\cdot,\tau)\|_{L^\infty}^2 + \|\nabla_x^2 u(\cdot,\tau)\|_{L^{q}}\,d\tau  \rt)
\]
for $q > d$, where $C>0$ is independent of $t$.
\end{lemma}

%%%%%%%%%%%%%%%%%%%%%%%%%%%%%%%%%%%%%%%%%%%%%%%%
%
%
%
%
%
% \subsection{Two dimensional case}
%
%
%
%
%%%%%%%%%%%%%%%%%%%%%%%%%%%%%%%%%%%%%%%%%%%%%%%%

\subsubsection{Two dimensional case}
In this part, we consider the two dimensional case to obtain  upper bounds of growth rates for
derivatives of $f$.
For convenience, we denote the vorticity field by $\omega = \nabla_x \times u = \pa_{x_1} u_2 - \pa_{x_2} u_1$.
We start by providing some upper bound estimates for $u$ and a decay estimate of $\omega$ in the lemma below.

\begin{lemma}\label{lem_regu}Let $d=2$. Suppose that $(f,u)$ is the strong solution  defined in Definition \ref{strong-sol} to the system \eqref{main}-\eqref{initial}. Under the same assumptions as in Theorem  \ref{main-thm}, we have
\[
\|u\|_{L^2(0,T;H^2)} \leq C\lt(1 + \sqrt{T}\rt),\qquad
\|\omega (\cdot, t)\|_{L^2} \leq Ce^{-Ct},
\]
and
\[
\sup_{t \geq 0}\|\nabla_x u(\cdot,t)\|_{L^2}+\|\pa_t u\|_{L^2(0,T;L^2)}\leq C,
\]
where $C > 0$ is independent of $t$.
\end{lemma}
\begin{proof} Let us reformulate the fluid equation in \eqref{main} as the equation for vorticity $\omega$:
\[
\pa_t \omega + (u\cdot \nabla_x) \omega- \Delta_x \omega = -\nabla_x \times \int_{\R^2} (u-v)f\,dv = -\nabla_x \times (\rho_f (u-u_f)).
\]
Multiplying the above equation by $\omega$ and integrating the resulting equation over $\T^2$ gives
\[
\frac12\frac{d}{dt} \|\omega\|_{L^2}^2 + \|\nabla_x \omega\|_{L^2}^2 \leq \int_{\T^3} |\nabla_x \omega| |\rho_f (u-u_f) |\, dx \leq  \frac12\|\nabla_x \omega\|_{L^2}^2 + \frac12 \|\rho_f (u-u_f)\|_{L^2}^2,
\]
i.e.,
\bq\label{est_vor}
\frac{d}{dt} \|\omega\|_{L^2}^2 + \|\nabla_x \omega\|_{L^2}^2 \leq \|\rho_f (u-u_f)\|_{L^2}^2.
\eq
This together with Lemma \ref{lem_drag} asserts
\[
\sup_{t \geq 0}\|\omega(\cdot,t)\|_{L^2}^2 + \int_0^\infty \|\nabla_x \omega(\cdot,\tau)\|_{L^2}^2\,d\tau \leq C,
\]
where $C > 0$ is independent of $t$. On the other hand, the following inequality holds
\[
\|u\|_{L^2(0,t;H^2)}^2 \leq C\|\omega\|_{L^2(0,t;H^1)}^2 + \|u\|_{L^2(0,t;L^2)}^2 \leq C(1 + t).
\]
We further estimate, due to Lemma \ref{lem_drag},
\begin{align*}
&\int_{\T^2} |\pa_t u|^2\,dx + \frac{d}{dt}\int_{\T^2}|\nabla_x u|^2\,dx \cr
&\quad \leq \int_{\T^2} |u|^2 |\nabla_x u|^2\,dx + \int_{\T^2} |\rho_f (u - u_f)|^2\,dx \leq \|u\|_{L^\infty}^2\|\nabla_x u\|_{L^2}^2 + Ce^{-Ct}
\leq C\|\nabla_x u\|_{L^2}^2 + Ce^{-Ct}.
\end{align*}
Integrating the above inequality with respect to time, we get
\[
\sup_{t \geq 0}\|\nabla_x u(\cdot,t)\|_{L^2}^2 + \int_0^\infty \|\pa_t u(\cdot,\tau)\|_{L^2}^2\,d\tau \leq C,
\]
where $C > 0$ is independent of $t$. If we apply the Poincar\'e inequality to \eqref{est_vor} and use Lemma \ref{lem_drag}, then we have
\[
\frac{d}{dt} \|\omega(\cdot,t)\|_{L^2}^2 + C\|\omega(\cdot,t)\|_{L^2}^2  \leq Ce^{-Ct},
\]
which implies that
\[
\|\omega(\cdot,t)\|_{L^2} \leq Ce^{-Ct},
\]
where $C>0$ is independent of $t$. This completes the proof.
\end{proof}

Next, we present upper bounds of estimates of $f$ in Sobolev spaces.

%For notational simplicity, let us use the notation $L^{p,q}_{x,t} := L^q(0,t;L^p(\T^2))$.

\begin{proposition}\label{prop_fw1}Let $d=2$. Suppose tha $(f,u)$ is the strong solution  defined in Definition \ref{strong-sol} to the system \eqref{main}-\eqref{initial}. Under the same assumptions as in Theorem  \ref{main-thm},  we have
\[
\|f(\cdot,\cdot,t)\|_{W^{1,q}} \leq \|f_0\|_{W^{1,q}} e^{C(1+t)}
\]
and
\[
\|f(\cdot,\cdot,t)\|_{W^{2,q}} \leq \|f_0\|_{W^{2,q}} e^{C(1+t)}
\]
for $ q>d$, where $C>0$ is independent of $t$.
\end{proposition}

\begin{proof} We recall from Lemma \ref{lem_f} that
\[
\|f(\cdot,\cdot,t)\|_{W^{1,q}} \leq \|f_0\|_{W^{1,q}} \exp\lt( Ct + C\int_0^t \|\nabla_x u(\cdot,\tau)\|_{L^\infty}\,d\tau\rt).
\]
In the rest of the proof, we will show that
\bq\label{est_ul}
\int_0^t \|\nabla_x u(\cdot,\tau)\|_{L^\infty}^2\,d\tau \leq C(1 + t^{\frac{1}{3}}) \leq C(1 + t),
\eq
which leads to
\begin{equation}\label{nabla-u-2D}
\int_0^t \|\nabla_x u(\cdot,\tau)\|_{L^\infty}\,d\tau\le \sqrt{t}\norm{u}_{L^2(0,t;L^\infty)}\le C(1+t).
\end{equation}

For this, we use the Stokes estimate of Giga-Sohr \cite{GS91} to get
\bq\label{est_ud2}
\|\pa_t u\|_{L^2(0,t;L^3)} + \|\Delta_x u\|_{L^2(0,t;L^3)} \leq \|u_0\|_{D_3^{1/2,2}} + C\|(u \cdot \nabla_x)u\|_{L^2(0,t;L^3)} + C\|\rho_f(u - u_f)\|_{L^2(0,t;L^3)},
\eq
where we can estimate
\[
\|u_0\|_{D_3^{1/2,2}} \leq C\|u_0\|_{\mathcal{V} \cap H^{4/3}}.
\]
Here $D^{1/2,2}_3$ denotes the Besov-type space given by $D^{1/2,2}_3 = B^1_{3,2}(\T^2) \cap L^3_\sigma(\T^2)$ and we used the standard embedding theorem, $\mathcal{V} \cap H^{4/3} \subset D^{1/2,2}_3$. For the estimate of the second term on the right hand side of \eqref{est_ud2}, we use the Gagliardo-Nirenberg-Sobolev interpolation inequality to find
\begin{align*}
\|(u \cdot \nabla_x)u\|_{L^3} &\leq \|u\|_{L^\infty}\|\nabla_x u\|_{L^3} \leq C\|\nabla_x u\|_{L^2}^{2/3}\|\Delta_x u\|_{L^2}^{1/3} + C\|\nabla_x u\|_{L^2}.
\end{align*}
This with Lemma \ref{lem_regu} and total energy estimate in Lemma \ref{lem_energy} gives
\begin{align*}
\|(u \cdot \nabla_x)u\|_{L^2(0,t;L^3)}^2  &\leq C\int_0^t \|\nabla_x u\|_{L^2}^{4/3}\|\Delta_x u\|_{L^2}^{2/3}\,d\tau + C\int_0^t \|\nabla_x u\|_{L^2}^2\,d\tau\cr
&\leq C\lt(\int_0^t \|\nabla_x u\|_{L^2}^2\,d\tau\rt)^{2/3}\lt(\int_0^t \|\Delta_x u\|_{L^2}^2\,d\tau \rt)^{1/3} + C\int_0^t \|\nabla_x u\|_{L^2}^2\,d\tau\cr
&\leq C\lt(1 + t\rt).
\end{align*}
Moreover, we find from Lemma \ref{lem_drag} that
\[
\|\rho_f(u - u_f)\|_{L^2(0,t;L^3)} \leq C,
\]
where $C>0$ is independent of $t$. This asserts
\[
\|\Delta_x u\|_{L^2(0,t;L^3)}^2 \leq C\lt(1 + t\rt).
\]
We then use the following Sobolev interpolation inequality
\[
\|\nabla_x u\|_{L^\infty}^2 \leq C\|\Delta_x u\|_{L^3}^{3/2}\|\nabla_x u\|_{L^2}^{1/2} +C\|\nabla_x u\|_{L^2}^2\leq C\|\Delta_x u\|_{L^3}^2 + C\|\nabla_x u\|_{L^2}^2
\]
to have
\[
\int_0^t \|\nabla_x u\|_{L^\infty}^2\,d\tau \leq C\int_0^t \|\Delta_x u\|_{L^3}^2\,d\tau + C\int_0^t \|\nabla_x u\|_{L^2}^2\,d\tau \leq C\lt(1 + t\rt).
\]
On the other hand, it follows from Lemma \ref{lem_f2}  that
\[
\|f(\cdot,\cdot,t)\|_{W^{2,q}} \leq \|f_0\|_{W^{2,q}}  \exp\lt( Ct + \int_0^t \|\nabla_x^2 u(\cdot,\tau)\|_{L^{q}}\,d\tau  \rt)
\]
for $q > d$, where $C>0$ is independent of $t$. Thus it suffices to show that
\[
\int_0^t \|\nabla_x^2 u(\cdot,\tau)\|_{L^{q}}^2\,d\tau \leq C(1 + t)
\]
for some $C>0$ which is independent of $t$. Similarly as before, we estimate
\[
\|\nabla_x^2 u\|_{L^2(0,t;L^{q})} \leq \|u_0\|_{D^{1/2, 2}_{q}} + C\|(u \cdot \nabla_x)u\|_{L^2(0,t;L^{q})} + C\|\rho_f(u - u_f)\|_{L^2(0,t;L^{q})},
\]
where $D^{1/2,2}_q= B^1_{q,2}(\T^2) \cap L^q_\sigma(\T^2)$.
Since
$\|\rho_f(u - u_f)\|_{L^2(0,t;L^{q})} \leq C$,
we show, due to Lemma \ref{lem_drag},  that
\[
\|(u \cdot \nabla_x)u\|_{L^2(0,t;L^{q})}  \leq C(1 + t).
\]
We use the Gagliardo-Nirenberg-Sobolev interpolation inequality to find
\[
\|(u \cdot \nabla_x)u\|_{L^{q}} \leq \|u\|_{L^\infty}\|\nabla_x u\|_{L^{q}} \leq C\lt(\|\nabla_x u\|_{L^2}^{1-\alpha}\|\nabla_x^2 u\|_{L^2}^\alpha + C\|\nabla_x u\|_{L^2} \rt),
\]
where $\alpha = (1/2)(1 - 1/q)$. Thus we have
\begin{align*}
\|(u \cdot \nabla_x)u\|_{L^2(0,t;L^{q})}^2  &\leq C\int_0^t \lt(\|\nabla_x u\|_{L^2}^{2(1-\alpha)}\|\nabla_x^2 u\|_{L^2}^{2\alpha} + C\|\nabla_x u\|_{L^2}^2 \rt) d\tau\cr
&\leq C\int_0^t \lt(\|\nabla_x u\|_{L^2}^2 + \|\nabla_x^2 u\|_{L^2}^2 \rt) d\tau\cr
&\leq C(1 + t).
\end{align*}
This completes the proof.
\end{proof}

%%%%%%%%%%%%%%%%%%%%%%%%%%%%%%%%%%%%%%%%%%%%%%%%
%
%
%
%
%
% \subsection{Three dimensional case}
%
%
%
%
%%%%%%%%%%%%%%%%%%%%%%%%%%%%%%%%%%%%%%%%%%%%%%%%

\subsubsection{Three dimensional case}
In this part, we consider the three dimensional case, and in this case, the fluid system in \eqref{main} becomes the Stokes system.

\begin{proposition}\label{prop_fw1_d3}Let $d=3$. Suppose that $(f,u)$ is the strong solution  defined in Definition \ref{strong-sol} to the system \eqref{main}-\eqref{initial}. Under the same assumptions as in Theorem  \ref{main-thm},  we have
\[
\|f(\cdot,\cdot,t)\|_{W^{1,q}} \leq \|f_0\|_{W^{1,q}} \,e^{C(1+t)}
\]
and
\[
\|f(\cdot,\cdot,t)\|_{W^{2,q}} \leq \|f_0\|_{W^{2,q}} \,e^{C(1+t)}
\]
for $q > 3$, where $C>0$ is independent of $t$.
\end{proposition}
\begin{proof}
Similarly as in the previous subsection, we consider the vorticity equation:
\[
\pa_t \omega  - \Delta_x \omega  = -\nabla_x \times (\rho_f (u-u_f)).
\]
Then by the standard maximal regularity estimate for the heat equation, we have
\[
\|\nabla_x \omega\|_{L^2(0,t;L^4)} \leq C\|\nabla_x \omega_0\|_{L^4} + C\|\rho_f (u-u_f)\|_{L^2(0,t;L^4)},
\]
where $C>0$ is independent of $t$. Moreover, by Lemma \ref{lem_drag}, we obtain
\[
\|\nabla_x \omega\|_{L^2(0,t;L^4)} \leq C\|\nabla_x \omega_0\|_{L^4} + C,
\]
where $C>0$ is independent of $t$. Then this together with Gagliardo-Nirenberg-Sobolev interpolation inequality and the Calderon-Zygmund type inequality \cite[p. 408]{MB02} implies
\begin{align}\label{april06-10}
\int_0^t \|\nabla_x u\|_{L^\infty}^2\,d\tau &\leq C\int_0^t \|\nabla_x^2 u\|_{L^4}^{8/5}\|u\|_{L^\infty}^{2/5}\,d\tau + C\int_0^t \|u\|_{L^2}^2\,d\tau \cr
&\leq C\int_0^t \|\nabla_x^2 u\|_{L^4}^2\,d\tau + Ct \cr
&\leq C\int_0^t \|\nabla_x \omega\|_{L^4}^2\,d\tau + Ct \cr
&\leq Ct,
\end{align}
where $C>0$ is independent of $t$.
On the other hand, using Lemma \ref{lem_f2} and \eqref{april06-10}, we have
\[
\|f(\cdot,\cdot,t)\|_{W^{2,q}} \leq \|f_0\|_{W^{2,q}}  \exp\lt( Ct + \int_0^t \|\nabla_x^2 u(\cdot,\tau)\|_{L^{q}}\,d\tau  \rt)
\]
for $q > d$, where $C>0$ is independent of $t$. Thus it suffices to show that
\[
\int_0^t \|\nabla_x^2 u(\cdot,\tau)\|_{L^{q}}\,d\tau \leq C(1 + t),
\]
for some $C>0$ which is independent of $t$. Similarly as before, we estimate
\begin{align}\label{nabla-u-3D}
\begin{aligned}
\int_0^t \|\nabla_x^2 u(\cdot,\tau)\|_{L^{q}}\,d\tau &\leq C\int_0^t \|\nabla_x \omega(\cdot,\tau)\|_{L^{q}}+ \|\nabla_x u(\cdot,\tau)\|_{L^{q}}\,d\tau \cr
&\leq C\|\omega_0\|_{L^{q}}t + C\|\rho_f (u-u_f)\|_{L^1(0,t;L^{q})} \cr
&\leq C(1 + t),
\end{aligned}
\end{align}
due to Lemma \ref{lem_drag}, where $C>0$ is independent of $t$. This completes the proof.
\end{proof}

\begin{remark}
Recalling  \eqref{nabla-u-2D} and \eqref{nabla-u-3D} for $q>3$,
it is worth noting that $\nabla_x u\in L^1_{\rm loc}(\R_+; L^{\infty}(\T^d))$, $d=2,3$ and, furthermore,
\begin{equation}\label{nabla-u-est}
\int_0^t \|\nabla_x u(\cdot,\tau)\|_{L^{\infty}(\T^d)} \,d\tau \le C(1 + t),\qquad d=2,3.
\end{equation}
\end{remark}

%%%%%%%%%%%%%%%%%%%%%%%%%%%%%%%%%%%%%%%%%%%%%%%%
%
%
%
%
%
% \subsection{Three dimensional case}
%
%
%
%
%%%%%%%%%%%%%%%%%%%%%%%%%%%%%%%%%%%%%%%%%%%%%%%%

\subsubsection{Growth estimates of $\nabla_x \omega$ and $\nabla_x^2 u$}

\begin{lemma}\label{omega-u-estimate}
Let $d=2,3$. Suppose that $(f,u)$ is the strong solution  defined in Definition \ref{strong-sol} to the system \eqref{main}-\eqref{initial}.
Under the same assumptions as in Theorem  \ref{main-thm},  we have
\[
\|\nabla_x \omega(\cdot, t)\|_{L^{\infty}} +\|\nabla_x^2 u(\cdot, t)\|_{L^{p}}
\leq Ce^{C(1+t)}
\]
for $p < \infty$, where $C > 0$ is independent of $t$.
\end{lemma}
\begin{proof}
Using the heat kernel, we represent $\omega$ as
\begin{align*}
\omega &= \Gamma \star \omega_0 - (1 - \delta_{d,3})\int_0^t \int_{\T^d} \Gamma(x-y,t-s)\lt((u \cdot \nabla_x \omega)(y,s)\rt) dyds\cr
&\quad - \int_0^t \int_{\T^d} \Gamma(x-y,t-s)\nabla_x \times \lt(\rho_f(y,s)(u - u_f)(y,s)\rt) dyds.
\end{align*}
The derivative estimate of $\omega$ becomes
\begin{align*}
&\|\nabla_x \omega(\cdot,t)\|_{L^\infty}\cr
&\quad  \leq \|\Gamma \star \nabla_x\omega_0\|_{L^\infty}\cr
&\qquad + \lt\| \int_0^t \int_{\T^d} \nabla_x \Gamma(x-y,t-s)\lt((1 - \delta_{d,3})(u \cdot \nabla_x \omega)(y,s) + \nabla_x \times \lt(\rho_f(y,s)(u - u_f)(y,s)\rt)\rt) dyds \rt\|_{L^\infty}\cr
&\quad \le \|\nabla_x \omega_0\|_{L^\infty}+C\int_0^t (t-s)^{-\frac{d}{4q} - \frac12} \lt( (1 - \delta_{d,3})\|(u \cdot \nabla_x \omega)(\cdot,s)\|_{L^{q}} + \| \nabla_x \times \lt(\rho_f(u - u_f)(\cdot,s)\rt)\|_{L^{q}} \rt)ds.
\end{align*}
We note that there exists $C>0$, which is independent of $t$, such that
\[
\|(u \cdot \nabla_x \omega)\|_{L^{q}} \leq \|u\|_{L^\infty}\|\nabla_x \omega\|_{L^{q}} \leq C\|\nabla_x \omega\|_{L^{q}}
\]
and
\begin{align}\label{est_dr2}
\begin{aligned}
\| \nabla_x \times(\rho_f(u - u_f))\|_{L^{q}} &\leq C\|u\nabla_x \rho_f \|_{L^{q}} + C\|\rho_f \omega\|_{L^{q}} + C\|\nabla_x (\rho_f u_f)\|_{L^{q}}\cr
&\leq C\|u\|_{L^\infty}\|\nabla_x \rho_f\|_{L^{q}} + C\|\rho_f\|_{L^\infty}\|\omega\|_{L^{q}} + C\|\nabla_x (\rho_f u_f)\|_{L^{q}}\cr
&\leq C\|\nabla_x f\|_{L^{q}} + C\|\omega\|_{L^{q}},
\end{aligned}
\end{align}
where we used the support estimate of $f$ in velocity. Therefore,
\begin{align*}
\|\nabla_x\omega(\cdot,t)\|_{L^\infty} &\leq \|\nabla_x \omega_0\|_{L^\infty}\cr
&\quad +C\int_0^t (t-s)^{-\frac{d}{4q} - \frac12} \lt( (1 - \delta_{d,3})\|\nabla_x \omega(\cdot,s)\|_{L^{q}} +C\|\nabla_x f(\cdot,s)\|_{L^{q}} + C\|\omega(\cdot,s)\|_{L^{q}}\rt)ds.
\end{align*}
Since $q > d$, we note that $1 - \frac1q\lt(\frac d2 + 1 \rt) > 0$. We recall from
Propositions \ref{prop_fw1} and \ref{prop_fw1_d3} that
\[
\|\nabla_x f\|_{L^{q}(0,t;L^{q})} \leq Ce^{Ct}.
\]
Next, we estimate $\nabla_x^k \omega$
 for $k=0,1$
\begin{align*}
\frac{d}{dt} \|\nabla_x^k \omega\|_{L^{q}}^{q} &= q\intt |\nabla_x^k \omega|^{q-2}\nabla_x^k \omega \pa_t \nabla_x^k \omega\,dx\cr
&\leq C_q(1 - \delta_{d,3}) \|\nabla_x u\|_{L^\infty}\|\nabla_x^k \omega\|_{L^{q}}^{q} - q(q-1)\intt  |\nabla_x^k \omega|^{q-2}|\nabla_x^{k+1} \omega|^2\,dx \cr
&\quad + q(q-1)\intt  |\nabla_x^k \omega|^{q-2}|\nabla_x^{k+1} \omega| |\nabla_x \times (\rho_f (u-u_f))|\,dx,
\end{align*}
where the last term on the right hand side of the above inequality can be bounded by
\begin{align*}
&q(q-1) \lt(\intt |\nabla_x^k \omega|^{q}\,dx \rt)^{\frac{q-2}{2q}}\lt( \intt  |\nabla_x^k \omega|^{q-2}|\nabla_x^{k+1} \omega|^2\,dx\rt)^{\frac12}\lt(\intt  |\nabla_x \times (\rho_f (u-u_f))|^{q}\,dx\rt)^{\frac{1}{q}}\cr
&\quad \leq q(q-1)\intt  |\nabla_x^k \omega|^{q-2}|\nabla_x^{k+1} \omega|^2\,dx + q(q-1)\|\nabla_x^k \omega\|_{L^{q}}^{q-2}\|\nabla_x \times (\rho_f (u-u_f))\|_{L^{q}}^2\cr
&\quad \leq q(q-1)\intt  |\nabla_x^k \omega|^{q-2}|\nabla_x^{k+1} \omega|^2\,dx + C_q\|\nabla_x^k \omega\|_{L^{q}}^{q} + C_q\|\nabla_x f\|_{L^{q}}^{q} + C\|\omega\|_{L^{q}}^{q},
\end{align*}
due to \eqref{est_dr2}.
Thus we obtain
\[
\frac{d}{dt} \|\omega\|_{W^{1,q}}^{q} \leq C\lt((1 - \delta_{d,3}) \|\nabla_x u\|_{L^\infty}+ 1 \rt)\|\omega\|_{W^{1,q}}^{q} + C\|\nabla_x f\|_{L^{q}}^{q}.
\]
Since
\[
\int_0^t \|\nabla_x u\|_{L^\infty}\,d\tau \leq C(1+t),
\]
we use Gr\"onwall's lemma to have
\begin{align*}
\|\omega(\cdot,t)\|_{W^{1,q}}^{q} &\leq \|\omega_0\|_{W^{1,q}}^{q}e^{C(1+t)} + Ce^{C(1+t)}\int_0^t \|\nabla_x f(\cdot,\tau)\|_{L^{q}}^{q}\,d\tau \leq Ce^{C(1+t)},
\end{align*}
where $C>0$ is independent of $t$. This yields
\[
\|\nabla_x \omega(\cdot,t)\|_{L^\infty}  \leq Ce^{C(1+t)}.
\]
It is also immediate due to Korn's inequality that
\[
\|\nabla_x^2 u(\cdot,t)\|_{L^p} \leq Ce^{C(1+t)}
\]
for any $1\le p<\infty$. This completes the proof.
\end{proof}

%%%%%%%%%%%%%%%%%%%%%%%%%%%%%%%%%%%%%%%%%%%%%%%%
%
%
%
%
%
% \subsection{Three dimensional case}
%
%
%
%
%%%%%%%%%%%%%%%%%%%%%%%%%%%%%%%%%%%%%%%%%%%%%%%%

\subsubsection{Exponential decay estimates}
In the subsection, we prove exponential decay of generalized moments for derivatives of $f$ up to second order.

First, we show that the first order derivatives of $f$ with power-law type of moments is of an exponential decay in time,  which is a part of Theorem \ref{main-thm} for the case that $k=1$.

\begin{proposition}\label{thm_decay_w1}Let $d=2$ or $3$. Suppose that $(f,u)$ is the strong solution  defined in Definition \ref{strong-sol} to the system \eqref{main}-\eqref{initial}. Under the same assumptions as in Theorem  \ref{main-thm},  there exists $p=p(q, d, \psi, f_0, u_0) > 2$ large enough such that
\[
\inttr | v - v_c(t)|^{p} \lt( |(\nabla_x f)(x,v,t)|^{q} + |(\nabla_v f)(x,v,t)|^{q}\rt)dxdv \leq Ce^{-Ct},
\]
where $C > 0$ is independent of $p$ and $t$.
\end{proposition}
\begin{proof}
Straightforward computations give
\begin{align*}
&\frac{1}{q}\frac{d}{dt} \inttr | v - v_c|^{p} |\pa_{v_j}f|^{q}\,dxdv\cr
&\quad = -\frac pq\inttr |v - v_c|^{p-2}(v - v_c) \cdot v_c'  |\pa_{v_j}f|^{q}\,dxdv + \inttr |v - v_c|^{p}  |\pa_{v_j}f|^{q-2} \pa_{v_j} f \pa_t \pa_{v_j}f\,dxdv
%\cr
%&\quad 
=: J_1 + J_2.
\end{align*}
First we  estimate $J_1$ as
\begin{align*}
J_1 &\leq \frac pq\inttr |v - v_c|^{p-1}|v_c'|  |\pa_{v_j}f|^{q}\,dxdv \cr
&\leq \frac{p}{4q}\inttr |v - v_c|^{p} |\pa_{v_j}f|^{q}\,dxdv + C_{p,q} e^{-Cpt}\inttr  |\pa_{v_j}f|^{q}\,dxdv.
\end{align*}
On the other hand, we split $J_2$ into two terms, i.e.
\begin{align*}
J_2 &= -\inttr |v - v_c|^{p}  |\pa_{v_j}f|^{q-2} \pa_{v_j} f \lt(\pa_{x_j} f + \nabla_v \cdot \lt((\pa_{v_j}F)f + F\pa_{v_j}f\rt)\rt)dxdv
%\cr
%& 
=: J_2^1 + J_2^2.
\end{align*}
Here we estimate $J_2^1$ as
\begin{align*}
J_2^1 &\leq \inttr |v - v_c|^{p}  |\pa_{v_j}f|^{q-1}|\pa_{x_j} f |\,dxdv \cr
&\leq \frac{q-1}{q}\inttr |v - v_c|^{p} |\pa_{v_j}f|^{q}\,dxdv + \frac{1}{q}\inttr |v - v_c|^{p} |\pa_{x_j}f|^{q}\,dxdv.
\end{align*}
For $J_2^2$, we use $\nabla_v \cdot \pa_{v_j}F = 0$ to get
\begin{align*}
J_2^2&= -\inttr |v - v_c|^{p}  |\pa_{v_j}f|^{q-2} \pa_{v_j} f \lt( (\pa_{v_j}F)\cdot \nabla_v f + (\nabla_v \cdot F)\pa_{v_j}f + F \cdot \nabla_v \pa_{v_j}f\rt)dxdv\cr
&= (d+1)\inttr (1 + \psi \star \rho_f)  |v - v_c|^{p}  |\pa_{v_j}f|^{q} \,dxdv -\inttr |v - v_c|^{p}  |\pa_{v_j}f|^{q-2} \pa_{v_j} f (F \cdot \nabla_v \pa_{v_j}f)\,dxdv.
\end{align*}
On the other hand, we use the integration by parts to estimate the second term on the right hand side of the above equality as
\begin{align*}
&-\frac{1}{q}\inttr |v - v_c|^{p}F \cdot  \nabla_v |\pa_{v_j}f|^{q}\,dxdv\cr
&\quad =\frac{1}{q} \inttr \lt(p|v-v_c|^{p-2}(v - v_c) \cdot F + |v - v_c|^{p} \nabla_v \cdot F \rt) |\pa_{v_j}f|^{q}\,dxdv\cr
&\quad \leq -\frac{p\psi_m M_0}{q} \inttr |v - v_c|^{p}|\pa_{v_j}f|^{q}\,dxdv + C_{p,q} e^{-Cp t}\inttr |\pa_{v_j}f|^{q}\,dxdv\cr
&\qquad -\frac {p(1 - C\delta)}{q} \inttr |v - v_c|^{p}|\pa_{v_j}f|^{q}\,dxdv + C_{p,q}\|u - u_c\|_{L^\infty}^{p} \inttr |\pa_{v_j}f|^{q}\,dxdv\cr
&\qquad - \frac{d}{q}\inttr (1 + \psi \star \rho_f)  |v - v_c|^{p}  |\pa_{v_j}f|^{q} \,dxdv,
\end{align*}
where we also used some estimates in the proof of Proposition \ref{thm_main} and $\delta>0$ will be determined later. Summing up above estimates, we obtain
\[
J_2 \leq \lt(d+1 - \frac{d}{q} -\frac{p\psi_m M_0}{q}-\frac {p(1 - C\delta)}{q}\rt) \inttr |v - v_c(t)|^{p}|\pa_{v_j}f|^{q}\,dxdv + C_{p,q} e^{-Cp t}\inttr |\pa_{v_j}f|^{q}\,dxdv,
\]
and as a result, we have
\begin{align*}
&\frac{d}{dt} \inttr | v - v_c|^{p} |\pa_{v_j}f|^{q}\,dxdv \cr
&\quad  \leq \lt(q(d+1) + q-1- d - p(1 + \psi_m M_0- C\delta)\rt) \inttr |v - v_c|^{p}|\pa_{v_j}f|^{q}\,dxdv\cr
&\qquad  + C_{p,q} e^{-Cp t}\inttr |\pa_{v_j}f|^{q}\,dxdv + \inttr |v - v_c|^{p} |\pa_{x_j}f|^{q}\,dxdv.
\end{align*}
In order to control the last term on the right hand side of the above inequality, we estimate
\begin{align*}
&\frac{1}{q}\frac{d}{dt} \inttr | v - v_c|^{p} |\pa_{x_j}f|^{q}\,dxdv\cr
&\quad = -\frac pq\inttr |v - v_c|^{p-2}(v - v_c) \cdot v_c'  |\pa_{x_j}f|^{q}\,dxdv + \inttr |v - v_c|^{p}  |\pa_{x_j}f|^{q-2} \pa_{x_j} f \pa_t \pa_{x_j}f\,dxdv\cr
&\quad =: K_1 + K_2.
\end{align*}
Then similarly as before, we get
\[
K_1 \leq \frac{p}{4q}\inttr |v - v_c|^{p} |\pa_{x_j}f|^{q}\,dxdv + C_{p,q} e^{-Cpt}\inttr  |\pa_{x_j}f|^{q}\,dxdv.
\]
The second term $K_2$ is rewritten as
\begin{align*}
K_2 &= -\inttr |v - v_c|^{p}  |\pa_{x_j}f|^{q-2} \pa_{x_j} f \lt( \nabla_v \cdot \lt((\pa_{x_j}F)f + F\pa_{x_j}f\rt)\rt)dxdv
%\cr
%& 
=: K_2^1 + K_2^2,
\end{align*}
where $K_2^1$ can be estimated as
\begin{align*}
K_2^1 &=  -\inttr |v - v_c|^{p}  |\pa_{x_j}f|^{q-2} \pa_{x_j} f \lt( (\nabla_v \cdot \pa_{x_j}F)f + (\pa_{x_j}F)\cdot \nabla_v f\rt) dxdv \cr
&\leq d\|\pa_{x_j}\psi\|_{L^\infty} \inttr |v - v_c|^{p}  |\pa_{x_j}f|^{q-1} f\,dxdv \cr
&\quad + C(1 + \|\pa_{x_j}u\|_{L^\infty}) \inttr |v - v_c|^{p}  |\pa_{x_j}f|^{q-1} |\nabla_v f|\,dxdv\cr
&\leq d\|\pa_{x_j}\psi\|_{L^\infty}\lt(\frac{q-1}{q}\inttr |v - v_c|^{p}  |\pa_{x_j}f|^{q}\,dxdv + \frac{1}{q}\inttr |v - v_c|^{p}  f^{q}\,dxdv\rt)\cr
&\quad + C(1 + \|\pa_{x_j}u\|_{L^\infty})\lt(\frac{q-1}{q}\inttr |v - v_c|^{p}  |\pa_{x_j}f|^{q}\,dxdv + \frac{1}{q}\inttr |v - v_c|^{p}  |\nabla_v f|^{q}\,dxdv\rt).
\end{align*}
We estimate $K_2^2$ as follows:
\begin{align*}
K_2^2 &= -\inttr |v - v_c|^{p}  |\pa_{x_j}f|^{q-2} \pa_{x_j} f \lt( (\nabla_v \cdot F)\pa_{x_j}f + F \cdot \nabla_v \pa_{x_j} f\rt)dxdv\cr
&=d\inttr (1 + \psi \star \rho_f)|v - v_c|^{p}  |\pa_{x_j}f|^{q}\,dxdv 
\cr
&\quad 
+ \frac{1}{q} \inttr \lt(p|v-v_c|^{p-2}(v - v_c) \cdot F + |v - v_c|^{p} \nabla_v \cdot F \rt) |\pa_{x_j}f|^{q}\,dxdv\cr
&= \frac{p}{q} \inttr |v-v_c|^{p-2}(v - v_c) \cdot F |\pa_{x_j}f|^{q}\,dxdv+ d\lt(1 - \frac{1}{q} \rt)\inttr (1 + \psi \star \rho_f)|v - v_c|^{p}  |\pa_{x_j}f|^{q}\,dxdv\cr
&\leq   -\lt(\frac{p\psi_m M_0}{q} + \frac {p(1 - C\delta)}{q}\rt) \inttr |v - v_c|^{p}|\pa_{x_j}f|^{q}\,dxdv + C_{p,q} e^{-Cp t}\inttr |\pa_{x_j}f|^{q}\,dxdv\cr
&\quad + d\lt(1 - \frac{1}{q} \rt)(1 + \psi_M)\inttr |v - v_c|^{p}  |\pa_{x_j}f|^{q}\,dxdv.
\end{align*}
Combining all of the above observations yields
\begin{align*}
&\frac{d}{dt} \inttr | v - v_c|^{p} |\pa_{x_j}f|^{q}\,dxdv \cr
&\quad \leq (C_q - Cp)\inttr | v - v_c|^{p} |\pa_{x_j}f|^{q}\,dxdv + C_{p,q} e^{-Cp t}\inttr |\pa_{x_j}f|^{q}\,dxdv + C_{p,q} e^{-Cp t} \cr
&\qquad + C\|\pa_{x_j}u\|_{L^\infty}^2\lt( \inttr |v - v_c|^{p}  |\pa_{x_j}f|^{q}\,dxdv + \inttr |v - v_c|^{p}  |\nabla_v f|^{q}\,dxdv\rt),
\end{align*}
where $C>0$ is independent of both $p$ and $t$ and $\delta>0$ is chosen small enough so that $1 + \psi_m - C\delta>0$. Thus we have
\begin{align*}
&\frac{d}{dt}\inttr | v - v_c|^{p} \lt( |\nabla_x f|^{q} + |\nabla_v f|^{q}\rt)dxdv \cr
&\quad \leq -C_0p\inttr | v - v_c|^{p} \lt( |\nabla_x f|^{q} + |\nabla_v f|^{q}\rt)dxdv + C_{p,q} e^{-Cp t}  \cr
&\qquad + C_{p,q} e^{-Cp t}\inttr \lt( |\nabla_x f|^{q} + |\nabla_v f|^{q}\rt)dxdv + \|\nabla_x u\|_{L^\infty}^2\inttr | v - v_c|^{p} \lt( |\nabla_x f|^{q} + |\nabla_v f|^{q}\rt)dxdv.
\end{align*}
We then use Proposition \ref{prop_fw1} to estimate
\[
 C_{p,q} e^{-Cp t}\inttr \lt( |\nabla_x f|^{q} + |\nabla_v f|^{q}\rt)dxdv \leq  C_{p,q} e^{-Cp t} e^{C(1+t)} \leq C_{p,q} e^{-C_1 p t}
\]
for $p > 0$ large enough. For convenience, we set
\[
\mathcal{F}(t) := \lt(\inttr | v - v_c|^{p} \lt( |\nabla_x f|^{q} + |\nabla_v f|^{q}\rt)dxdv\rt) \exp\lt(-\int_0^t  \|\nabla_x u\|_{L^\infty}^2\,d\tau\rt),
\]
then $\mathcal{F}$ satisfies
\[
\frac{d}{dt}\mathcal{F} \leq C_{p,q} e^{-C_1p t}\exp\lt(-\int_0^t  \|\nabla_x u\|_{L^\infty}^2\,d\tau\rt)  - C_0p\mathcal{F}.
\]
Applying the Gr\"onwall's lemma, we get
\[
\mathcal{F}(t) \leq \mathcal{F}_0 e^{-C_0 p t} + C_{p,q} e^{-C_0 p t} \int_0^t e^{(C_0 - C_1)ps} \exp\lt(-\int_0^s  \|\nabla_x u\|_{L^\infty}^2\,d\tau\rt)ds,
\]
and this together with \eqref{est_ul} asserts
\begin{align*}
&\inttr | v - v_c|^{p} \lt( |\nabla_x f|^{q} + |\nabla_v f|^{q}\rt)dxdv\cr
&\quad \leq \lt(\inttr | v - v_c|^{p} \lt( |\nabla_x f_0|^{q} + |\nabla_v f_0|^{q}\rt)dxdv\rt)\exp\lt(\int_0^t  \|\nabla_x u\|_{L^\infty}^2\,d\tau\rt)e^{-C_0pt}\cr
&\qquad + C_{p,q} e^{-C_0 p t} \int_0^t e^{(C_0 - C_1)ps} \exp\lt(\int_s^t  \|\nabla_x u\|_{L^\infty}^2\,d\tau\rt)ds\cr
&\quad \leq \lt(\inttr | v - v_c|^{p} \lt( |\nabla_x f_0|^{q} + |\nabla_v f_0|^{q}\rt)dxdv\rt) e^{-(C_0p - C)t} +C_{p,q} e^{-(C_0 p -C)t}\int_0^t e^{(C_0 - C_1)ps} \,ds \cr
&\quad \leq \lt(\inttr | v - v_c|^{p} \lt( |\nabla_x f_0|^{q} + |\nabla_v f_0|^{q}\rt)dxdv\rt) e^{-(C_0p - C)t} +  C\lt(e^{-(C_0p - C)t}  + e^{-(C_1 p - C)t}\rt),
\end{align*}
where the constants $C_0, C_1$, and $C$ are positive and independent of $t$. This completes the proof.
\end{proof}

Now, we are ready to present the proof of Theorem \ref{main-thm}. \newline

\begin{mainthm}
Due to results of Proposition \ref{thm_main} and Propsotion \ref{thm_decay_w1}, it suffices to show that there exists $p=p(q,  d, \psi, f_0, u_0) > 2$ large enough such that
\[
\inttr | v - v_c(t)|^{p} \lt( |(\nabla_x^2 f)(x,v,t)|^{q} + |(\nabla_x \nabla_v f)(x,v,t)|^{q}+ |(\nabla_v^2 f)(x,v,t)|^{q}\rt)dxdv \leq Ce^{-Ct},
\]
where $C > 0$ is independent of $p$ and $t$.

Since the proof is rather lengthy, we divide it into four steps:
\begin{itemize}
\item In Step A, we estimate the generalized $p$-th moments of the second order derivative of $f$ in spatial variable and derive
\begin{align*}
&\frac{d}{dt} \inttr |v - v_c|^{p} |\nabla_x^2 f|^{q}\,dxdv\cr
&\quad \leq (C+\|\nabla_x u\|_{L^\infty}) \inttr |v - v_c|^{p}\lt( |\nabla_x^2 f|^{q} +  |\nabla_v\nabla_x f|^{q} \rt)dxdv 
\cr
&\qquad 
+ C\inttr |v - v_c|^{p} \lt(|f|^{q} + |\nabla_x f|^{q}\rt)dxdv\cr
&\qquad +  C\lt(1 + \norm{\nabla^2_x u}^{\frac{q^2}{d}}_{L^\infty(0,t;L^q)}\rt)\inttr |v - v_c|^{p} |\nabla_v f|^{q}\,dxdv +  Ce^{-Cpt} \inttr  |\nabla_x^2 f|^{q}\,dxdv\cr
&\qquad -p\psi_mM_0 \inttr |v - v_c|^{p}|\nabla_x^2 f|^{q}\,dxdv -p \inttr |v - v_c|^{p}|\nabla_x^2 f|^{q}\,dxdv\cr
&\qquad  - d\inttr (1 + \psi \star \rho_f)  |v - v_c|^{p}  |\nabla_x^2 f|^{q} \,dxdv.
\end{align*}
\item In Step B, we estimate the mixed derivatives and derive
\begin{align*}
&\frac{d}{dt}\inttr |v - v_c|^{p} |\nabla_x \nabla_v f|^{q}\,dxdv\cr
&\quad \leq (C+\|\nabla_x u\|_{L^\infty}) \inttr |v - v_c|^{p} |\nabla_x \nabla_v f|^{q}\,dxdv + \inttr |v - v_c|^{p} |\nabla_x^2 f|^{q}\,dxdv\cr
&\qquad + C\inttr |v - v_c|^{p} |\nabla_v f|^{q}\,dxdv + (C+\|\nabla_x u\|_{L^\infty})\inttr |v - v_c|^{p} |\nabla_v^2 f|^{q}\,dxdv\cr
&\qquad - p\psi_m M_0\inttr |v - v_c|^{p}|\nabla_x \nabla_v f|^{q}\,dxdv + C_{p,q} e^{-Cp t}\inttr |\nabla_x \nabla_vf|^{q}\,dxdv\cr
&\qquad - p \inttr |v - v_c|^{p}|\nabla_x \nabla_v f|^{q}\,dxdv  - d\inttr (1 + \psi \star \rho_f)  |v - v_c|^{p}  |\nabla_x \nabla_v f|^{q} \,dxdv.
\end{align*}
\item In Step C, we estimate the generalized $p$-th moments of the second order derivative of $f$ in velocity variable and derive
\begin{align*}
&\frac{d}{dt}\inttr |v - v_c|^{p} |\nabla_v^2 f|^{q}\,dxdv\cr
&\quad \leq  C\inttr |v - v_c|^{p} |\nabla_v^2 f|^{q}\,dxdv +  \inttr |v - v_c|^{p} |\nabla_x \nabla_v f|^{q}\,dxdv\cr
&\qquad - p\psi_m M_0 \inttr |v - v_c|^{p}|\nabla_v^2 f|^{q}\,dxdv + C_{p,q} e^{-Cp t}\inttr |\nabla_v^2 f|^{q}\,dxdv\cr
&\qquad   -p \inttr |v - v_c|^{p}|\nabla_v^2 f|^{q}\,dxdv  - d\inttr (1 + \psi \star \rho_f)  |v - v_c|^{p}  |\nabla_v^2 f|^{q} \,dxdv.
\end{align*}
\item In Step D, we combine all of the estimates in the previous steps together with the lower order estimates to conclude our desired result.
\end{itemize}

{\bf Step A.-} For $i,j=1,\dots,d$, we first estimate
\begin{align*}
\frac{1}{q}\frac{d}{dt} \inttr |v - v_c|^{p} |\pa_{x_i} \pa_{x_j} f|^{q}\,dxdv
&= -\frac pq \inttr |v - v_c|^{p-2}(v - v_c) \cdot v_c'  |\pa_{x_i} \pa_{x_j} f|^{q}\,dxdv \cr
&\quad + \inttr |v - v_c|^{p} |\pa_{x_i} \pa_{x_j} f|^{q-2} \pa_{x_i} \pa_{x_j} f \,\pa_t \pa_{x_i} \pa_{x_j} f\,dxdv
\cr
&
=: I_1 + I_2,
\end{align*}
where $I_1$ can be easily estimated as
\[
I_1 \leq \frac{p}{4q} \inttr |v - v_c|^{p} |\pa_{x_i} \pa_{x_j} f|^{q}\,dxdv + Ce^{-Cpt} \inttr  |\pa_{x_i} \pa_{x_j} f|^{q}\,dxdv.
\]
For the estimate of $I_2$, we find
\begin{align*}
I_2 &= -\inttr |v - v_c|^{p} |\pa_{x_i} \pa_{x_j} f|^{q-2} \pa_{x_i} \pa_{x_j} f \lt( (\nabla_v \cdot \pa_{x_i} \pa_{x_j}F) f + \pa_{x_i} \pa_{x_j} F \cdot \nabla_v f \rt) \,dxdv\cr
&\quad -\inttr |v - v_c|^{p} |\pa_{x_i} \pa_{x_j} f|^{q-2} \pa_{x_i} \pa_{x_j} f \lt( (\nabla_v \cdot \pa_{x_j}F) \pa_{x_i} f + \pa_{x_j} F \cdot \nabla_v \pa_{x_i} f\rt) \,dxdv \cr
&\quad -\inttr |v - v_c|^{p} |\pa_{x_i} \pa_{x_j} f|^{q-2} \pa_{x_i} \pa_{x_j} f \lt( (\nabla_v \cdot \pa_{x_i} F) \pa_{x_j} f + \pa_{x_i} F \cdot \nabla_v \pa_{x_j} f \rt) \,dxdv \cr
&\quad -\inttr |v - v_c|^{p} |\pa_{x_i} \pa_{x_j} f|^{q-2} \pa_{x_i} \pa_{x_j} f \lt( (\nabla_v \cdot F) \pa_{x_i} \pa_{x_j} f - F \cdot \nabla_v \pa_{x_i} \pa_{x_j}f \rt) \,dxdv 
\cr
&
=: \sum_{i=1}^8 I_2^i.
\end{align*}
We first estimate $I_2^2$ as 
\begin{align*}
I_2^2 &\leq C\inttr  |v - v_c|^{p} |\pa_{x_i} \pa_{x_j} f|^{q-1}|\nabla_v f|\,dxdv+\inttr  |v - v_c|^{p} |\pa_{x_i} \pa_{x_j} f|^{q-1}|\pa_{x_i}\pa_{x_j}u||\nabla_v f|\,dxdv\cr
&\leq C \inttr |v - v_c|^{p} |\pa_{x_i} \pa_{x_j} f|^{q}\,dxdv +  C\inttr |v - v_c|^{p} |\nabla_v f|^{q}\,dxdv \cr
&\quad +\inttr  |v - v_c|^{p} |\pa_{x_i} \pa_{x_j} f|^{q-1}|\pa_{x_i}\pa_{x_j}u||\nabla_v f|\,dxdv.
\end{align*}
The last term in the above inequality can be estimated as follows:
\begin{align*}
& \inttr  |v - v_c|^{p} |\pa_{x_i} \pa_{x_j} f|^{q-1}|\pa_{x_i}\pa_{x_j}u||\nabla_v f|\,dxdv\cr
 &\quad =\int_{\R^d} |v - v_c|^{p} \int_{\T^d} |\pa_{x_i} \pa_{x_j} f|^{q-1}|\pa_{x_i}\pa_{x_j}u||\nabla_v f|\,dxdv\cr
&\quad \le \int_{\R^d} |v - v_c|^{p} \norm{\nabla^2_x f}^{q-1}_{L^q_x}\norm{\nabla^2_x u}_{L^q}\norm{\nabla_v f}_{L^{\infty}_x}\,dv\cr
&\quad \le \norm{\nabla^2_x u}_{L^q_x L^{\infty}_t}\int_{\R^d} |v - v_c|^{p} \norm{\nabla^2_x f}^{q-1}_{L^q_x}\norm{\nabla_v f}^{1-\frac{d}{q}}_{L^{q}_x}\norm{\nabla_x\nabla_v f}^{\frac{d}{q}}_{L^{q}_x}\,dv\cr
&\quad \le  \norm{\nabla^2_x u}_{L^q_x L^{\infty}_t}\int_{\R^d} |v - v_c|^{p(\frac{(q-1)+1-\frac{d}{q}+\frac{d}{q}}{q})} \norm{\nabla^2_x f}^{q-1}_{L^q_x}\norm{\nabla_v f}^{1-\frac{d}{q}}_{L^{q}_x}\norm{\nabla_x\nabla_v f}^{\frac{d}{q}}_{L^{q}_x}\,dv\cr
&\quad \le  \norm{\nabla^2_x u}_{L^q_x L^{\infty}_t} \bke{\int_{\R^d}  |v - v_c|^{p}  \norm{\nabla^2_x f}^{q}_{L^q_x}\,dv}^{\frac{q-1}{q}}\bke{\int_{\R^d}  |v - v_c|^{p}  \norm{\nabla_v f}^{q}_{L^q_x}\,dv}^{\frac{q-d}{q^2}}
\cr
&\hspace{5cm}  \times \bke{\int_{\R^d}  |v - v_c|^{p} \norm{\nabla_x\nabla_v f}^{q}_{L^q_x}\,dv}^{\frac{d}{q^2}}\cr
&\quad = \norm{\nabla^2_x u}_{L^q_x L^{\infty}_t} \bke{ \inttr  |v - v_c|^{p}\abs{\nabla^2_x f}^{q}\,dxdv}^{\frac{q-1}{q}} \bke{ \inttr  |v - v_c|^{p}\abs{\nabla_v f}^{q}\,dxdv}^{\frac{q-d}{q^2}}\cr
&\hspace{5cm}  \times \bke{ \inttr  |v - v_c|^{p}\abs{\nabla_x \nabla_v  f}^{q}\,dxdv}^{\frac{d}{q^2}}\cr
&\quad \le C \inttr  |v - v_c|^{p}\abs{\nabla^2_x f}^{q}\,dxdv+ C\inttr  |v - v_c|^{p}\abs{\nabla_x \nabla_v  f}^{q}\,dxdv\cr
&\qquad + C\norm{\nabla^2_x u}^{\frac{q^2}{d}}_{L^\infty(0,t;L^q)} \inttr  |v - v_c|^{p}\abs{\nabla_v f}^{q}\,dxdv.
\end{align*}
Next, we then estimate other terms $I_2^i$, $1\le i\le 8$ and $i\neq 2$, separately.
\begin{align*}
I_2^1 &\leq d\|\nabla_x^2\psi\|_{L^\infty} \inttr  |v - v_c|^{p} |\pa_{x_i} \pa_{x_j} f|^{q-1}|f|\,dxdv\cr
&\leq C \inttr |v - v_c|^{p} |\pa_{x_i} \pa_{x_j} f|^{q}\,dxdv +  C\inttr |v - v_c|^{p} |f|^{q}\,dxdv,
\end{align*}
\begin{align*}
I_2^3 &\leq d\|\pa_{x_j}\psi\|_{L^\infty}\inttr  |v - v_c|^{p} |\pa_{x_i} \pa_{x_j} f|^{q-1}|\pa_{x_i}f|\,dxdv\cr
&\leq C \inttr |v - v_c|^{p} |\pa_{x_i} \pa_{x_j} f|^{q}\,dxdv +  C\inttr |v - v_c|^{p} |\pa_{x_i}f|^{q}\,dxdv,
\end{align*}
\begin{align*}
I_2^4 &\leq (\|\pa_{x_j}u\|_{L^\infty}+ C)\inttr  |v - v_c|^{p} |\pa_{x_i} \pa_{x_j} f|^{q-1}|\nabla_v\pa_{x_i}f|\,dxdv\cr
&\leq (\|\pa_{x_j}u\|_{L^\infty}+ C) \inttr |v - v_c|^{p} |\pa_{x_i} \pa_{x_j} f|^{q}\,dxdv 
%\cr
%&\quad 
+  (\|\pa_{x_j}u\|_{L^\infty}+ C)\inttr |v - v_c|^{p} |\nabla_v\pa_{x_i}f|^{q}\,dxdv,
\end{align*}
\begin{align*}
I_2^5 &\leq d\|\pa_{x_i}\psi\|_{L^\infty}\inttr  |v - v_c|^{p} |\pa_{x_i} \pa_{x_j} f|^{q-1}|\pa_{x_j}f|\,dxdv\cr
&\leq C \inttr |v - v_c|^{p} |\pa_{x_i} \pa_{x_j} f|^{q}\,dxdv +  C\inttr |v - v_c|^{p} |\pa_{x_j}f|^{q}\,dxdv,
\end{align*}
\begin{align*}
I_2^6 &\leq (\|\pa_{x_i}u\|_{L^\infty}+ C)\inttr  |v - v_c|^{p} |\pa_{x_i} \pa_{x_j} f|^{q-1}|\nabla_v\pa_{x_j}f|\,dxdv\cr
&\leq (\|\pa_{x_i}u\|_{L^\infty}+ C) \inttr |v - v_c|^{p} |\pa_{x_i} \pa_{x_j} f|^{q}\,dxdv
%\cr
%&\quad 
+  (\|\pa_{x_i}u\|_{L^\infty}+ C)\inttr |v - v_c|^{p} |\nabla_v\pa_{x_j}f|^{q}\,dxdv,
\end{align*}
\[
I_2^7 \leq d(1 + \|\psi\|_{L^\infty})\inttr |v - v_c|^{p} |\pa_{x_i} \pa_{x_j} f|^{q}\,dxdv,
\]
\begin{align*}
I_2^8 &=-\frac{1}{q}\inttr |v - v_c|^{p}F \cdot \nabla_v |\pa_{x_i}\pa_{x_j} f|^{q}\,dxdv
%\cr
%&
=\frac{1}{q}\inttr \nabla_v \cdot( |v - v_c|^{p}F) |\pa_{x_i}\pa_{x_j} f|^{q}\,dxdv\cr
&\leq -\frac{p\psi_m M_0}{q} \inttr |v - v_c|^{p}|\pa_{x_i}\pa_{x_j} f|^{q}\,dxdv + C_{p,q} e^{-Cp t}\inttr |\pa_{x_i}\pa_{x_j} f|^{q}\,dxdv\cr
&\quad -\frac {p(1 - C\delta)}{q} \inttr |v - v_c|^{p}|\pa_{x_i}\pa_{x_j} f|^{q}\,dxdv + C_{p,q}\|u - u_c\|_{L^\infty}^{p} \inttr |\pa_{x_i}\pa_{x_j} f|^{q}\,dxdv\cr
&\quad  - \frac{d}{q}\inttr (1 + \psi \star \rho_f)  |v - v_c|^{p}  |\pa_{x_i}\pa_{x_j} f|^{q} \,dxdv.
\end{align*}
This yields
\begin{align*}
&\frac{d}{dt} \inttr |v - v_c|^{p} |\pa_{x_i} \pa_{x_j} f|^{q}\,dxdv\cr
&\quad \leq (C+\|\pa_{x_j}u\|_{L^\infty}) \inttr |v - v_c|^{p}\lt( |\pa_{x_i} \pa_{x_j} f|^{q} +  |\nabla_v\pa_{x_i}f|^{q}+  |\nabla_v\pa_{x_j}f|^{q}\rt)dxdv \cr
&\qquad + C\inttr |v - v_c|^{p} \lt(|f|^{q} + |\pa_{x_i} f|^{q}+ |\pa_{x_j} f|^{q}\rt)dxdv\cr
&\qquad +  C\lt(1 + \norm{\nabla^2_x u}^{\frac{q^2}{d}}_{L^\infty(0,t;L^q)}\rt)\inttr |v - v_c|^{p} |\nabla_v f|^{q}\,dxdv +  Ce^{-Cpt} \inttr  |\pa_{x_i} \pa_{x_j} f|^{q}\,dxdv\cr
&\qquad -p\psi_mM_0 \inttr |v - v_c(t)|^{p}|\pa_{x_i}\pa_{x_j} f|^{q}\,dxdv -p \inttr |v - v_c|^{p}|\pa_{x_i}\pa_{x_j} f|^{q}\,dxdv\cr
&\qquad  - d\inttr (1 + \psi \star \rho_f)  |v - v_c|^{p}  |\pa_{x_i}\pa_{x_j} f|^{q} \,dxdv,
\end{align*}
and this subsequently implies
\begin{align*}
&\frac{d}{dt} \inttr |v - v_c|^{p} |\nabla_x^2 f|^{q}\,dxdv\cr
&\quad \leq (C+\|\nabla_x u\|_{L^\infty}) \inttr |v - v_c|^{p}\lt( |\nabla_x^2 f|^{q} +  |\nabla_v\nabla_x f|^{q} \rt)dxdv 
%\cr
%&\qquad 
+ C\inttr |v - v_c|^{p} \lt(|f|^{q} + |\nabla_x f|^{q}\rt)dxdv\cr
&\qquad +  C\lt(1 + \norm{\nabla^2_x u}^{\frac{q^2}{d}}_{L^\infty(0,t;L^q)}\rt)\inttr |v - v_c|^{p} |\nabla_v f|^{q}\,dxdv +  Ce^{-Cpt} \inttr  |\nabla_x^2 f|^{q}\,dxdv\cr
&\qquad -p\psi_m M_0\inttr |v - v_c|^{p}|\nabla_x^2 f|^{q}\,dxdv -p \inttr |v - v_c|^{p}|\nabla_x^2 f|^{q}\,dxdv\cr
&\qquad  - d\inttr (1 + \psi \star \rho_f)  |v - v_c|^{p}  |\nabla_x^2 f|^{q} \,dxdv,
\end{align*}
where $C>0$ is independent of $t$. \newline

{\bf Step B.-} We next estimate
\begin{align*}
&\frac{1}{q}\frac{d}{dt}\inttr |v - v_c|^{p} |\pa_{x_i} \pa_{v_j} f|^{q}\,dxdv
= -\frac pq \inttr |v - v_c|^{p-2}(v - v_c) \cdot v_c'  |\pa_{x_i} \pa_{v_j} f|^{q}\,dxdv \cr
&\qquad + \inttr |v - v_c|^{p} |\pa_{x_i} \pa_{v_j} f|^{q-2} \pa_{x_i} \pa_{v_j} f \,\pa_t \pa_{x_i} \pa_{v_j} f\,dxdv
%\cr
%&
=: J_1 + J_2,
\end{align*}
where $J_1$ can be easily estimated as
\[
J_1 \leq \frac{p}{4q} \inttr |v - v_c|^{p} |\pa_{x_i} \pa_{v_j} f|^{q}\,dxdv + Ce^{-Cpt} \inttr  |\pa_{x_i} \pa_{v_j} f|^{q}\,dxdv.
\]
For the estimate of $J_2$, we find
\begin{align*}
J_2 &= -\inttr |v - v_c|^{p} |\pa_{x_i} \pa_{v_j} f|^{q-2} \pa_{x_i} \pa_{v_j} f \lt(  \pa_{x_i} \pa_{x_j} f + (\pa_{x_i} \nabla_v \cdot F) \pa_{v_j} f  \rt) \,dxdv\cr
&\quad  -\inttr |v - v_c|^{p} |\pa_{x_i} \pa_{v_j} f|^{q-2} \pa_{x_i} \pa_{v_j} f \lt((\nabla_v \cdot F)\pa_{x_i} \pa_{v_j} f  + \pa_{x_i} \pa_{v_j} F \cdot \nabla_v f \rt) \,dxdv\cr
&\quad  -\inttr |v - v_c|^{p} |\pa_{x_i} \pa_{v_j} f|^{q-2} \pa_{x_i} \pa_{v_j} f \lt( \pa_{v_j} F \cdot \nabla_v \pa_{x_i} f + \pa_{x_i} F \cdot \nabla_v \pa_{v_j} f  \rt) \,dxdv\cr
&\quad  -\inttr |v - v_c|^{p} |\pa_{x_i} \pa_{v_j} f|^{q-2} \pa_{x_i} \pa_{v_j} f \lt( F \cdot \nabla_v \pa_{x_i} \pa_{v_j} f  \rt) \,dxdv
\cr
&
=: \sum_{i=1}^7 J_2^i.
\end{align*}
We then estimate each $J_2^i,i=1,\dots,7$ as follows:
\begin{align*}
J_2^1 &
\leq  \inttr |v - v_c|^{p} |\pa_{x_i} \pa_{v_j} f|^{q-1}|\pa_{x_i}\pa_{x_j}f|
\cr
&
\leq \inttr |v - v_c|^{p} |\pa_{x_i} \pa_{v_j} f|^{q}\,dxdv + \inttr |v - v_c|^{p} |\pa_{x_i}\pa_{x_j}f|^{q},
\end{align*}
\begin{align*}
J_2^2 &\leq d\|\pa_{x_i}\psi\|_{L^\infty} \inttr |v - v_c|^{p} |\pa_{x_i} \pa_{v_j} f|^{q-1}|\pa_{v_j}f|
\cr
&
\leq C\inttr |v - v_c|^{p} |\pa_{x_i} \pa_{v_j} f|^{q}\,dxdv + C\inttr |v - v_c|^{p} |\pa_{v_j}f|^{q},
\end{align*}
\[
J_2^3 \leq  d(1 + \|\psi\|_{L^\infty})\inttr |v - v_c|^{p} |\pa_{x_i} \pa_{v_j} f|^{q},
\]
\begin{align*}
J_2^4 &\leq \|\pa_{x_i}\psi\|_{L^\infty} \inttr |v - v_c|^{p} |\pa_{x_i} \pa_{v_j} f|^{q-1}|\nabla_v f|\cr
&\leq C\inttr |v - v_c|^{p} |\pa_{x_i} \pa_{v_j} f|^{q}+ C\inttr |v - v_c|^{p} |\nabla_v f|^{q},
\end{align*}
\begin{align*}
J_2^5 &\leq (1 + \|\psi\|_{L^\infty}) \inttr |v - v_c|^{p} |\pa_{x_i} \pa_{v_j} f|^{q-1}|\nabla_v \pa_{x_i}f|\cr
&\leq C\inttr |v - v_c|^{p} |\pa_{x_i} \pa_{v_j} f|^{q}\,dxdv + C\inttr |v - v_c|^{p} |\nabla_v \pa_{x_i} f|^{q},
\end{align*}
\begin{align*}
J_2^6 &\leq (C+\|\pa_{x_i} u\|_{L^\infty}) \inttr |v - v_c|^{p} |\pa_{x_i} \pa_{v_j} f|^{q-1}|\nabla_v \pa_{v_j}f|\cr
&\leq (C+\|\pa_{x_i} u\|_{L^\infty})\inttr |v - v_c|^{p} |\pa_{x_i} \pa_{v_j} f|^{q}
\cr
&\quad 
+ (C+\|\pa_{x_i} u\|_{L^\infty})\inttr |v - v_c|^{p} |\nabla_v \pa_{v_j} f|^{q},
\end{align*}
and
\begin{align*}
J_2^7 &=-\frac{1}{q}\inttr |v - v_c|^{p}F \cdot \nabla_v |\pa_{x_i}\pa_{v_j} f|^{q}=\frac{1}{q}\inttr \nabla_v \cdot( |v - v_c|^{p}F) |\pa_{x_i}\pa_{v_j} f|^{q}\cr
&\leq -\frac{p\psi_m M_0}{q} \inttr |v - v_c|^{p}|\pa_{x_i}\pa_{v_j} f|^{q} + C_{p,q} e^{-Cp t}\inttr |\pa_{x_i}\pa_{v_j} f|^{q}\cr
&\quad -\frac {p(1 - C\delta)}{q} \inttr |v - v_c|^{p}|\pa_{x_i}\pa_{v_j} f|^{q}+ C_{p,q}\|u - u_c\|_{L^\infty}^{p} \inttr |\pa_{x_i}\pa_{v_j} f|^{q}\cr
&\quad - \frac{d}{q}\inttr (1 + \psi \star \rho_f)  |v - v_c|^{p}  |\pa_{x_i}\pa_{v_j} f|^{q}.
\end{align*}
This gives
\begin{align*}
&\frac{d}{dt}\inttr |v - v_c|^{p} |\pa_{x_i} \pa_{v_j} f|^{q}\,dxdv\cr
&\quad \leq (C+\|\pa_{x_i} u\|_{L^\infty}) \inttr |v - v_c|^{p} |\pa_{x_i} \pa_{v_j} f|^{q}\,dxdv + \inttr |v - v_c|^{p} |\pa_{x_i}\pa_{x_j}f|^{q}\,dxdv,\cr
&\qquad + C\inttr |v - v_c|^{p} |\nabla_v f|^{q}\,dxdv + C\inttr |v - v_c|^{p} |\nabla_v \pa_{x_i} f|^{q}\,dxdv\cr
&\qquad  + (C+\|\pa_{x_i} u\|_{L^\infty})\inttr |v - v_c|^{p} |\nabla_v \pa_{v_j} f|^{q}\,dxdv -\frac{p(1 + \psi_m - C\delta)}{q} \int |v - v_c|^{p}|\pa_{x_i}\pa_{v_j} f|^{q}\,dxdv \cr
&\qquad + C_{p,q} e^{-Cp t}\inttr |\pa_{x_i}\pa_{v_j} f|^{q}\,dxdv - \frac{d}{q}\inttr (1 + \psi \star \rho_f)  |v - v_c|^{p}  |\pa_{x_i}\pa_{v_j} f|^{q} \,dxdv,
\end{align*}
and this further yields
\begin{align*}
&\frac{d}{dt}\inttr |v - v_c|^{p} |\nabla_x \nabla_v f|^{q}\,dxdv\cr
&\quad \leq (C+\|\nabla_x u\|_{L^\infty}) \inttr |v - v_c|^{p} |\nabla_x \nabla_v f|^{q}\,dxdv + \inttr |v - v_c|^{p} |\nabla_x^2 f|^{q}\,dxdv\cr
&\qquad + C\inttr |v - v_c|^{p} |\nabla_v f|^{q}\,dxdv + (C+\|\nabla_x u\|_{L^\infty})\inttr |v - v_c|^{p} |\nabla_v^2 f|^{q}\,dxdv\cr
&\qquad - p\psi_m M_0\inttr |v - v_c|^{p}|\nabla_x \nabla_v f|^{q}\,dxdv + C_{p,q} e^{-Cp t}\inttr |\nabla_x \nabla_vf|^{q}\,dxdv\cr
&\qquad - p \inttr |v - v_c|^{p}|\nabla_x \nabla_v f|^{q}\,dxdv  - d\inttr (1 + \psi \star \rho_f)  |v - v_c|^{p}  |\nabla_x \nabla_v f|^{q} \,dxdv,
\end{align*}
where $C>0$ is independent of $t$. \newline

{\bf Step C.-} We now estimate
\begin{align*}
\frac{1}{q}\frac{d}{dt}\inttr |v - v_c|^{p} |\pa_{v_i} \pa_{v_j} f|^{q}\,dxdv
&= -\frac pq \inttr |v - v_c|^{p-2}(v - v_c) \cdot v_c'  |\pa_{v_i} \pa_{v_j} f|^{q}\,dxdv \cr
&\quad + \inttr |v - v_c|^{p} |\pa_{v_i} \pa_{v_j} f|^{q-2} \pa_{v_i} \pa_{v_j} f \pa_t \pa_{v_i} \pa_{v_j} f\,dxdv
\cr
&
=: K_1 + K_2,
\end{align*}
where $K_1$ can be easily estimated as
\[
K_1 \leq \frac{p}{4q} \inttr |v - v_c|^{p} |\pa_{v_i} \pa_{v_j} f|^{q}\,dxdv + Ce^{-Cpt} \inttr  |\pa_{v_i} \pa_{v_j} f|^{q}\,dxdv.
\]
For the estimate of $K_2$, we find
\begin{align*}
K_2 &= -\inttr |v - v_c|^{p} |\pa_{v_i} \pa_{v_j} f|^{q-2} \pa_{v_i} \pa_{v_j} f \lt( \pa_{v_i}\pa_{x_j}f + \pa_{x_i}\pa_{v_j}f   \rt) \,dxdv\cr
&\quad -  \inttr |v - v_c|^{p} |\pa_{v_i} \pa_{v_j} f|^{q-2} \pa_{v_i} \pa_{v_j} f \lt(   (\nabla_v \cdot F) \pa_{v_i}\pa_{v_j}f + \pa_{v_j} F \cdot \nabla_v \pa_{v_i} f  \rt) \,dxdv\cr
&\quad -  \inttr |v - v_c|^{p} |\pa_{v_i} \pa_{v_j} f|^{q-2} \pa_{v_i} \pa_{v_j} f \lt( \pa_{v_i} F \cdot \nabla_v \pa_{v_j} f + F \cdot \nabla_v \pa_{v_i}\pa_{v_j} f    \rt) \,dxdv
\cr
&
=: \sum_{i=1}^6 K_2^i.
\end{align*}
We estimate $K_2^i$. $i=1,\dots, 6$, separately.
\begin{align*}
K_2^1 + K_2^2 &\leq \inttr |v - v_c|^{p} |\pa_{v_i} \pa_{v_j} f|^{q-1}\lt( |\pa_{v_i} \pa_{x_j} f| + |\pa_{x_i} \pa_{v_j} f| \rt) dxdv\cr
&\leq  \inttr |v - v_c|^{p} |\pa_{v_i} \pa_{v_j} f|^{q}\,dxdv +  \inttr |v - v_c|^{p} |\pa_{v_i} \pa_{x_j} f|^{q}\,dxdv
\cr
&\quad 
+  \inttr |v - v_c|^{p} |\pa_{x_i} \pa_{v_j} f|^{q}\,dxdv,
\end{align*}
\[
K_2^3 \leq d(1 + \|\psi\|_{L^\infty})  \inttr |v - v_c|^{p} |\pa_{v_i} \pa_{v_j} f|^{q}\,dxdv,
\]
\begin{align*}
K_2^4 + K_2^5 &\leq (1 + \|\psi\|_{L^\infty})\inttr |v - v_c|^{p} |\pa_{v_i} \pa_{v_j} f|^{q-1}\lt( |\nabla_v \pa_{v_i} f| + |\nabla_v \pa_{v_j} f| \rt) dxdv\cr
&\leq  C\inttr |v - v_c|^{p} |\pa_{v_i} \pa_{v_j} f|^{q}\,dxdv + C\inttr |v - v_c|^{p} |\nabla_v \pa_{v_i} f|^{q}\,dxdv
\cr
&\quad 
+  C\inttr |v - v_c|^{p} |\nabla_v \pa_{v_j} f|^{q}\,dxdv,
\end{align*}
and
\begin{align*}
K_2^6 &= -\frac{1}{q}\inttr |v - v_c|^{p}F \cdot \nabla_v |\pa_{v_i}\pa_{v_j} f|^{q}\,dxdv\cr
&=\frac{1}{q}\inttr \nabla_v \cdot( |v - v_c|^{p}F) |\pa_{v_i}\pa_{v_j} f|^{q}\,dxdv\cr
&\leq -\frac{p\psi_m M_0}{q} \inttr |v - v_c|^{p}|\pa_{v_i}\pa_{v_j} f|^{q}\,dxdv + C_{p,q} e^{-Cp t}\inttr |\pa_{v_i}\pa_{v_j} f|^{q}\,dxdv\cr
&\quad -\frac {p(1 - C\delta)}{q} \inttr |v - v_c|^{p}|\pa_{v_i}\pa_{v_j} f|^{q}\,dxdv + C_{p,q}\|u - u_c\|_{L^\infty}^{p} \inttr |\pa_{v_i}\pa_{v_j} f|^{q}\,dxdv\cr
&\quad - \frac{d}{q}\inttr (1 + \psi \star \rho_f)  |v - v_c|^{p}  |\pa_{v_i}\pa_{v_j} f|^{q} \,dxdv.
\end{align*}
Thus by choosing $\delta >0$ small enough we get
\begin{align*}
&\frac{d}{dt}\inttr |v - v_c|^{p} |\pa_{v_i} \pa_{v_j} f|^{q}\,dxdv\cr
&\quad \leq  C\inttr |v - v_c|^{p} |\pa_{v_i} \pa_{v_j} f|^{q}\,dxdv +  \inttr |v - v_c|^{p} |\pa_{v_i} \pa_{x_j} f|^{q}\,dxdv \cr
&\qquad +  \inttr |v - v_c|^{p} |\pa_{x_i} \pa_{v_j} f|^{q}\,dxdv   + C\inttr |v - v_c|^{p} |\nabla_v \pa_{v_i} f|^{q}\,dxdv\cr
&\qquad  +  C\inttr |v - v_c|^{p} |\nabla_v \pa_{v_j} f|^{q}\,dxdv - p(1 + \psi_m M_0 - C\delta)\int |v - v_c|^{p}|\pa_{v_i}\pa_{v_j} f|^{q}\,dxdv \cr
&\qquad + C_{p,q} e^{-Cp t}\inttr |\pa_{v_i}\pa_{v_j} f|^{q}\,dxdv -d\inttr (1 + \psi \star \rho_f)  |v - v_c|^{p}  |\pa_{v_i}\pa_{v_j} f|^{q} \,dxdv,
\end{align*}
and this yields
\begin{align*}
&\frac{d}{dt}\inttr |v - v_c|^{p} |\nabla_v^2 f|^{q}\,dxdv
%\cr
%&\quad 
\leq  C\inttr |v - v_c|^{p} |\nabla_v^2 f|^{q}\,dxdv +  \inttr |v - v_c|^{p} |\nabla_x \nabla_v f|^{q}\,dxdv\cr
&\qquad - p(1 + \psi_mM_0 - C\delta)\inttr |v - v_c|^{p}|\nabla_v^2 f|^{q}\,dxdv + C_{p,q} e^{-Cp t}\inttr |\nabla_v^2 f|^{q}\,dxdv\cr
&\qquad    - d\inttr (1 + \psi \star \rho_f)  |v - v_c|^{p}  |\nabla_v^2 f|^{q} \,dxdv,
\end{align*}
where $C>0$ is independent of $t$. \newline

{\bf Step D.-} We now combine all of the above estimates to have
\begin{align*}
&\frac{d}{dt}\inttr |v - v_c|^{p} \lt(|\nabla_x^2 f|^{q} + |\nabla_x \nabla_v f|^{q} + |\nabla_v^2 f|^{q}\rt)\,dxdv\cr
&\quad \leq (C+\|\nabla_x u\|_{L^\infty}) \inttr |v - v_c|^{p}\lt( |\nabla_x^2 f|^{q} +  |\nabla_v\nabla_x f|^{q} + |\nabla_v^2 f|^{q}\rt)dxdv \cr
&\qquad + C\inttr |v - v_c|^{p} \lt(|f|^{q} + |\nabla_x f|^{q} + |\nabla_v f|^{q} \rt)dxdv \cr
&\qquad +  C(1 + \norm{\nabla^2_x u}^{\frac{q^2}{d}}_{L^q_x L^{\infty}_t})\inttr |v - v_c|^{p} |\nabla_v f|^{q}\,dxdv \cr
&\qquad +  Ce^{-Cpt} \inttr   \lt(|\nabla_x^2 f|^{q} + |\nabla_x \nabla_v f|^{q} + |\nabla_v^2 f|^{q}\rt)\,dxdv\cr
&\qquad -(p(\psi_m M_0 +1 - C\delta)+C)\inttr |v - v_c(t)|^{p} \lt(|\nabla_x^2 f|^{q} + |\nabla_x \nabla_v f|^{q} + |\nabla_v^2 f|^{q}\rt) dxdv,
\end{align*}
and this together with Propositions \ref{prop_fw1}, \ref{prop_fw1_d3},  \ref{thm_decay_w1}, and \eqref{nabla-u-est} yields
\begin{align*}
&\frac{d}{dt}\inttr |v - v_c|^{p} \lt(|\nabla_x^2 f|^{q} + |\nabla_x \nabla_v f|^{q} + |\nabla_v^2 f|^{q}\rt)\,dxdv\cr
&\quad \leq (C+\|\nabla_x u\|_{L^\infty}-(p(\psi_m M_0 +1- C\delta)) \inttr |v - v_c|^{p}\lt( |\nabla_x^2 f|^{q} +  |\nabla_v\nabla_x f|^{q} + |\nabla_v^2 f|^{q}\rt)dxdv\cr
&\qquad + Ce^{-Cpt},
\end{align*}
where $C>0$ is independent of $t$. We finally use almost the same argument as in the proof of Proposition \ref{thm_decay_w1} to conclude the desired exponential decay estimate.
\end{mainthm}

 %%%%%%%%%%%%%%%%%%%%%%%%%%%%%%%%%%%%%%%
%
%
% \appendix
%
%%%%%%%%%%%%%%%%%%%%%%%%%%%%%%%%%%%%%%%

 \section*{Acknowledgments}
Y.-P. Choi's work is supported by NRF-2017R1C1B2012918, POSCO Science Fellowship of POSCO TJ Park Foundation, and Yonsei University Research Fund of 2019-22-021.
K. Kang's work is partially supported by NRF-2019R1A2C1084685
 and NRF-2015R1A5A1009350. H. K. Kim's work is supported by 
NRF-2018R1D1A1B07049357.
J.-M. Kim's work is supported by NRF-2020R1C1C1A01006521 and a Research Grant of Andong National University.

%%%%%%%%%%%%%%%%%%%%%%%%%%%%%%%%%%%%%%%
%
%
% \appendix
%
%%%%%%%%%%%%%%%%%%%%%%%%%%%%%%%%%%%%%%%

\appendix

\section{Proof of Lemma \ref{lem_hk}}\label{Lemma33}

In this part, we provide the details on the proof of Lemma \ref{lem_hk}.

We first estimate
\[
\lt|\Gamma(x,t) - \frac{1}{(2\pi)^d}\rt|\leq C\sum_{|\xi| \geq 1} e^{-t|\xi|^2}
\leq C\int_1^\infty \int_{\pa B(0,r)} e^{-r^2 t} \,dSdr
\leq C\int_1^\infty r^{d-1} e^{-r^2t}\,dr,
\]
where $C>0$ is independent of $t$. If $d=2$, we easily get
\[
\int_1^\infty re^{-r^2 t}\,dr = \frac{e^{-t}}{2t}.
\]
For $d=3$, we obtain
\[
\int_1^\infty r^2 e^{-r^2 t}\,dr = \frac{1}{2t}\int_1^\infty  r (2rt) e^{-r^2 t}\,dr = \frac{1}{2t} e^{-t} + \frac1{2t} \int_1^\infty e^{-r^2 t}\,dr.
\]
On the other hand, we find
\[
\frac1{2t} \int_1^\infty e^{-r^2 t}\,dr = \frac{1}{2t^{3/2}} \int_{\sqrt t}^\infty e^{-x^2}\,dx \le  \frac{e^{-t}}{t^{3/2}}.
\]
Indeed, in case that $t\ge 1$, it is direct that
\[
\int_{\sqrt t}^\infty e^{-x^2}\,dx\le \int_{\sqrt t}^\infty 2x e^{-x^2}=e^{-t}.
\]
In case that $t<1$, the integral is bounded since 
\[
\int_{\sqrt t}^\infty e^{-x^2}\,dx\le \int_{0}^\infty  e^{-x^2}=\sqrt{\pi}.
\]
This gives
\[
\frac{1}{2t^{3/2}}\int_{\sqrt t}^\infty e^{-x^2}\,dx\le \frac{1}{2t^{3/2}}\bke{e^{-t}\chi_{\bket{t\ge 1}}+ \sqrt{\pi}\chi_{\bket{t<1}}}\le C\frac{e^{-t}}{t^{3/2}}.
\]
Note that the last integral is uniformly bounded in $t$. Thus we have
\[
\lt\|\Gamma(\cdot,t) - \frac{1}{(2\pi)^d}\rt\|_{L^\infty} \leq \left\{ \begin{array}{ll}
\displaystyle Ct^{-1}e^{-t} & \textrm{if $d=2$,}\\[2mm]
\displaystyle C\bke{t^{-1} + t^{-\frac{3}{2}}}e^{-t} & \textrm{if $d=3$,}
 \end{array} \right.
\]
where $C>0$ is independent of $t$. Finally we use the $L^p$-interpolation inequality together with the fact
\[
\|\Gamma(\cdot,t)\|_{L^1} = \frac{1}{(4\pi t)^{d/2}} \sum_{\xi \in \Z^d} \intt  e^{-\frac{|x- 2\pi \xi|^2}{4t}}\,dx =  \frac{1}{(4\pi t)^{d/2}} \int_{\R^d}  e^{-\frac{|x|^2}{4t}}\,dx = 1 \quad \forall\, t \geq 0
\]
to conclude the first assertion. Similarly, we also estimate
\begin{align*}
\lt|\nabla_x\Gamma(x,t) \rt| &\leq \lt|\frac{1}{(2\pi)^d} \sum_{\xi \in \Z^d \setminus \{0\}} |\xi| e^{-t|\xi|^2 + i \xi \cdot x} \rt| \leq C\sum_{|\xi| \geq 1} |\xi| e^{-t|\xi|^2} \leq C\int_1^\infty r^d e^{-r^2t}\,dr.
\end{align*}
Thus for $d=2$ we obtain from the above that
\[
\|\nabla_x \Gamma(\cdot,t)\|_{L^\infty} \leq C\bke{t^{-1} + t^{-\frac{3}{2}}}e^{-t},
\]
where $C>0$ is independent of $t$. For $d=3$, we use the change of variable and integration by parts to get
\[
\int_1^\infty r^3 e^{-r^2t}\,dr = \frac12\int_1^\infty re^{-rt}\,dr = \frac{e^{-t}}{2t}\lt(1 + \frac1t \rt).
\]
This implies
\begin{equation}\label{april03-30}
\|\nabla_x \Gamma(\cdot,t)\|_{L^\infty} \leq
C\bke{t^{-1} + t^{-\frac{d+1}{2}}}e^{-t}.
\end{equation}
On the other hand, we have
\begin{align}\label{april03-40}
\begin{aligned}
\intt |\nabla_x \Gamma(x,t)|\,dx &\leq \frac{1}{(4\pi t)^{d/2}} \sum_{\xi \in \Z^d} \intt  \frac{|x- 2\pi \xi|}{2t}e^{-\frac{|x- 2\pi \xi|^2}{4t}}\,dx\cr
&\leq   \frac{C}{t^{d/2 + 1}} \int_{\R^d}  |x| e^{-\frac{|x|^2}{4t}}\,dx\le Ct^{-\frac12} \int_0^\infty r^d e^{-r^2}\,dr\leq Ct^{-\frac12}.
\end{aligned}
\end{align}
Interpolating estimates \eqref{april03-30} and  \eqref{april03-40}, it follows that
\[
\|\nabla_x \Gamma(\cdot,t)\|_{L^p} \leq \|\nabla_x \Gamma(\cdot,t)\|_{L^\infty}^{1-\frac1p}\|\nabla_x \Gamma(\cdot,t)\|_{L^1}^{\frac1p}\leq C \lt(t^{-\lt(1 - \frac{1}{2p}\rt)} + t^{-\frac d2\lt( 1 - \frac1p\rt) - \frac12} \rt)e^{-t\lt(1 - \frac1p\rt)}.
\]
\begin{remark} If we assume higher regularity for the initial vorticity $\omega_0$, for instance $\omega_0 \in \dot H^s(\T^d)$ with $s > d/2$, then we obtain
\[
\|(\Gamma \star \omega_0)(\cdot,t)\|_{L^p} \leq C\|\omega_0\|_{\dot H^s} e^{-t/2} \quad \mbox{for} \quad 1 \leq p \leq \infty \quad \mbox{and} \quad t \geq 0,
\]
where $C>0$ is independent of $p$ and $t$. Indeed, it follows from \eqref{sol_homo} that
\[
(\Gamma \star \omega_0)(x,t) = \sum_{\xi \in \Z^d \setminus \{0\}} e^{-t|\xi|^2 + i \xi \cdot x}c_\mf(\omega_0)(\xi),
\]
where we used the fact that $\intt\omega_0\,dx=0$. Then we estimate the right hand side of the above equality as
\begin{align*}
\lt|\sum_{\xi \in \Z^d \setminus \{0\}} e^{-t|\xi|^2 + i \xi \cdot x}c_\mf(\omega_0)(\xi) \rt| &=\lt|\sum_{\xi \in \Z^d \setminus \{0\}} \frac{e^{-t|\xi|^2 + i \xi \cdot x} }{|\xi|^s}|\xi|^s c_\mf(\omega_0)(\xi) \rt|\cr
&\leq \lt(\sum_{\xi \in \Z^d \setminus \{0\}} \frac{e^{-t|\xi|^2}}{|\xi|^{2s}} \rt)^{1/2}\lt(\sum_{\xi \in \Z^d \setminus \{0\}} |\xi|^{2s}|c_\mf(\omega_0)(\xi)|^2  \rt)^{1/2}\cr
&\leq e^{-t/2} \lt(\sum_{\xi \in \Z^d \setminus \{0\}} \frac1{|\xi|^{2s}} \rt)^{1/2}\|\omega_0\|_{\dot H^s}.
\end{align*}
Since the sum on the right hand side of the above inequality is bounded for $2s > d$, we conclude the desired result.

\end{remark}

\section{Proof of Lemma \ref{lem_f2}}\label{app_a}

In this appendix, we provide the details of proof of Lemma \ref{lem_f2}. \newline

($\diamond$ Estimate of $\|\nabla_x^2 f\|_{L^{q}}$): For $i,j=1,\dots,d$, we find from the kinetic equation in \eqref{main} that
\begin{align*}
\pa_t \pa_{x_i}\pa_{x_j} f + v \cdot \nabla_x \pa_{x_i} \pa_{x_j} f
&= - (\nabla_v \cdot \pa_{x_i} \pa_{x_j}F) f - \pa_{x_i} \pa_{x_j} F \cdot \nabla_v f - (\nabla_v \cdot \pa_{x_j}F) \pa_{x_i} f - \pa_{x_j} F \cdot \nabla_v \pa_{x_i} f\cr
&\quad - (\nabla_v \cdot \pa_{x_i} F) \pa_{x_j} f - \pa_{x_i} F \cdot \nabla_v \pa_{x_j} f - (\nabla_v \cdot F) \pa_{x_i} \pa_{x_j} f - F \cdot \nabla_v \pa_{x_i} \pa_{x_j}f.
\end{align*}
Then we get
\begin{align*}
&\frac{d}{dt} \inttr |\pa_{x_i}\pa_{x_j} f|^{q}\,dxdv
\cr
&\quad 
= q\inttr |\pa_{x_i}\pa_{x_j} f|^{q-2} \pa_{x_i}\pa_{x_j} f \pa_t \pa_{x_i}\pa_{x_j} f f\,dxdv\cr
&\quad =- q\inttr |\pa_{x_i}\pa_{x_j} f|^{q-2} \pa_{x_i}\pa_{x_j} f  \lt(  (\nabla_v \cdot \pa_{x_i} \pa_{x_j}F) f + \pa_{x_i} \pa_{x_j} F \cdot \nabla_v f\rt) dxdv\cr
&\qquad - q\inttr |\pa_{x_i}\pa_{x_j} f|^{q-2} \pa_{x_i}\pa_{x_j} f  \lt(  (\nabla_v \cdot \pa_{x_j}F) \pa_{x_i} f + \pa_{x_j} F \cdot \nabla_v \pa_{x_i} f\rt) dxdv\cr
&\qquad - q\inttr |\pa_{x_i}\pa_{x_j} f|^{q-2} \pa_{x_i}\pa_{x_j} f  \lt( (\nabla_v \cdot \pa_{x_i} F) \pa_{x_j} f + \pa_{x_i} F \cdot \nabla_v \pa_{x_j} f\rt) dxdv\cr
&\qquad - q\inttr |\pa_{x_i}\pa_{x_j} f|^{q-2} \pa_{x_i}\pa_{x_j} f  \lt( (\nabla_v \cdot F) \pa_{x_i} \pa_{x_j} f + F \cdot \nabla_v \pa_{x_i} \pa_{x_j}f\rt) dxdv
\cr
&\quad 
=: \sum_{i=1}^8 I_i,
\end{align*}
where $I_i,i=1,\dots,8$ can be estimated as follows:
\begin{align*}
& I_1 %&\leq q \|\nabla_v \cdot \pa_{x_i}\pa_{x_j} F\|_{L^\infty}\|\pa_{x_i}\pa_{x_j} f\|_{L^{q}}^{q-1}\|f\|_{L^{q}} 
\leq qd\|\pa_{x_i}\pa_{x_j}\psi\|_{L^\infty}  \|\pa_{x_i}\pa_{x_j} f\|_{L^{q}}^{q-1}\|f\|_{L^{q}},
%\cr
\qquad 
I_2   \leq q \inttr  |\pa_{x_i}\pa_{x_j} f|^{q-1} |\pa_{x_i}\pa_{x_j} F| |\nabla_v f|\,dxdv,\cr
& I_3 %&\leq q \|\nabla_v \cdot \pa_{x_j} F\|_{L^\infty}\|\pa_{x_i}\pa_{x_j} f\|_{L^{q}}^{q-1}\|\pa_{x_i}f\|_{L^{q}} 
\leq q d\|\pa_{x_j} \psi\|_{L^\infty}\|\pa_{x_i}\pa_{x_j} f\|_{L^{q}}^{q-1}\|\pa_{x_i}f\|_{L^{q}},\qquad
%\cr
I_4 \leq q\|\pa_{x_j} F\|_{L^\infty}\|\pa_{x_i}\pa_{x_j} f\|_{L^{q}}^{q-1}\|\nabla_v \pa_{x_i}f\|_{L^{q}},
\cr
&I_5\leq q d\|\pa_{x_i} \psi\|_{L^\infty}\|\pa_{x_i}\pa_{x_j} f\|_{L^{q}}^{q-1}\|\pa_{x_j}f\|_{L^{q}},\qquad
%\cr
I_6 \leq q\|\pa_{x_i} F\|_{L^\infty}\|\pa_{x_i}\pa_{x_j} f\|_{L^{q}}^{q-1}\|\nabla_v \pa_{x_j}f\|_{L^{q}},\cr
&I_7 = qd \inttr (1 + \psi \star \rho_f) |\pa_{x_i}\pa_{x_j} f|^{q}\,dxdv,\qquad
%\cr
I_8 = -d \inttr (1 + \psi \star \rho_f) |\pa_{x_i}\pa_{x_j} f|^{q}\,dxdv.
\end{align*}
Combining all of the above estimates, we have
\begin{align*}
&\frac{d}{dt} \inttr |\pa_{x_i}\pa_{x_j} f|^{q}\,dxdv 
%\cr
%&\quad 
\leq qd  \|\pa_{x_i}\pa_{x_j} f\|_{L^{q}}^{q-1}\lt( \|\pa_{x_i}\pa_{x_j}\psi\|_{L^\infty}\|f\|_{L^{q}} + \|\pa_{x_j} \psi\|_{L^\infty} \|\pa_{x_i}f\|_{L^{q}} \rt)\cr
&\qquad + q\|\pa_{x_i}\pa_{x_j} f\|_{L^{q}}^{q-1}\lt( \|\pa_{x_j} F\|_{L^\infty}\|\nabla_v \pa_{x_i}f\|_{L^{q}}+ d\|\pa_{x_i} \psi\|_{L^\infty} \|\pa_{x_j}f\|_{L^{q}} +\|\pa_{x_i} F\|_{L^\infty}\|\nabla_v \pa_{x_j}f\|_{L^{q}} \rt)\cr
&\qquad + q \inttr  |\pa_{x_i}\pa_{x_j} f|^{q-1} |\pa_{x_i}\pa_{x_j} F| |\nabla_v f|\,dxdv + (q-1)d\inttr (1 + \psi \star \rho_f) |\pa_{x_i}\pa_{x_j} f|^{q}\,dxdv.
\end{align*}
On the other hand, we estimate the second order derivative of the force field $F$ as
\begin{align*}
|\pa_{x_i}\pa_{x_j} F| &\leq |(\pa_{x_i}\pa_{x_j} \psi) \star \rho_f u_f| + |v \pa_{x_i}\pa_{x_j}\psi \star \rho_f| + |\pa_{x_i}\pa_{x_j}u|\cr
&\leq \|\pa_{x_i}\pa_{x_j} \psi\|_{L^\infty}\|\rho_f u_f\|_{L^1} + R_v(t) \|\pa_{x_i}\pa_{x_j} \psi\|_{L^\infty}\|\rho_f\|_{L^1}+ |\pa_{x_i}\pa_{x_j}u|\cr
&\leq C + |\pa_{x_i}\pa_{x_j}u|,
\end{align*}
and this implies
\begin{align*}
&\inttr  |\pa_{x_i}\pa_{x_j} f|^{q-1} |\pa_{x_i}\pa_{x_j}u| |\nabla_v f|\,dxdv\cr
&\quad \leq Cq \|\pa_{x_i}\pa_{x_j} f\|_{L^{q}}^{q-1}\|\nabla_v f\|_{L^{q}} +  q  \inttr  |\pa_{x_i}\pa_{x_j} f|^{q-1} |\pa_{x_i}\pa_{x_j}u| |\nabla_v f|\,dxdv\cr
&\quad \leq \|\nabla_v f\|_{L^\infty} \inttr  |\pa_{x_i}\pa_{x_j} f|^{q-1} |\pa_{x_i}\pa_{x_j}u|\,dxdv\cr
&\quad \leq C_q\|\nabla_v f\|_{W^{1,q}}  \|\pa_{x_i}\pa_{x_j} f\|_{L^{q}}^{q-1} \|\pa_{x_i}\pa_{x_j}u\|_{L^{q}}
\end{align*}
for $q > d$, where we used the Morrey's inequality. Thus we obtain
\begin{align*}
\frac{d}{dt} \|\pa_{x_i}\pa_{x_j} f\|_{L^{q}}
& \leq Cd  \lt( \|f\|_{L^{q}} +  \|\pa_{x_i}f\|_{L^{q}} \rt) + C( 1  + \|\nabla_x u\|_{L^\infty})(\|\nabla_v \pa_{x_i}f\|_{L^{q}} + \|\nabla_v \pa_{x_j}f\|_{L^{q}})\cr
&\quad + Cd \|\pa_{x_j}f\|_{L^{q}} + C \|\nabla_v f\|_{L^{q}} + C_q\|\nabla_v f\|_{W^{1,q}} \|\pa_{x_i}\pa_{x_j}u\|_{L^{q}}\cr
\end{align*}
due to \eqref{est_F}. This asserts
\begin{align*}
\frac{d}{dt}\|\nabla_x^2 f\|_{L^{q}} &\leq Cd\lt( \|f\|_{L^{q}} +  \|\nabla_x f\|_{L^{q}} \rt)+ C( 1  + \|\nabla_x u\|_{L^\infty})\|\nabla_v \nabla_x f\|_{L^{q}} \cr
&\quad + C\|\nabla_v f\|_{L^{q}}+ C_q\|\nabla_v f\|_{W^{1,q}} \|\nabla_x^2 u\|_{L^{q}}.
\end{align*}

($\diamond$ Estimate of $\|\nabla_x\nabla_v f\|_{L^{q}}$): For $i,j=1,\dots,d$, we find
\begin{align*}
\pa_t \pa_{x_i} \pa_{v_j}f + v \cdot \nabla_x \pa_{x_i} \pa_{v_j}f
&= -\pa_{x_i} \pa_{x_j} f - (\pa_{x_i} \nabla_v \cdot F) \pa_{v_j} f - (\nabla_v \cdot F)\pa_{x_i} \pa_{v_j} f - \pa_{x_i} \pa_{v_j} F \cdot \nabla_v f\cr
&\quad - \pa_{v_j} F \cdot \nabla_v \pa_{x_i} f - \pa_{x_i} F \cdot \nabla_v \pa_{v_j} f - F \cdot \nabla_v \pa_{x_i} \pa_{v_j} f.
\end{align*}
Then we have
\begin{align*}
&\frac{d}{dt}\inttr |\pa_{x_i} \pa_{v_j} f|^{q}\,dxdv 
%\cr
%&\quad 
= q \inttr |\pa_{x_i} \pa_{v_j} f|^{q-2} \pa_{x_i} \pa_{v_j} f \pa_t \pa_{x_i} \pa_{v_j} f\,dxdv\cr
& \quad = -q \inttr |\pa_{x_i} \pa_{v_j} f|^{q-2} \pa_{x_i} \pa_{v_j} f \lt( \pa_{x_i} \pa_{x_j} f + (\pa_{x_i} \nabla_v \cdot F) \pa_{v_j} f + (\nabla_v \cdot F)\pa_{x_i} \pa_{v_j} f\rt) dxdv\cr
& \qquad -q \inttr |\pa_{x_i} \pa_{v_j} f|^{q-2} \pa_{x_i} \pa_{v_j} f \lt(\pa_{x_i} \pa_{v_j} F \cdot \nabla_v f + \pa_{v_j} F \cdot \nabla_v \pa_{x_i} f + \pa_{x_i} F \cdot \nabla_v \pa_{v_j} f\rt) dxdv\cr
& \qquad -q \inttr |\pa_{x_i} \pa_{v_j} f|^{q-2} \pa_{x_i} \pa_{v_j} f \lt( F \cdot \nabla_v \pa_{x_i} \pa_{v_j} f\rt) dxdv
\cr
&\quad 
=: \sum_{i=1}^7 J_i.
\end{align*}
We estimate
\begin{align*}
J_1 &\leq q \|\pa_{x_i} \pa_{v_j} f\|_{L^{q}}^{q-1} \|\pa_{x_i} \pa_{x_j} f\|_{L^{q}},\qquad %\cr
J_2 %\leq q\|\pa_{x_i} \nabla_v \cdot F\|_{L^\infty}\|\pa_{x_i} \pa_{v_j} f\|_{L^{q}}^{q-1} \|\pa_{v_j} f\|_{L^{q}} \cr
\leq qd\|\pa_{x_i} \psi\|_{L^\infty}\|\pa_{x_i} \pa_{v_j} f\|_{L^{q}}^{q-1} \|\pa_{v_j} f\|_{L^{q}},\cr
J_3 &= qd \inttr (1 + \psi \star \rho_f)|\pa_{x_i}\pa_{v_j}f|^{q}\,dxdv, \qquad %\cr
J_4 \leq q\|\pa_{x_i}\psi\|_{L^\infty}\|\pa_{x_i} \pa_{v_j} f\|_{L^{q}}^{q-1} \|\pa_{x_j} f\|_{L^{q}}, \cr
J_5 &= q\inttr (1 + \psi \star \rho_f)|\pa_{x_i}\pa_{v_j}f|^{q}\,dxdv,  \qquad %\cr
J_6 \leq q\|\pa_{x_i} F\|_{L^\infty}\|\pa_{x_i} \pa_{v_j} f\|_{L^{q}}^{q-1} \|\nabla_v \pa_{v_j}f\|_{L^{q}},\cr
J_7 &=-d\inttr (1 + \psi \star \rho_f)|\pa_{x_i}\pa_{v_j}f|^{q}\,dxdv.
\end{align*}
Thus we obtain
\begin{align*}
\frac{d}{dt}\|\pa_{x_i} \pa_{v_j} f\|_{L^{q}}^{q}
&\leq q \|\pa_{x_i} \pa_{v_j} f\|_{L^{q}}^{q-1} \|\pa_{x_i} \pa_{x_j} f\|_{L^{q}} + qd\|\pa_{x_i} \psi\|_{L^\infty}\|\pa_{x_i} \pa_{v_j} f\|_{L^{q}}^{q-1} \|\pa_{v_j} f\|_{L^{q}}\cr
&\quad + q\|\pa_{x_i}\psi\|_{L^\infty}\|\pa_{x_i} \pa_{v_j} f\|_{L^{q}}^{q-1} \|\pa_{x_j} f\|_{L^{q}} + q\|\pa_{x_i} F\|_{L^\infty}\|\pa_{x_i} \pa_{v_j} f\|_{L^{q}}^{q-1} \|\nabla_v \pa_{v_j}f\|_{L^{q}}\cr
&\quad + (q(d+1)-d)\inttr (1 + \psi \star \rho_f)|\pa_{x_i}\pa_{v_j}f|^{q}\,dxdv\cr
&\leq q \|\pa_{x_i} \pa_{v_j} f\|_{L^{q}}^{q-1} \|\pa_{x_i} \pa_{x_j} f\|_{L^{q}} + qd\|\pa_{x_i} \psi\|_{L^\infty}\|\pa_{x_i} \pa_{v_j} f\|_{L^{q}}^{q-1} \|\pa_{v_j} f\|_{L^{q}}\cr
&\quad + q\|\pa_{x_i}\psi\|_{L^\infty}\|\pa_{x_i} \pa_{v_j} f\|_{L^{q}}^{q-1} \|\pa_{x_j} f\|_{L^{q}} + q\|\pa_{x_i} F\|_{L^\infty}\|\pa_{x_i} \pa_{v_j} f\|_{L^{q}}^{q-1} \|\nabla_v \pa_{v_j}f\|_{L^{q}}\cr
&\quad + (q(d+1)-d)(1 + \|\psi\|_{L^\infty})\|\pa_{x_i}\pa_{v_j}f\|^{q}_{L^{q}},
\end{align*}
and this implies
\begin{align*}
\frac{d}{dt}\|\pa_{x_i} \pa_{v_j} f\|_{L^{q}}
&\leq \|\pa_{x_i} \pa_{x_j} f\|_{L^{q}} + d\|\pa_{x_i} \psi\|_{L^\infty}\|\pa_{v_j} f\|_{L^{q}} + \|\pa_{x_i}\psi\|_{L^\infty} \|\pa_{x_j} f\|_{L^{q}}\cr
&\quad + \|\pa_{x_i} F\|_{L^\infty} \|\nabla_v \pa_{v_j}f\|_{L^{q}} + \lt(\lt(d+1-\frac{d}{q}\rt)(1 + \|\psi\|_{L^\infty})\rt)\|\pa_{x_i}\pa_{v_j}f\|_{L^{q}}.
\end{align*}
This together with \eqref{est_F} yields
\begin{align*}
\frac{d}{dt}\|\nabla_x \nabla_v f\|_{L^{q}} &\leq \|\nabla_x^2 f\|_{L^{q}} + Cd\|\nabla_v f\|_{L^{q}} + C \|\nabla_x f\|_{L^{q}}\cr
&\quad + C(1 + \|\nabla_x u\|_{L^\infty}^2) \|\nabla_v^2 f\|_{L^{q}} + C\lt(d+1-\frac{d}{q}\rt)\|\nabla_x \nabla_v f\|_{L^{q}},
\end{align*}
where $C > 0$ depends on $R_v(0)$ and $\psi$, but independent of $t$ and $q$. \newline

($\diamond$ Estimate of $\|\nabla_v^2 f\|_{L^{q}}$): Note that $\pa_{v_i}\pa_{v_j}f $ satisfies
\begin{align*}
&\pa_t \pa_{v_i}\pa_{v_j}f + v \cdot \nabla_x \pa_{v_i}\pa_{v_j}f 
\cr
&\quad 
= - \pa_{v_i}\pa_{x_j}f - \pa_{x_i}\pa_{v_j}f - (\nabla_v \cdot F) \pa_{v_i}\pa_{v_j}f - \pa_{v_j} F \cdot \nabla_v \pa_{v_i} f - \pa_{v_i} F \cdot \nabla_v \pa_{v_j} f - F \cdot \nabla_v \pa_{v_i}\pa_{v_j} f.
\end{align*}
Thus we get
\begin{align*}
&\frac{d}{dt}\inttr |\pa_{v_i}\pa_{v_j} f|^{q}\,dxdv 
\cr
&\quad 
= q \inttr |\pa_{v_i}\pa_{v_j} f|^{q-2} \pa_{v_i}\pa_{v_j} f \pa_t \pa_{v_i}\pa_{v_j} f\,dxdv\cr
&\quad = -q \inttr |\pa_{v_i}\pa_{v_j} f|^{q-2} \pa_{v_i}\pa_{v_j} f \lt(\pa_{v_i}\pa_{x_j}f + \pa_{x_i}\pa_{v_j}f + (\nabla_v \cdot F) \pa_{v_i}\pa_{v_j}f \rt) dxdv\cr
&\qquad  -q \inttr |\pa_{v_i}\pa_{v_j} f|^{q-2} \pa_{v_i}\pa_{v_j} f \lt(\pa_{v_j} F \cdot \nabla_v \pa_{v_i} f+\pa_{v_i} F \cdot \nabla_v \pa_{v_j} f+ F \cdot \nabla_v \pa_{v_i}\pa_{v_j} f \rt) dxdv
\cr
&\quad 
=: \sum_{i=1}^6 K_i,
\end{align*}
where $K_i,i=1,\dots,6$ can be estimated as follows.
\begin{align*}
K_1 &\leq q\|\pa_{v_i}\pa_{v_j} f\|_{L^{q}}^{q-1}\|\pa_{v_i}\pa_{x_j}f\|_{L^{q}}, \qquad %\cr
K_2 \leq q\|\pa_{v_i}\pa_{v_j} f\|_{L^{q}}^{q-1}\|\pa_{v_j}\pa_{x_i}f\|_{L^{q}},\cr
K_3 &= qd\inttr (1 + \psi \star \rho_f)|\pa_{v_i}\pa_{v_j} f|^{q}\,dxdv, \qquad %\cr
K_4 = q\inttr (1 + \psi \star \rho_f)|\pa_{v_i}\pa_{v_j} f|^{q}\,dxdv,\cr
K_5 &= q\inttr (1 + \psi \star \rho_f)|\pa_{v_i}\pa_{v_j} f|^{q}\,dxdv,\qquad %\cr
K_6 = -d\inttr (1 + \psi \star \rho_f)|\pa_{v_i}\pa_{v_j} f|^{q}\,dxdv.
\end{align*}
This deduces
\begin{align*}
\frac{d}{dt}\|\pa_{v_i}\pa_{v_j} f\|_{L^{q}}^{q} &\leq q\|\pa_{v_i}\pa_{v_j} f\|_{L^{q}}^{q-1}\|\pa_{v_i}\pa_{x_j}f\|_{L^{q}} + q\|\pa_{v_i}\pa_{v_j} f\|_{L^{q}}^{q-1}\|\pa_{v_j}\pa_{x_i}f\|_{L^{q}}\cr
&\quad + (q(d+2)-d)\inttr (1 + \psi \star \rho_f)|\pa_{v_i}\pa_{v_j} f|^{q}\,dxdv\cr
&\leq q\|\pa_{v_i}\pa_{v_j} f\|_{L^{q}}^{q-1}\|\pa_{v_i}\pa_{x_j}f\|_{L^{q}} + q\|\pa_{v_i}\pa_{v_j} f\|_{L^{q}}^{q-1}\|\pa_{v_j}\pa_{x_i}f\|_{L^{q}}\cr
&\quad + (q(d+2)-d)(1 + \|\psi\|_{L^\infty} )\|\pa_{v_i}\pa_{v_j} f\|_{L^{q}}^{q},
\end{align*}
and subsequently
\[
\frac{d}{dt}\|\pa_{v_i}\pa_{v_j} f\|_{L^{q}} \leq \|\pa_{v_i}\pa_{x_j}f\|_{L^{q}} + \|\pa_{v_j}\pa_{x_i}f\|_{L^{q}} + \lt(d+2-\frac{d}{q}\rt)(1 + \|\psi\|_{L^\infty} )\|\pa_{v_i}\pa_{v_j} f\|_{L^{q}}.
\]
Hence we have
\[
\frac{d}{dt}\|\nabla_v^2 f\|_{L^{q}} \leq\|\nabla_v^2 f\|_{L^{q}} + \|\nabla_x \nabla_v f\|_{L^{q}} + C\lt(d+2-\frac{d}{q}\rt)\|\nabla_v^2  f\|_{L^{q}}.
\]
Finally, we combine all of the above estimates to have
\begin{align*}
&\frac{d}{dt}\lt(\|\nabla_x^2 f\|_{L^{q}} + \|\nabla_x \nabla_v  f\|_{L^{q}} + \|\nabla_v^2 f\|_{L^{q}} \rt)\cr
&\quad \leq Cd\lt( \|f\|_{L^{q}} +  \|\nabla_x f\|_{L^{q}} + \|\nabla_v f\|_{L^{q}} \rt)+ C( 1  + \|\nabla_x u\|_{L^\infty}^2)(\|\nabla_v \nabla_x f\|_{L^{q}} + \|\nabla_v^2 f\|_{L^{q}})\cr
&\qquad + C\|\nabla_v f\|_{W^{1,q}} \|\nabla_x^2 u\|_{L^{q}} +C\lt(d+1-\frac{d}{q}\rt)(\|\nabla_x \nabla_v f\|_{L^{q}} +\|\nabla_v^2  f\|_{L^{q}}) + \|\nabla_x^2 f\|_{L^{q}} \cr
&\quad \leq Cd\|f\|_{W^{1,q}}+ Cd( 1  + \|\nabla_x u\|_{L^\infty}^2)( \|\nabla_x^2 f\|_{L^{q}}  + \|\nabla_v \nabla_x f\|_{L^{q}} + \|\nabla_v^2 f\|_{L^{q}})
%\cr
%&\qquad 
+ C\|\nabla_v f\|_{W^{1,q}} \|\nabla_x^2 u\|_{L^{q}},
\end{align*}
where $C > 0$ is independent of $t$. We then combine this with Lemma \ref{lem_f} to obtain
\[
\frac{d}{dt}\|f\|_{W^{2,q}} \leq Cd( 1  + \|\nabla_x u\|_{L^\infty}^2 + \|\nabla_x^2 u\|_{L^{q}})\|f\|_{W^{2,q}},
\]
where $C>0$ is independent of $t$. Applying Gr\"onwall's lemma to the above concludes the desired result.

%%%%%%%%%%%%%%%%%%%%%%%%%
%%
%%
%%\begin{thebibliography}{00}
%%
%%%%%%%%%%%%%%%%%%%%%%%%%

\end{document}